%% file: linear_independences.tex
\documentclass[12pt]{amsart}
\usepackage[latin1]{inputenc}
\usepackage{amsbsy}
\usepackage{amscd} 
\usepackage{amsfonts}
\usepackage{amsgen} 
\usepackage{amsmath}
\usepackage{amsopn} 
\usepackage{amssymb}
\usepackage{amstext}
\usepackage{amsthm} 
\usepackage{amsxtra}
\usepackage{array}
\usepackage[all]{xy}
\usepackage{epsfig}
\usepackage{extarrows}
\usepackage{dsfont}
\usepackage{float}
\usepackage{hhline}
\usepackage{longtable}
\usepackage{mathtools}
\usepackage{mathrsfs} 
\usepackage{paralist}
\usepackage{rotating}
\usepackage{tikz}

\include{PartitionCommands}

\theoremstyle{plain} 
\newtheorem{thm}{Theorem}[section]
\newtheorem*{thm*}{Theorem}
\newtheorem*{thm*conj*}{Theorem/Conjecture}
\newtheorem{prop}[thm]{Proposition}
\newtheorem{lem}[thm]{Lemma}
\newtheorem*{lem*}{Lemma}
\newtheorem{cor}[thm]{Corollary}
\newtheorem*{cor*}{Corollary}
\theoremstyle{definition}
\newtheorem{defn}[thm]{Definition}

\newtheorem*{defn*}{Definition}
\newtheorem{rem}[thm]{Remark}
\newtheorem{ex}[thm]{Example}
\newtheorem{question}[thm]{Question}

\newtheorem{notation}[thm]{Notation}
\newtheorem*{notation*}{Notation}
\newtheorem{obs}[thm]{Observation}

\numberwithin{equation}{section}

\renewcommand{\theta}{\vartheta}
\renewcommand{\epsilon}{\varepsilon}
\renewcommand{\subset}{\subseteq}

\newcommand{\N}{\mathbb N}

\newcommand{\C}{\mathbb C}

\newcommand{\otop}{\;\raisebox{0.02cm}{\scalebox{0.72}{$\bigcirc\textnormal{\hspace{-10.8pt}}\raisebox{-0.07cm}{$\top$}$}}\;}

\DeclareMathOperator{\col}{Col}

\DeclareMathOperator{\Hom}{Hom}
\DeclareMathOperator{\id}{id}

\DeclareMathOperator{\rl}{rl}

\DeclareMathOperator{\Span}{span}

\usepackage{hyperref}

\begin{document}
\title{Linear independences of maps associated to partitions}
\author{Stefan Jung}
\address{Saarland University, Fachbereich Mathematik, Postfach 151150,
66041 Saarbr\"ucken, Germany}
\email{jung@math.uni-sb.de}
\date{\today}
\subjclass[2010]{}
\thanks{The author was funded by the ERC Advanced Grant NCDFP, held by Roland Speicher. This article is part of the author's PhD thesis.}

\begin{abstract}
Given a suitable collection of partitions of sets, there exists a connection to easy quantum groups via intertwiner maps.
A sufficient condition for this correspondence to be one-to-one are particular linear independences on the level of those maps.
In the case of non-crossing partitions, a proof of this linear independence can be traced back to a matrix determinant formula, developed by W. Tutte.
We present a revised and adapted version of Tutte's work and the link to the problem above, believing that this self-contained article will assist others in the field of easy quantum groups.
In particular, we fixed some errors in the original work and adapted notations, definitions, statements and proofs.
\end{abstract}

\maketitle

\section{Introduction}
In \cite{woronowiczpseudogroups} S. L. Woronowicz defined compact matrix quantum groups (CMQGs), roughly speaking as unital \(C^*\)-algebras \(A\) generated by the entries \(u_{ij}\) of an \(N\times N\) matrix \(u\) such that there exists a suitable comultiplication on \(A\).
In \cite{woronowicztannakakrein} the same author introduced \emph{Tannaka-Krein duality for compact matrix quantum groups}, which can be sketched as follows:
Given for non-negative integers \(k\) and \(l\) two words \(\omega\in\{1,*\}^k\) and \(\omega'\in\{1,*\}^l\) over the alphabet \(\{1,*\}\), one can consider so-called \emph{intertwiner spaces}
\[\Hom(\omega,\omega'):=\Big\{\,T:(\C^N)^{\otimes k}\rightarrow (\C^N)^{\otimes l}\;|\;T\textnormal{ linear},\; T\big(u^{\otop\omega}\big)=\big(u^{\otop\omega'}\big)T\Big\}.\]
Tannaka-Krein duality says that the correspondence between CMQGs and such collections of operator spaces is one-to-one.
In particular, one can write down the axioms of such a collection of intertwiner spaces in the form of \emph{concrete monoidal \(W^*\)-categories \(R\)}.
\newline
In \cite{banicaspeicherliberation} T. Banica and R. Speicher used this duality to define \emph{easy quantum groups}:
Given a so-called \emph{category of  partitions} \(\mathcal{C}\) and a natural number \(N\in\N\), one can associate to every partition \(p\in\mathcal{C}\) a linear map \(T_p\).
Defining
\[\Hom(k,l):=\Span(\,\{\,T_p\;|\;p\in\mathcal{C}(k,l)\}\,),\]
where \(k\) and \(l\) are non-negative integers, one obtains a concrete monoidal \(W^*\)-category \(R_N(\mathcal{C})\) and is able to apply Tannaka-Krein duality to this situation.
The result is a unique CMQG \(G_N(\mathcal{C})\), called easy quantum group.
See the preliminaries for precise definitions of the notations.
\vspace{11pt}
\newline
Considering the two-step construction of easy quantum groups as just displayed,
\begin{equation}\label{eqn:from_partitions_to_easy_quantum_groups}
\mathcal{C}\overset{\Psi}{\mapsto} R_N(\mathcal{C})\overset{\Phi}{\mapsto} G_N(\mathcal{C}),
\end{equation}
the second functor, \(\Phi\), that associates a CMQG \(G_N(\mathcal{C})\) to a given concrete  monoidal \(W^*\)-category \(R_N(\mathcal{C})\), is injective by Woronowicz's Tannaka-Krein duality.
The whole construction in Equation \ref{eqn:from_partitions_to_easy_quantum_groups} is surjective by definition of easy quantum groups, but the question is, if we also have injectivity for the first mapping \(\Psi\).
The main result of this work concentrates on the category of non-crossing partitions \(\mathcal{NC}\) and is given by Theorem \ref{thm:main_result:linear_independence_of_T_ps}:
Consider a category of non-crossing partitions \(\mathcal{C}\subseteq \mathcal{NC}\).
Then for \(N\ge 4\) and any fixed \(n\in\N_0\) the collection of maps
\(\displaystyle\big(T_p\big)_{p\in\mathcal{C}(0,n)}\)
is linearly independent.
\newline
From this it can easily be deduced that the functor \(\mathcal{C}\overset{\Psi}{\mapsto} R_N(\mathcal{C})\) is injective for the considered cases and so is the whole construction in Equation \ref{eqn:from_partitions_to_easy_quantum_groups}, see Corollary \ref{cor:sufficient_condition_for_different_easy_quantum_groups}.
In order to prove the above, we recapture W. Tutte's paper \cite{tutte}, see below.
\vspace{11pt}
\newline
The organization of this work is as follows:
In Section \ref{sec:reformulating_the_injectivity_of_Psi_as_a_problem_of_linear_independence} we reformulate the question of injectivity of the functor \(\Psi\) as a linear independence problem on the level of maps \(T_p\) associated to partitions \(p\).
The subsequent Section \ref{sec:linear_(in)dependencies_in_the_general_situation} deals with this problem in the general situation (of arbitrary partitions).
Proposition \ref{prop:linear_independence_if_blocks_do_not_exceed_N} and Corollary \ref{cor:different_categories_imply_different_easy_QGs_if_N_exceeds_number_of_points} show that for two different given categories of partitions \(\mathcal{C}_1\) and \(\mathcal{C}_2\) we find at least some \(N\in\N\) such that the respective easy quantum groups \(G_N(\mathcal{C}_1)\) and \(G_N(\mathcal{C}_2)\) differ.
With Proposition \ref{prop:partitions_with_at_most_N_blocks_generate_everything} we also state some kind of converse result:
Considering the category \(\mathcal{P}\) of all partitions, a linear basis for the space \(\Hom(0,n)\) is already given by the collection of those \(T_p\) where \(p\in\mathcal{C}(0,n)\) has at most \(N\) blocks.
\vspace{11pt}
\newline
The main part of this work is Section \ref{sec:linear_independence_in_the_free_case}.
It focusses on the non-crossing case, i.e. free easy quantum groups. It is known that different categories of non-crossing partitions produce different easy quantum groups, at least if \(N\ge 4\), i.e. the associated matrix of generators \(u\) has at least 4 rows and columns.
A common reference from which this can be deduced is W. Tutte's work \cite{tutte} on the matrix of chromatic joints.
It is a purely combinatorial work on the level of partitions on one line.
One of its advantages is, that it does not use any deep results from other theories.
Its proofs are straightforward once the relevant objects have been properly identified.
As the consequence of Tutte's work in the context of easy quantum groups is well-known, it is worth to motivate this part of the article:
\newline
The aim of Section \ref{sec:linear_independence_in_the_free_case} is to present a self-contained proof of the linear independence of maps \(T_p\) in the non-crossing case.
Some of Tutte's arguments turned out to be wrong and in this virtue we adapted the original work.
Other parts are changed to fit with common notations in the context of easy quantum groups or they are extended to justify results where proofs have been omitted in Tutte's article.
In addition, we linked definitions and arguments with the graphical presentations of partitions (see the preliminaries).
In this sense the present work contributes to the theory of free easy quantum groups.
See also Section \ref{subsec:comments_on_tuttes_work}, the purpose of which is to make clear for which ideas and results credit is due to Tutte and his work.
\vspace{11pt}
\newline
Throughout this work we denote for any \(n\in\N\) with \([n]\) the set \(\{1,\ldots,n\}\) and define \([0]=\emptyset\).
Given two Hilbert spaces \(H_1\) and \(H_2\), let \(B(H_1,H_2)\) denote the bounded linear operators from \(H_1\) to \(H_2\).
If both \(H_1\) and \(H_2\) have the same finite vector space  dimension \(N\), we identify \(B(H_1,H_2)\) with the \(N\times N\)-matrices \(M_N(\C)\) via some orthonormal bases of \(H_1\) and \(H_2\).
\section{Preliminaries}
In this section we define compact matrix quantum groups and give an overview on partitions of sets. We also describe how categories of such partitions define easy quantum groups and how Woronowicz's Tannaka-Krein duality makes parts of this construction injective.
Details about compact matrix quantum groups can be found in \cite{woronowiczpseudogroups} and Tannaka-Krein duality has been proved in \cite{woronowicztannakakrein}. More informations on  partitions and their link to easy quantum groups can be found in \cite{banicaspeicherliberation}.
\begin{rem}
In this article we settled for the context of (uni-coloured) partitions and (orthogonal) easy quantum groups as described in \cite{banicaspeicherliberation}.
We remark that P. Tarrago and M. Weber generalized in \cite{tarragoweberclassificationpartitions} and \cite{tarragoweberclassificationunitaryQGs} this notions to two-coloured partitions and unitary easy quantum groups.
However, the questions of linear independence stay the same.
\end{rem}
The following definition is an adaptation of Woronowicz's compact matrix quantum groups in \cite{woronowiczpseudogroups} to the real case.
\begin{defn}
Let \(A\) be a unital \(C^*\)-algebra, \(N\in \N\) and \(u_G=(u_{ij})\in M_N(A)\) be an \(N\!\times\! N\)-matrix over \(A\). Assume that the following holds:
\begin{itemize}
\item[(i)] The entries of \(u\) generate \(A\) as a \(C^*\)-algebra.
\item [(ii)] The matrix \(u\) is invertible and its entries are self-adjoint.
\item[(iii)] There is a unital \(^*\)-homomorphism \(\Delta:A\rightarrow A\otimes A\) (called comultiplication on \(A\)) that fulfils
\[\Delta(u_{ij})=\sum_{k=1}^{N}u_{ik}\otimes u_{kj}\quad\quad \forall 1\le i,j\le N.\]
\end{itemize}
Then we denote \(A\) also by \(C(G)\) and call it the \emph{non-commutative functions over an orthogonal compact matrix quantum group \(G\) of size \(N\)}.
\end{defn}
\begin{defn}\label{defn:partitions}
Let \(k,l\in\N_0\). 
A \emph{partition on \(k\) upper and \(l\) lower points} is a collection of non-empty disjoint subsets of \([k]\dot{\cup}[l]\) such that their union is \([k]\dot{\cup}[l]\).
The subsets of \([k]\dot{\cup}[l]\) describing the partition are called \emph{blocks} and its number is denoted by \(b(p)\).
\end {defn}
In the following we will only speak of partitions when dealing with the objects above.
We illustrate partitions by drawing the elements of  \([k]\) and \([l]\) as actual points into two horizontal lines on top of each other. Points that lie in the same block are connected via a grid of lines. Here are three examples \(p\), \(q\) and \(r\) of such partitions:
\vspace{11pt}
\begin{equation}\label{eqn:examples_of_partitions}
p=
\setlength{\unitlength}{0.5cm}
\begin{picture}(6,2.4)
\put(3.2,0){\circle{0.6}}
\put(2.8,-0.33) {\linethickness{0.3cm}\color{white}\line(1,0){0.8}}
\put(0,-1.5) {$\circ$}
\put(1,-1.5) {$\circ$}
\put(2,-1.5) {$\circ$}
\put(3,-1.5) {$\circ$}
\put(4,-1.5) {$\circ$}
\put(0,1.5) {$\circ$}
\put(1,1.5) {$\circ$}
\put(2,1.5) {$\circ$}
\put(0,-2.3) {$1'$}
\put(1,-2.3) {$2'$}
\put(2,-2.3) {$3'$}
\put(3,-2.3) {$4'$}
\put(4,-2.3) {$5'$}
\put(0,2.3) {$1$}
\put(1,2.3) {$2$}
\put(2,2.3) {$3$}
\put(0.2,-1){\line(0,1){1}}
\put(1.2,-1){\line(0,1){1}}
\put(2.2,-1){\line(0,1){0.7}}
\put(3.2,-1){\line(0,1){1.6}}
\put(4.2,-1){\line(0,1){1}}
\put(0.2,0.4){\line(0,1){1}}
\put(1.2,0.6){\line(0,1){0.8}}
\put(2.2,0.6){\line(0,1){0.8}}
\put(1.2,0.6){\line(1,0){2}}
\put(0.2,0){\line(1,0){2.7}}
\put(3.5,0){\line(1,0){0.7}}
\end{picture}
,\quad\quad
q=
\begin{picture}(5,2.4)
\put(0,-1.5) {$\circ$}
\put(1,-1.5) {$\circ$}
\put(2,-1.5) {$\circ$}
\put(0,1.5) {$\circ$}
\put(1,1.5) {$\circ$}
\put(2,1.5) {$\circ$}
\put(3,1.5) {$\circ$}
\put(0,-2.3) {$1'$}
\put(1,-2.3) {$2'$}
\put(2,-2.3) {$3'$}
\put(0,2.3) {$1$}
\put(1,2.3) {$2$}
\put(2,2.3) {$3$}
\put(3,2.3) {$4$}
\put(0.2,-1){\line(0,1){0.8}}
\put(1.2,-1){\line(0,1){1.2}}
\put(2.2,-1){\line(0,1){1.2}}
\put(0.2,0.2){\line(0,1){1.2}}
\put(1.2,0.2){\line(0,1){1.2}}
\put(2.2,0.6){\line(0,1){0.8}}
\put(3.2,0.2){\line(0,1){1.2}}
\put(0.2,0.2){\line(1,0){1}}
\put(2.2,0.2){\line(1,0){1}}
\end{picture}
,\quad\quad
r=
\begin{picture}(3,2.4)
\put(0,-1.5) {$\circ$}
\put(1,-1.5) {$\circ$}
\put(2,-1.5) {$\circ$}
\put(0,-2.3) {$1'$}
\put(1,-2.3) {$2'$}
\put(2,-2.3) {$3'$}
\put(0.2,-1){\line(0,1){1.4}}
\put(1.2,-1){\line(0,1){1.}}
\put(2.2,-1){\line(0,1){1.4}}
\put(0.2,0.4){\line(1,0){2}}
\end{picture}
\vspace{30pt}
\end{equation}
The numbers \(k\) and/or \(l\) are allowed to be zero, see for instance the partition \(r\).
\newline
A block with only one element is called a \emph{singleton}.
We have  primed the lower points in the pictures above just to distinguish the set \([l]\) from \([k]\) in the disjoint union \([k]\dot{\cup}[l]\).
\newline
Note that in \(p\) the points 1', 2' and 5' are not connected to the points 2, 3 and  4'. The corresponding connecting lines need to cross (if we are only allowed to draw them between the upper and lower row of points).
We will call a partition \emph{crossing} whenever in its illustration at least two lines belonging to different blocks cross each other.
Otherwise we call it \emph{non-crossing}.
We will denote the set of all partitions on \(k\) upper and \(l\) lower points by \(\mathcal{P}(k,l)\) and define \(\mathcal{P}:=\bigcup_{k,l\in\N_0}\mathcal{P}(k,l)\). Analogously we define the notions \(\mathcal{NC}(k,l)\) and \(\mathcal{NC}\) if we restrict to non-crossing partitions.
For example  the partition \(q\) from above fulfils
\[q\in \mathcal{NC}(4,3)\subseteq\mathcal{NC}.\]
Keep in mind that empty rows are allowed, so it holds for the partition \(r\) from  above
\[r\in\mathcal{NC}(0,3)\subseteq\mathcal{NC}.\]
\newline
We already introduce one further notion that will appear repeatedly later on:
If we are not interested in the detailed structure of some non-empty (!) parts of the partition, we sometimes replace them by the symbol \(\square\).
Taking for example the partition \(p\) above we could write
\[
p=
\setlength{\unitlength}{0.5cm}
\begin{picture}(5,1.5)
\put(3.2,0){\circle{0.6}}
\put(2.8,-0.33) {\linethickness{0.3cm}\color{white}\line(1,0){0.8}}
\put(0.88,-1.56) {$\square$}
\put(2,-1.5) {$\circ$}
\put(3,-1.5) {$\circ$}
\put(4,-1.5) {$\circ$}
\put(-0.12,1.4) {$\square$}
\put(1,1.5) {$\circ$}
\put(2,1.5) {$\circ$}
\put(1.2,-1){\line(0,1){1}}
\put(2.2,-1){\line(0,1){0.7}}
\put(3.2,-1){\line(0,1){1.6}}
\put(4.2,-1){\line(0,1){1}}
\put(0.2,0.4){\line(0,1){1}}
\put(1.2,0.6){\line(0,1){0.8}}
\put(2.2,0.6){\line(0,1){0.8}}
\put(1.2,0.6){\line(1,0){2}}
\put(1.2,0){\line(1,0){1.7}}
\put(3.5,0){\line(1,0){0.7}}
\end{picture}.\vspace{16pt}
\]
The square in the upper row represents an arbitrary (non-empty) subpartition that is not connected to any of the other points.
In this case it is actually just a singleton.
The square in the lower row represents a subpartition such that there is at least one point that is connected to the rightmost point in the lower row.
In our case this square represents two points connected to each other.
\newline
If we want to allow the subpartition to be empty,  we use dashed lines to draw the corresponding square and connecting lines.
Although there is no point in doing so for the moment, it would be correct to write\vspace{6pt}
\[
p=
\setlength{\unitlength}{0.5cm}
\begin{picture}(6,1.5)
\put(3.2,0){\circle{0.6}}
\put(2.8,-0.33) {\linethickness{0.3cm}\color{white}\line(1,0){0.8}}
\put(0.88,-1.56) {$\square$}
\put(2,-1.5) {$\circ$}
\put(3,-1.5) {$\circ$}
\put(4,-1.5) {$\circ$}
\put(4.88,-1.56) {$\square$}
\put(4.83,-1.26) {\linethickness{0.1cm}\color{white}\line(1,0){0.7}}
\put(5.22,-1.61) {\linethickness{0.1cm}\color{white}\line(0,1){0.7}}
\put(-0.12,1.4) {$\square$}
\put(1,1.5) {$\circ$}
\put(2,1.5) {$\circ$}
\put(1.2,-1){\line(0,1){1}}
\put(2.2,-1){\line(0,1){0.7}}
\put(3.2,-1){\line(0,1){1.6}}
\put(4.2,-1){\line(0,1){1}}
\put(0.2,0.4){\line(0,1){1}}
\put(1.2,0.6){\line(0,1){0.8}}
\put(2.2,0.6){\line(0,1){0.8}}
\multiput(5.2,-1)(0,0.35){3}{\line(0,1){0.2}}
\put(1.2,0.6){\line(1,0){2}}
\put(1.2,0){\line(1,0){1.7}}
\put(3.5,0){\line(1,0){0.7}}
\multiput(4.1,0)(0.4,0){3}{\line(1,0){0.3}}
\end{picture}
=
\setlength{\unitlength}{0.5cm}
\begin{picture}(5,1.5)
\put(3.2,0){\circle{0.6}}
\put(2.8,-0.33) {\linethickness{0.3cm}\color{white}\line(1,0){0.8}}
\put(0,-1.5) {$\circ$}
\put(1,-1.5) {$\circ$}
\put(2,-1.5) {$\circ$}
\put(3,-1.5) {$\circ$}
\put(3.88,-1.56) {$\square$}
\put(-0.12,1.44) {$\square$}
\put(-0.17,1.74) {\linethickness{0.1cm}\color{white}\line(1,0){0.7}}
\put(0.21,1.39) {\linethickness{0.1cm}\color{white}\line(0,1){0.7}}
\put(1,1.5) {$\circ$}
\put(2,1.5) {$\circ$}
\put(0.2,-1){\line(0,1){1}}
\put(1.2,-1){\line(0,1){1}}
\put(2.2,-1){\line(0,1){0.7}}
\put(3.2,-1){\line(0,1){1.6}}
\put(4.2,-1){\line(0,1){1}}
\put(1.2,0.6){\line(0,1){0.8}}
\put(2.2,0.6){\line(0,1){0.8}}
\multiput(0.2,0.5)(0,0.35){3}{\line(0,1){0.2}}
\put(1.2,0.6){\line(1,0){2}}
\put(0.2,0){\line(1,0){2.7}}
\put(3.5,0){\line(1,0){0.7}}
\end{picture}.\vspace{16pt}
\]
In this case the dashed structure in the lower row is indeed empty and the one on the upper row represents again a singleton.
\vspace{11pt}
\newline
We consider now some uni- and bivariate operations on partitions that will allow us to define so called \emph{categories of partitions}.
We use the partitions from Equations \ref{eqn:examples_of_partitions} to give examples.
\newline
The \emph{tensor product} of two partitions \(p\in\mathcal{P}(k,l)\) and \(q\in\mathcal{P}(k',l')\) is defined by horizontal concatenation, i.e. by placing their pictures side by side and considering this a partition in \(\mathcal{P}(k+k',l+l')\):
\[p\otimes q=
\setlength{\unitlength}{0.5cm}
\begin{picture}(5,1)
\put(3.2,0){\circle{0.6}}
\put(2.8,-0.33) {\linethickness{0.3cm}\color{white}\line(1,0){0.8}}
\put(0,-1.5) {$\circ$}
\put(1,-1.5) {$\circ$}
\put(2,-1.5) {$\circ$}
\put(3,-1.5) {$\circ$}
\put(4,-1.5) {$\circ$}
\put(0,1.5) {$\circ$}
\put(1,1.5) {$\circ$}
\put(2,1.5) {$\circ$}
\put(0.2,-1){\line(0,1){1}}
\put(1.2,-1){\line(0,1){1}}
\put(2.2,-1){\line(0,1){0.7}}
\put(3.2,-1){\line(0,1){1.6}}
\put(4.2,-1){\line(0,1){1}}
\put(0.2,0.4){\line(0,1){1}}
\put(1.2,0.6){\line(0,1){0.8}}
\put(2.2,0.6){\line(0,1){0.8}}
\put(1.2,0.6){\line(1,0){2}}
\put(0.2,0){\line(1,0){2.7}}
\put(3.5,0){\line(1,0){0.7}}
\end{picture}
\begin{picture}(4.2,2.4)
\put(0,-1.5) {$\circ$}
\put(1,-1.5) {$\circ$}
\put(2,-1.5) {$\circ$}
\put(0,1.5) {$\circ$}
\put(1,1.5) {$\circ$}
\put(2,1.5) {$\circ$}
\put(3,1.5) {$\circ$}
\put(0.2,-1){\line(0,1){0.8}}
\put(1.2,-1){\line(0,1){1.2}}
\put(2.2,-1){\line(0,1){1.2}}
\put(0.2,0.2){\line(0,1){1.2}}
\put(1.2,0.2){\line(0,1){1.2}}
\put(2.2,0.6){\line(0,1){0.8}}
\put(3.2,0.2){\line(0,1){1.2}}
\put(0.2,0.2){\line(1,0){1}}
\put(2.2,0.2){\line(1,0){1}}
\end{picture}\]
\newline
The \emph{involution} is an operator on \(\mathcal{P}\) that maps a partition  \(r\in\mathcal{P}(k,l)\) to an element \(r^*\in(\mathcal{P}(l,k)\) given by mirroring \(r\) at some horizontal axis:\vspace{11pt}
\[p^*=
\reflectbox{\scalebox{-1}{
\setlength{\unitlength}{0.5cm}
\begin{picture}(5,1)
\put(3.2,0){\circle{0.6}}
\put(2.8,-0.33) {\linethickness{0.3cm}\color{white}\line(1,0){0.8}}
\put(0,-1.5) {$\circ$}
\put(1,-1.5) {$\circ$}
\put(2,-1.5) {$\circ$}
\put(3,-1.5) {$\circ$}
\put(4,-1.5) {$\circ$}
\put(0,1.5) {$\circ$}
\put(1,1.5) {$\circ$}
\put(2,1.5) {$\circ$}
\put(0.2,-1){\line(0,1){1}}
\put(1.2,-1){\line(0,1){1}}
\put(2.2,-1){\line(0,1){0.7}}
\put(3.2,-1){\line(0,1){1.6}}
\put(4.2,-1){\line(0,1){1}}
\put(0.2,0.4){\line(0,1){1}}
\put(1.2,0.6){\line(0,1){0.8}}
\put(2.2,0.6){\line(0,1){0.8}}
\put(1.2,0.6){\line(1,0){2}}
\put(0.2,0){\line(1,0){2.7}}
\put(3.5,0){\line(1,0){0.7}}
\end{picture}
}}\]
\newline
If \(s\in\mathcal{P}(k,l)\) and \(t\in\mathcal{P}(l,m)\), we can construct the \emph{composition} \(ts=t\circ s\).
It is defined by vertical concatenation:
We place the partition \(t\) below \(s\) and connect each lower points of \(s\) with the corresponding upper points of \(t\).
This way we might obtain blocks in the middle that are connected neither to any of the very upper nor lower points.
We denote these blocks as \emph{remaining loops} and their number as \(\textnormal{rl}(t,s)\).
Finally we erase all remaining loops and middle points:
\[pq=p\circ q=
\begin{picture}(5,2.4)
\put(3.2,0){\circle{0.6}}
\put(2.8,-0.33) {\linethickness{0.3cm}\color{white}\line(1,0){0.8}}
\put(0,-1.5) {$\circ$}
\put(1,-1.5) {$\circ$}
\put(2,-1.5) {$\circ$}
\put(3,-1.5) {$\circ$}
\put(4,-1.5) {$\circ$}
\put(0,1.5) {$\circ$}
\put(1,1.5) {$\circ$}
\put(2,1.5) {$\circ$}
\put(0.2,-1){\line(0,1){1}}
\put(1.2,-1){\line(0,1){1}}
\put(2.2,-1){\line(0,1){0.7}}
\put(3.2,-1){\line(0,1){1.6}}
\put(4.2,-1){\line(0,1){1}}
\put(0.2,0.4){\line(0,1){1}}
\put(1.2,0.6){\line(0,1){0.8}}
\put(2.2,0.6){\line(0,1){0.8}}
\put(1.2,0.6){\line(1,0){2}}
\put(0.2,0){\line(1,0){2.7}}
\put(3.5,0){\line(1,0){0.7}}
\end{picture}
\circ\quad
\begin{picture}(4,2.4)
\put(0,-1.5) {$\circ$}
\put(1,-1.5) {$\circ$}
\put(2,-1.5) {$\circ$}
\put(0,1.5) {$\circ$}
\put(1,1.5) {$\circ$}
\put(2,1.5) {$\circ$}
\put(3,1.5) {$\circ$}
\put(0.2,-1){\line(0,1){0.8}}
\put(1.2,-1){\line(0,1){1.2}}
\put(2.2,-1){\line(0,1){1.2}}
\put(0.2,0.2){\line(0,1){1.2}}
\put(1.2,0.2){\line(0,1){1.2}}
\put(2.2,0.6){\line(0,1){0.8}}
\put(3.2,0.2){\line(0,1){1.2}}
\put(0.2,0.2){\line(1,0){1}}
\put(2.2,0.2){\line(1,0){1}}
\end{picture}
=
\raisebox{1cm}{
\begin{picture}(5,2.4)
\put(0,-1.5) {$\circ$}
\put(1,-1.5) {$\circ$}
\put(2,-1.5) {$\circ$}
\put(0,1.5) {$\circ$}
\put(1,1.5) {$\circ$}
\put(2,1.5) {$\circ$}
\put(3,1.5) {$\circ$}
\put(0.2,-1){\line(0,1){0.8}}
\put(1.2,-1){\line(0,1){1.2}}
\put(2.2,-1){\line(0,1){1.2}}
\put(0.2,0.2){\line(0,1){1.2}}
\put(1.2,0.2){\line(0,1){1.2}}
\put(2.2,0.6){\line(0,1){0.8}}
\put(3.2,0.2){\line(0,1){1.2}}
\put(0.2,0.2){\line(1,0){1}}
\put(2.2,0.2){\line(1,0){1}}
\end{picture}}
\textnormal{\hspace{-2.63cm}}
\raisebox{-1cm}{
\begin{picture}(5,2.4)
\put(3.2,0){\circle{0.6}}
\put(2.8,-0.33) {\linethickness{0.3cm}\color{white}\line(1,0){0.8}}
\put(0,1.5) {$\circ$}
\put(1,1.5) {$\circ$}
\put(2,1.5) {$\circ$}
\put(0,-1.5) {$\circ$}
\put(1,-1.5) {$\circ$}
\put(2,-1.5) {$\circ$}
\put(3,-1.5) {$\circ$}
\put(4,-1.5) {$\circ$}
\put(0.2,2){\line(0,1){0.4}}
\put(1.2,2){\line(0,1){0.4}}
\put(2.2,2){\line(0,1){0.4}}
\put(0.2,-1){\line(0,1){1}}
\put(1.2,-1){\line(0,1){1}}
\put(2.2,-1){\line(0,1){0.7}}
\put(3.2,-1){\line(0,1){1.6}}
\put(4.2,-1){\line(0,1){1}}
\put(0.2,0.4){\line(0,1){1}}
\put(1.2,0.6){\line(0,1){0.8}}
\put(2.2,0.6){\line(0,1){0.8}}
\put(1.2,0.6){\line(1,0){2}}
\put(0.2,0){\line(1,0){2.7}}
\put(3.5,0){\line(1,0){0.7}}
\end{picture}
}
\!\!\!=
\begin{picture}(5,2.4)
\put(3.2,0){\circle{0.6}}
\put(2.8,-0.33) {\linethickness{0.3cm}\color{white}\line(1,0){0.8}}
\put(0,-1.5) {$\circ$}
\put(1,-1.5) {$\circ$}
\put(2,-1.5) {$\circ$}
\put(3,-1.5) {$\circ$}
\put(4,-1.5) {$\circ$}
\put(0,1.8) {$\circ$}
\put(1,1.8) {$\circ$}
\put(2,1.8) {$\circ$}
\put(3,1.8) {$\circ$}
\put(0.2,-1){\line(0,1){1}}
\put(1.2,-1){\line(0,1){1}}
\put(2.2,-1){\line(0,1){0.8}}
\put(3.2,-1){\line(0,1){1.6}}
\put(4.2,-1){\line(0,1){1}}
\put(0.2,0.6){\line(0,1){1}}
\put(1.2,0.6){\line(0,1){1}}
\put(2.2,0.8){\line(0,1){0.8}}
\put(3.2,0.6){\line(0,1){1}}
\put(0.2,0.6){\line(1,0){3}}
\put(0.2,0){\line(1,0){2.7}}
\put(3.5,0){\line(1,0){0.7}}
\end{picture}
\]
\vspace{22pt}

Note that the involution deserves its name as it holds \((s^*)^*=s\) and \((ts)^*=s^*t^*\) whenever \(t\) and \(s\) are composable.
\newline
The operator \vspace{2pt}\(\textnormal{rot}_{\!\!\raisebox{2pt}{\rotatebox{270}{$^\curvearrowright$}}}\) is defined to take the rightmost point in the upper row of a partition and move it to the right end of the lower row without changing the connections to other points.
\[\textnormal{rot}_{\!\!\raisebox{2pt}{\rotatebox{270}{$^\curvearrowright$}}}(q)=
\textnormal{rot}_{\!\!\raisebox{2pt}{\rotatebox{270}{$^\curvearrowright$}}}(\;\,
\setlength{\unitlength}{0.5cm}
\begin{picture}(3.8,2.4)
\put(0,-1.5) {$\circ$}
\put(1,-1.5) {$\circ$}
\put(2,-1.5) {$\circ$}
\put(0,1.5) {$\circ$}
\put(1,1.5) {$\circ$}
\put(2,1.5) {$\circ$}
\put(3,1.5) {$\circ$}
\put(0.2,-1){\line(0,1){0.8}}
\put(1.2,-1){\line(0,1){1.2}}
\put(2.2,-1){\line(0,1){1.2}}
\put(0.2,0.2){\line(0,1){1.2}}
\put(1.2,0.2){\line(0,1){1.2}}
\put(2.2,0.6){\line(0,1){0.8}}
\put(3.2,0.2){\line(0,1){1.2}}
\put(0.2,0.2){\line(1,0){1}}
\put(2.2,0.2){\line(1,0){1}}
\end{picture}
)
=
\setlength{\unitlength}{0.5cm}
\begin{picture}(5,2)
\put(0,-1.5) {$\circ$}
\put(1,-1.5) {$\circ$}
\put(2,-1.5) {$\circ$}
\put(3,-1.5) {$\circ$}
\put(0,1.5) {$\circ$}
\put(1,1.5) {$\circ$}
\put(2,1.5) {$\circ$}
\put(0.2,-1){\line(0,1){0.8}}
\put(1.2,-1){\line(0,1){1.2}}
\put(2.2,-1){\line(0,1){1}}
\put(0.2,0.2){\line(0,1){1.2}}
\put(1.2,0.2){\line(0,1){1.2}}
\put(2.2,0.6){\line(0,1){0.8}}
\put(3.2,-1){\line(0,1){1}}
\put(0.2,0.2){\line(1,0){1}}
\put(2.2,0){\line(1,0){1}}
\end{picture}\]\vspace{6pt}

Analogously, we define the operation \vspace{2pt}\(\textnormal{rot}_{\!\!\raisebox{2pt}{\rotatebox{270}{$^\curvearrowleft$}}}\) into the opposite direction and the corresponding operations on the left side of the partition, \vspace{2pt}\(\textnormal{rot}_{\raisebox{-6pt}{\rotatebox{90}{$^\curvearrowleft$}}}\) and \vspace{2pt}\(\textnormal{rot}_{\raisebox{-6pt}{\rotatebox{90}{$^\curvearrowright$}}}\).
A partition \(q'\) is called a \emph{rotated version of \(q\)} if it is obtained from \(q\) by repeated application of these four maps.
Note that a rotation operator is only defined on partitions where there exists a point with whom the rotation can be performed.
On the partition \(r\) from above only the rotations from the lower to the upper row are well-defined.
\begin{rem}\label{rem:identity_partitions_deserve_their_names}
The partition \(\idpartww\) is called  \emph{identity partitions}.
Taking the \(k\)-th tensor powers of this partition, we obtain the neutral element with respect to composition from the left on  \(\mathcal{P}(k,l)\) as well as the neutral element with respect to composition from the right on \(\mathcal{P}(m,k)\).
For example
\[\left(\raisebox{-4pt}{\idpartww}\otimes\raisebox{-4pt}{\idpartww}\otimes \raisebox{-4pt}{\idpartww}\right)
\circ
\setlength{\unitlength}{0.5cm}
\begin{picture}(4,1.5)
\put(0,-1.5) {$\circ$}
\put(1,-1.5) {$\circ$}
\put(2,-1.5) {$\circ$}
\put(0,1.5) {$\circ$}
\put(1,1.5) {$\circ$}
\put(2,1.5) {$\circ$}
\put(3,1.5) {$\circ$}
\put(0.2,-1){\line(0,1){0.8}}
\put(1.2,-1){\line(0,1){1.2}}
\put(2.2,-1){\line(0,1){1.2}}
\put(0.2,0.2){\line(0,1){1.2}}
\put(1.2,0.2){\line(0,1){1.2}}
\put(2.2,0.6){\line(0,1){0.8}}
\put(3.2,0.2){\line(0,1){1.2}}
\put(0.2,0.2){\line(1,0){1}}
\put(2.2,0.2){\line(1,0){1}}
\end{picture}
=\;q\;=
\begin{picture}(4,1.5)
\put(0,-1.5) {$\circ$}
\put(1,-1.5) {$\circ$}
\put(2,-1.5) {$\circ$}
\put(0,1.5) {$\circ$}
\put(1,1.5) {$\circ$}
\put(2,1.5) {$\circ$}
\put(3,1.5) {$\circ$}
\put(0.2,-1){\line(0,1){0.8}}
\put(1.2,-1){\line(0,1){1.2}}
\put(2.2,-1){\line(0,1){1.2}}
\put(0.2,0.2){\line(0,1){1.2}}
\put(1.2,0.2){\line(0,1){1.2}}
\put(2.2,0.6){\line(0,1){0.8}}
\put(3.2,0.2){\line(0,1){1.2}}
\put(0.2,0.2){\line(1,0){1}}
\put(2.2,0.2){\line(1,0){1}}
\end{picture}
\circ
\left(\raisebox{-4pt}{\idpartww}\otimes\raisebox{-4pt}{\idpartww}\otimes \raisebox{-4pt}{\idpartww}\otimes \raisebox{-4pt}{\idpartww}\right).
\vspace{11pt}
\]
\end{rem}
Rotating the  identity partition onto one line we obtain the two so-called \emph{ pair partitions} \(\paarpartww,\baarpartww\).
\vspace{11pt}
\newline

\begin{defn}\label{defn:category_of_partitions}
A \emph{category of partitions \(\mathcal{C}\)} is a subset of \(\mathcal{P}\) such that 
\begin{itemize}
\item[(i)] it contains the identity partition, \(\idpartww\),
\item[(ii)] it contains the two pair partitions \(\paarpartww\), \(\baarpartww\) and
\item[(iii)] it is closed under composition, involution and taking tensor products.
\end{itemize}
A category of partitions is called \emph{non-crossing} if all its elements are non-crossing.
We define for all \(k,l\!\in\!\N_0\) the subsets 
\[\mathcal{C}(k,l):=\mathcal{C}\cap\mathcal{P}(k,l)\]
\end{defn}
The two pair partitions mentioned above guarantee that a category of partitions is rotation-invariant, compare {banicaspeicherliberation}.
\begin{prop}\label{prop:mcpp_is_equivalent_to_closure_under_rotation}
A category of partitions is closed under taking rotated versions of elements.
\end{prop}
\begin{proof}
Consider for example the operator \(\textnormal{rot}_{\!\!\raisebox{2pt}{\rotatebox{270}{$^\curvearrowright$}}}\) and for \(k\in\N\) a partition \(p\in\mathcal{P}(k,l)\).
Then it holds
\[\textnormal{rot}_{\!\!\raisebox{2pt}{\rotatebox{270}{$^\curvearrowright$}}}(p)=\left(\raisebox{-4pt}{$\idpartww^{\otimes k}$}\otimes \raisebox{-4pt}{$\idpartww$}\right)\circ p\circ\left(\raisebox{-4pt}{$\idpartww^{\otimes k-1}$}\otimes \paarpartww\right).\] 
Analogously, one proves the statement for the other three rotations.
\end{proof}
In the following we associate with a given partition \(p\) a linear map \(T_p\) on some finite-dimensional Hilbert space.
In fact, this will result in a family of maps \(\left(T_p(N)\right)_{N\in\N}\) as we can vary the dimension \(N\) of the considered Hilbert space.
\begin{defn}
Consider a partition \(p\in\mathcal{P}(k,l)\) and two multi-indices \(i=(i_1,\ldots,i_k)\in\N^k\) and \(j=(j_1,\ldots,j_l)\in\N^l\). The pair \((i,j)\) defines a labelling of \(p\) by mapping the upper points of \(p\) from left to right to the numbers \(i_1,\ldots,i_k\) and the lower points to \(j_1,\ldots,j_l\).
\newline
We call \((i,j)\) a \emph{valid labelling} for \(p\) if connected points of the partitions are mapped to the same numbers. Otherwise it is called an \emph{invalid labelling}. 
\newline
Analogously, we can only consider \(i\) (or only \(j\)) and call it a valid labelling for the upper (or lower) points of \(p\)  if all upper (or all lower) points that are connected get the same labels.
\end{defn}
Of course \((i,j)\) can only be a valid labelling for \(p\) if  both \(i\) and \(j\) are valid labellings for their respective row of points.
\begin{ex}
Consider again the partition \(p\) from Equation \ref{eqn:examples_of_partitions}:
\vspace{11pt}
\[
p=
\setlength{\unitlength}{0.5cm}
\begin{picture}(6,2.4)
\put(3.2,0){\circle{0.6}}
\put(2.8,-0.33) {\linethickness{0.3cm}\color{white}\line(1,0){0.8}}
\put(0,-1.5) {$\circ$}
\put(1,-1.5) {$\circ$}
\put(2,-1.5) {$\circ$}
\put(3,-1.5) {$\circ$}
\put(4,-1.5) {$\circ$}
\put(0,1.5) {$\circ$}
\put(1,1.5) {$\circ$}
\put(2,1.5) {$\circ$}
\put(0,-2.3) {$1'$}
\put(1,-2.3) {$2'$}
\put(2,-2.3) {$3'$}
\put(3,-2.3) {$4'$}
\put(4,-2.3) {$5'$}
\put(0,2.3) {$1$}
\put(1,2.3) {$2$}
\put(2,2.3) {$3$}
\put(0.2,-1){\line(0,1){1}}
\put(1.2,-1){\line(0,1){1}}
\put(2.2,-1){\line(0,1){0.7}}
\put(3.2,-1){\line(0,1){1.6}}
\put(4.2,-1){\line(0,1){1}}
\put(0.2,0.4){\line(0,1){1}}
\put(1.2,0.6){\line(0,1){0.8}}
\put(2.2,0.6){\line(0,1){0.8}}
\put(1.2,0.6){\line(1,0){2}}
\put(0.2,0){\line(1,0){2.7}}
\put(3.5,0){\line(1,0){0.7}}
\end{picture}\vspace{1cm}
\]
Then the following holds:
\begin{itemize}
\item \(i\!=\!(5,6,6)\) is a valid labelling for the upper row but \(i'\!=\!(5,6,5)\) is not.
\item \(j=(3,3,7,6,3)\) is a valid labelling for the lower row but \(j'=(3,3,7,2,8)\) is not.
\item \((i,j)\) is a valid labelling for \(p\) because both \(i\) and \(j\) are valid for their respective row of \(p\) and, in addition, the points 2 and 3 and \(4'\) are all labelled by `6'.
\end{itemize}
\end{ex}
\begin{defn}\label{defn:delta_p}
Let \(k,l\in\N_0\) and \(p\in\mathcal{P}(k,l)\). Then we define
\[\delta_p:\N^k\times\N^l\rightarrow \{0,1\}\;;\; (i,j)\mapsto
\begin{cases}
1&,(i,j)\textnormal{ is a valid labelling for \(p\)}\\
0&,\textnormal otherwise.
\end{cases}\]
\end{defn}
\begin{defn}\label{defn:T_p}
Let \(N\!\in\!\N\) and \(p\in\mathcal{P}(k,l)\) for some \(k,l\in\N_0\).
Consider the Hilbert space \(\C^N\) with canonical orthonormal basis \((e_i)_{i\in [N]}\). Then we define a linear map \(T_p\) as follows:
\[T_p:\left(\C^N\right)^{\otimes k}\rightarrow \left(\C^N\right)^{\otimes l}\;;\;e_{i_1}\otimes\ldots\otimes e_{i_k}\mapsto\sum_{j\in [N]^l}\delta_p(i,j)\left(e_{j_1}\otimes\ldots\otimes e_{j_l}\right)\]
\end{defn}
The important observation at this point is, that the operations on partitions from the last section translate in a nice way to operations on the maps \(T_p\), compare \cite[Prop. 1.9]{banicaspeicherliberation}.
\begin{lem}\label{lem:properties_of_the_T_p}
Let \(N\in\N\) and \(p,q\) be partitions. Then it holds:
\begin{itemize}
\item[(1)] \(T_{q\otimes p}=T_q\otimes T_p\).
\item[(2)] \(T_{p^*}=\left(T_{p}\right)^*\)
\item[(3)] \(T_{qp}=N^{-\textnormal{rl(q,p)}}\left(T_q\circ T_p\right)\)
\end{itemize}
where for the last statement \(p\) must be composable with \(q\) from the left and \(\textnormal{rl}(q,p)\) denotes the number of remaining loops when writing \(p\) on top of \(q\) and connecting the middle points.
\end{lem}
For fixed \(N\in\N\) and any category of partitions \(\mathcal{C}\) the collection \(\big(T_p\big)_{p\in\mathcal{C}}\) gives, via a concrete monoidal \(W^*\)-category \(R_N(\mathcal{C})\) and Tannaka-Krein duality, rise to a CMQG \(G_N(\mathcal{C})\), see for example \cite{tarragoweberclassificationunitaryQGs}.
We state this result in form of the following definition.
\begin{defn}\label{defn:easy_QGs}
Consider \(N\in\N\), let \(u=\big(u_{ij}\big)\) be a  matrix of generators and \(\mathcal{C}\) a category of partitions.
For \(p\in\mathcal{C}\) let \(T_p\) be the linear map as defined in Definition \ref{defn:T_p} for the above \(N\in\N\).
\begin{itemize}
\item[(a)] Given \(k\in\N_0\), we define
\[u^{\otop k}:=\sum_{i_1,\ldots,i_k=1}^{N}\sum_{j_1,\ldots,j_k=1}^{N}E_{i_1j_1}\otimes\ldots\otimes E_{i_kj_k}\otimes u_{i_1j_1}\cdots u_{i_kj_k}.\]
\item[(b)] The (well-defined and unique) real compact matrix quantum group \(G_N(\mathcal{C})=(A,u)\) that fulfils for all \(k,l\in\N_0\)
\[\big\{\,T:(\C^N)^{\otimes k}\rightarrow (\C^N)^{\otimes l}\,|\,T\textnormal{ linear},\;Tu^{\otop k}\;=\;u^{\otop l}T\big\}=\Span\big(\{\,T_p\,|\,p\in\mathcal{C}(k,l)\,\}\big)\]
is called an easy quantum group.
\item[(c)] Two easy quantum groups \(G_{N_1}(\mathcal{C}_1)\) and \(G_{N_2}(\mathcal{C}_1)\) are considered the same (in the sense of equivalence) if \(N_1=N_2\) and if they have the same intertwiner spaces, i.e. for all \(k,l\in\N_0\) it holds
\[\Span\big(\{\,T_p\,|\,p\in\mathcal{C}_1(k,l)\,\}\big)=\Span\big(\{\,T_p\,|\,p\in\mathcal{C}_2(k,l)\,\}\big).\]
\end{itemize}
\end{defn}
Note that the definition of equality/equivalence in part (c) is the usual definition of (weak) equivalence for CMQGs as defined in \cite{woronowiczpseudogroups}.
\section{Reformulating the injectivity of \(\Psi\) as a problem of linear independence}
\label{sec:reformulating_the_injectivity_of_Psi_as_a_problem_of_linear_independence}
As mentioned in the previous section, the question of (in)equality of two easy quantum groups \(G_N(\mathcal{C}_1)\) and \(G_N(\mathcal{C}_2)\) can be formulated on the level of intertwiner maps:
\begin{question}\label{quest:inital_question_chapter_3}
Consider two different categories of partitions \(\mathcal{C}_1\neq \mathcal{C}_2\) and fix some \(N\in\N\).
\newline
Do we have 
\[(\Span\big(\{T_p\,|\,p\in\mathcal{C}_1(k,l)\}\big)\neq\Span\big(\{T_p\,|\,p\in\mathcal{C}_2(k,l)\}\big)\]
for at least  one pair \((k,l)\)?
In that case the associated easy quantum groups \(G_N(\mathcal{C}_1)\) and \(G_N(\mathcal{C}_2)\) differ.
\end{question}
Our first simplification is the restriction to the case \(k=0\), the empty word.
\begin{prop}\label{prop:equivalence_of_rotated_situation}
Let \(\mathcal{C}_1\) and \(\mathcal{C}_2\) be two categories of partitionsand consider \(k,l\in\N_0\).
Then the sets
\[\{T_p\;|\;p\in\mathcal{C}_1(k,l)\}\quad\quad\textnormal{and}\quad\quad
\{T_q\;|\;q\in\mathcal{C}_2(k,l)\}\]
span the same space of operators in \(B(\left(\C^N\right)^{\otimes k},\left(\C^N\right)^{\otimes l})\) if and only if their rotated versions span the same space of operators.
\end{prop}
\begin{proof}
Without restriction assume \(k>0\) and consider the rotation \(\textnormal{rot}_{\!\!\raisebox{2pt}{\rotatebox{270}{$^\curvearrowright$}}}\).
Note that by Proposition \ref{prop:mcpp_is_equivalent_to_closure_under_rotation} the categories above are closed under rotations.
\newline
The statement above is just the observation that an equation \(\sum_p\alpha_pT_p=\sum_q\beta_qT_q\) implies 
\[\sum_p \alpha_q\textnormal{rot}_{\!\!\raisebox{2pt}{\rotatebox{270}{$^\curvearrowright$}}}(T_p)
=\sum_q \beta_q\textnormal{rot}_{\!\!\raisebox{2pt}{\rotatebox{270}{$^\curvearrowright$}}}(T_q).\]
Of course the analogous argument works for all other rotations, showing the `only if' part of the claim.
The `if' part is the observation that rotation operations are invertible.
\end{proof}
\begin{cor}\label{cor:restriction_to_lower_points_for_linear_independence}
Let \(\mathcal{C}_1\) and \(\mathcal{C}_2\) be two categories of partitions.
Fix \(N\in\N\) and let \(G_N(\mathcal{C}_1)=(A_1,u_1)\) and \(G_N(\mathcal{C}_2)=(A_2,u_2)\) be the respective easy quantum groups.
Then \(G_N(\mathcal{C}_1)\) and \(G_N(\mathcal{C}_2)\) are equal if and only if for all \(n\in\N_0\) the spaces
\[\Span\big(\{T_p\;|\;p\in\mathcal{C}_1(0,n)\}\big)\quad\quad\textnormal{and}\quad\quad
\Span\big(\{T_q\;|\;q\in\mathcal{C}_2(0,n)\}\big)\]
coincide.
\end{cor}
This result enables us to formulate a sufficient condition for different categories to produce different easy quantum groups:
\begin{cor}\label{cor:sufficient_condition_for_different_easy_quantum_groups}
Consider the situation of Corollary \ref{cor:restriction_to_lower_points_for_linear_independence} with two different categories \(\mathcal{C}_1\) and \(\mathcal{C}_2\). 
If there is an \(n\in\N_0\) such that \(\mathcal{C}_1(0,n)\neq\mathcal{C}_2(0,n)\)  and the collection 
\[\Big(T_p\Big)_{p\in\mathcal{C}_1(0,n)\cup\mathcal{C}_2(0,n) }\]
is linear independent, then \(G_N(\mathcal{C}_1)\) and \(G_N(\mathcal{C}_2)\) are different.
\end{cor}
\begin{proof}
The corollary describes a sufficient condition for the operator spaces \(\Hom_1(0,n)\) and \(\Hom_2(0,n)\) to be different.
Using Tannaka-Krein duality, we conclude in this case that the associated easy quantum groups cannot be equal.
\end{proof}
In order to use Corollary \ref{cor:sufficient_condition_for_different_easy_quantum_groups} (finally to answer Question \ref{quest:inital_question_chapter_3}) we aim to solve the following problem:
\begin{question}
Consider a category of partitions \(\mathcal{C}\) and \(N\in\N\).
Under which conditions, on \(\mathcal{C}\) and/or \(N\), is the collection
\[\big(T_p\big)_{p\in\mathcal{C}(0,n)}\]
for a given \(n\in\N_0\) linearly independent?
\end{question}
Given linear independence as described above, we have that different subcategories of \(\mathcal{C}\) produce different easy quantum groups.
\begin{rem}
Apart from the question of one-to-one correspondence between categories of partitions and easy quantum groups, we have other problems and theories depending a lot on the question of linear independence of the maps \(T_p\).
For example the fusion rules of easy quantum groups are well-established in this linear independent situation but not in the general one, see \cite{freslon} and \cite{freslonweber}.
\end{rem}
\section{Linear (in)dependences in the general situation}
\label{sec:linear_(in)dependencies_in_the_general_situation}
In virtue of Corollary \ref{cor:sufficient_condition_for_different_easy_quantum_groups} it is worth to investigate linear (in)dependences of the maps \(T_p\) where \(p\) is from some fixed \(\mathcal{C}(0,n)\).
Recall that in this case \(T_p\) is a map from \(\left(\C^N\right)^{\otimes 0}=\C\) to \(\left(\C^N\right)^{\otimes n}\).
It turns out that the crucial point is the relation between the natural number \(N\in\N\) and the number of blocks in the partitions \(p\in\mathcal{C}(0,n)\).
\newline
In order to state the results and their proofs we need two further notations. They are well known, see for example \cite[p. 35]{mingospeicher} and \cite[Def. 9.14]{nicaspeicher}, respectively.
\begin{defn}
Let \(n\in\N_0\) and \(j=(j_1,\ldots,j_n)\in\N^n\) a multi-index.
Then we denote with \(\ker(j)\in\mathcal{P}(0,n)\) the partition on \(n\) lower points where \(j\) numbers its points from left to right such that points are connected by \(\ker(j)\) if and only if they have the same number.
\end{defn}
\begin{defn}
Let \(p,q\in\mathcal{P}(0,n)\) for some \(n\in\N_0\).
We write \(p\preceq q\) if and only if each valid labelling \((i,j)\) of \(q\) is also a valid labelling of \(p\).
\end{defn}
In other words, \(p\preceq q\) holds if we can obtain \(p\) from \(q\) by refining the blocks in the partition.
This is obviously a partial ordering, the partition that has only singletons as blocks is the minimal element in \(\mathcal{P}(0,n)\) and the one-block partition, where all points are connected to each other, is the maximal element.
The following result can also be found in \cite[Lem. 3.4]{maassen}, for example.
\begin{prop}\label{prop:linear_independence_if_blocks_do_not_exceed_N}
Let \(N\!\in\!\N\) and \(n\in\N_0\).
Consider for every partition \(p\in\mathcal{P}(0,n)\) the linear map \(T_p: \left(\C^N\right)^{\otimes 0}=\C\rightarrow \left(\C^N\right)^{\otimes n}\) as defined in Definition \ref{defn:T_p}.
Denote further
\[\mathcal{P}_N(0,n):=\{p\in \mathcal{P}(0,n)\;|\; p\textnormal{ has at most }N\textnormal{ blocks}\}.\] 
Then the collection of maps
\[\Big(T_p\Big)_{p\in\mathcal{P}_N(0,n)}\]
is linearly independent.
\end{prop}
\begin{proof}
Let \(e_1,\ldots, e_N\) be the standard orthonormal basis of \(\C^N\) that has already been used to define the maps \(T_p\).
Given a multi-index \(j=(j_1,\ldots,j_n)\) we write \(e_j:=e_{j_1}\otimes\ldots\otimes e_{j_n}\in \left(\C^N\right)^{\otimes n}\).
\newline
Consider now a linear combination
\begin{equation}\label{eqn:linear_combination_of_T_p's_that_should_be_zero}
0=\sum_{p\in\mathcal{P}_N(0,n)}\alpha_pT_p.
\end{equation}
We prove \(\alpha_p=0\) by induction on the number of blocks of \(p\). 
The base case is the largest possible number of blocks, say \(M\), so consider an arbitrary partition \(p\) with exactly \(M\) blocks.
The idea is to show that  there is  a direction \(\langle v\rangle\) in \(\left(\C^N\right)^{\otimes n}\) such that all \(T_q(1)\) are orthogonal to \(\langle v\rangle\) except \(T_p(1)\).
\newline
As \(M\) does not exceed \(N\) by assumption, we find a multi-index \(j\in[N]^n\) such that \(\ker(j)\) coincides with \(p\).
We now claim the following: 
\begin{equation}\label{eqn:T_q(1)_is_orthogonal_to_(e_j)}
\langle T_q(1),e_{j}\rangle=
\begin{cases}
1&,q=p\\
0&,q\neq p
\end{cases}.
\end{equation}
To prove this, recall that for \(q\in\mathcal{P}(0,n)\) the map \(T_q\) is uniquely defined by the image \(T_q(1)\) and it holds 
\begin{equation}\label{eqn:image_of_T_q(1)}
T_q(1)=\sum_{\begin{matrix}\scriptstyle i\in[N]^{n} \\\scriptstyle q\preceq \ker(i)\end{matrix}} e_i.
\end{equation}
Therefore, the case \(q=p\) in the claimed Equation \ref{eqn:T_q(1)_is_orthogonal_to_(e_j)} is  clear.
Now assume \(q\neq p\) and consider a multi-index \(i\in[N]^n\) with \(q\preceq\ker(i)\) as in the summation in Equation \ref{eqn:image_of_T_q(1)}.
As \(q\) has at most as many blocks as \(p=\ker(j)\), there are two points in the same block of \(q\) which are in different blocks of \(\ker(j)\).
Together with \(\ker(i)\succeq q\) we deduce that these two points are in the same block of \(\ker(i)\), i.e.  \(i\) and \(j\) must differ at least at one entry.
But then we have \(\langle e_i,e_j\rangle=0\) for all  \(\ker(i)\succeq q\) and this shows the claimed identity \(\langle T_q(1),e_j\rangle=0\) for \(q\neq p\).
\newline
Combining Equation \ref{eqn:T_q(1)_is_orthogonal_to_(e_j)} with the assumption, Equation \ref{eqn:linear_combination_of_T_p's_that_should_be_zero}, we deduce
\[\alpha_p=\langle \sum_{p\in\mathcal{P}_N(0,n)}\alpha_pT_p(1),e_j\rangle=0.\]
Of course this argument holds separately for all \(p\) with exactly \(M\) blocks, finishing the base case. 
\newline
The induction step is just the observation that we can repeat the arguments above for coefficients \(\alpha_q\) not yet proven to be zero.
After at most \(M\) steps we have proved \(\alpha_p=0\) for all \(p\), so linear independence holds as claimed. 
\end{proof}
We now prove the converse result of Proposition \ref{prop:linear_independence_if_blocks_do_not_exceed_N}: Whilst for given \(n\in\N_0\) the partitions \(\mathcal{P}_N(0,n)\) give rise to a linear independent set of maps \(T_p\), we prove now that the remaining maps \(T_q\) do not enlarge the generated space of linear maps.
\begin{prop}\label{prop:partitions_with_at_most_N_blocks_generate_everything}
In the situation of Proposition \ref{prop:linear_independence_if_blocks_do_not_exceed_N} we have
\[\Span\big(\{T_p\;|\;p\in\mathcal{P}_N(0,n)\}\big)=\Span\big(\{T_p\;|\;p\in\mathcal{P}(0,n)\}\big).\]
Hence, by Proposition \ref{prop:linear_independence_if_blocks_do_not_exceed_N}, the collection \(\big(T_p\big)_{p\in P_N(0,n)}\) is a basis for \(\Hom(0,n)\).
\end{prop}
\begin{proof}
We use the same notations as in the proof of Propostion \ref{prop:linear_independence_if_blocks_do_not_exceed_N}.
Recall that \(b(p)\) denotes the number of blocks of a partition \(p\).
\newline
For a partition \(q\) with more than \(N\) blocks we have to prove 
\[T_q\in\Span\big(\{T_p\;|\;p\in\mathcal{P}_N(0,n)\}\big).\]
To do so, we recursively construct linear combinations \(L_N,L_{N-1},\ldots,L_1\) of \(T_p\)'s fulfilling the following:
\[b\big(\ker(j)\big)\ge M\quad\Rightarrow\quad\langle L_M(1),e_j\rangle=\langle T_q(1),e_j\rangle\quad,\]
for \(1\le M\le N\).
Roughly speaking, with decreasing \(M\), \(L_M(1)\) will coincide with \(T_q(1)\) in more and more directions until finally \(L_1(1)\) coincides with \(T_q(1)\), so \(L_1=T_q\).
\newline
\textbf{Base case: Construction of \(L_N\):}
Consider the linear combination
\begin{equation}\label{eqn:L_N}
L_N:=\sum_{\begin{matrix}\scriptstyle q\preceq p\\ \scriptstyle b(p)= N\end{matrix}}T_p
\end{equation}
and we have to show that
\[
\langle L_N(1),e_j\rangle=\langle T_q(1),e_j\rangle
\]
whenever \(\ker(j)\) has \(N\) blocks.
Using the definitions of \(L_N\) and \(T_q\), this Equation reads
\begin{equation}\label{eqn:compare_L_N_and_T_q}
\langle \sum_{\begin{matrix}\scriptstyle q\preceq p\\ \scriptstyle b(p)= N\end{matrix}}\sum_{\begin{matrix}\scriptstyle i\in[N]^{n} \\\scriptstyle p\preceq \ker(i)\end{matrix}} e_i,e_j\rangle
=\langle \sum_{\begin{matrix}\scriptstyle i\in[N]^{n} \\\scriptstyle q\preceq \ker(i)\end{matrix}} e_i,e_j\rangle.
\end{equation}
We have to prove that this is true whenever \(b\big(\ker(j)\big)=N\).
\newline
\textbf{Case 1: \(q\not\preceq \ker(j)\): }
On both the left and the right side of Equation \ref{eqn:compare_L_N_and_T_q} it holds \(q\preceq \ker(i)\).
Together with \(q\not\preceq \ker(j)\) and transitivity of \(\preceq\) this implies \(\ker(i)\not\preceq \ker(j)\), so \(i\neq j\) in all cases and both sides vanish.
\newline
\textbf{Case 2: \(q\preceq \ker(j)\): }
As \(\ker(j)\) has \(N\) blocks, \(\ker(j)\) is equal to exactly one partition \(\hat{p}\in P_N(0,n)\) and for all other partitions \(p\in P_N(0,n)\) it holds \(p\not\preceq \ker(j)\).
As above, we deduce for \(p\not\preceq \ker(j)\) and \(p\preceq\ker(i)\) that \(i\neq j\).
Consequently, Equation \ref{eqn:compare_L_N_and_T_q} reads
\[\langle \sum_{\begin{matrix}\scriptstyle i\in[N]^{n} \\\scriptstyle \hat{p}\preceq \ker(i)\end{matrix}} e_i,e_j\rangle
=\langle \sum_{\begin{matrix}\scriptstyle i\in[N]^{n} \\\scriptstyle q\preceq \ker(i)\end{matrix}} e_i,e_j\rangle.\]
As \(\hat{p}:=\ker(j)\) and \(q\preceq \ker(j)\), we have that \(e_j\) appears on both sides as one of the summands \(e_i\).
So we end up with a true statement, finishing the base case.
\newline
\textbf{Induction step: Constructing \(L_{M}\) from \(L_{M+1}\):}
\newline
Assume that there is for some \(1\le M< N\) a linear combination
\[L_{M+1}=\sum_{\begin{matrix}\scriptstyle q\preceq p\\ \scriptstyle M+1 \le b(p)\le N\end{matrix}}\alpha_pT_p\]
such that
\begin{equation}\label{eqn:compare_L_M+1_and_T_q}
b\big(\ker(j)\big)\ge M+1\quad\Rightarrow\quad\langle L_{M+1}(1),e_j\rangle=\langle T_q(1),e_j\rangle.
\end{equation}
We want to construct some \(L_M\) such that the analogue of Equation \ref{eqn:compare_L_M+1_and_T_q} holds for all \(j\) with \(b\big(\ker(j)\big)\ge M\).
\newline
For a given \(p\succeq q\) with \(b(p)=M\) consider any \(\hat{k}\in[N]^n\) with \(\ker(\hat{k})=p\) and define
\[\alpha_p:=1-\langle L_{M+1}(1),e_{\hat{k}}\rangle.\]
Note that this is independent of the chosen \(\hat{k}\).
\newline
Now enlarge the sum \(L_{M+1}\) in the following way:
\begin{equation}\label{eqn:defn_of_L_M}
L_M:=L_{M+1}+\sum_{\begin{matrix}\scriptstyle q\preceq p\\ \scriptstyle b(p)=M\end{matrix}}\alpha_pT_p.
\end{equation}
We have to show that \(L_M\) fulfils
\begin{equation}\label{eqn:compare_L_M_and_T_q}
b\big(\ker(j)\big)\ge M\quad\Rightarrow\quad\langle L_{M}(1),e_j\rangle=\langle T_q(1),e_j\rangle.
\end{equation}
\textbf{Case 1: \(b\big(\ker(j))\ge M+1\):} Considering in Equation \ref{eqn:defn_of_L_M} the added summands \(\alpha_pT_p\) (with \(b(p)=M\)), we see 
\[\langle T_p(1),e_j\rangle=0\]
because there are at least \(M+1\) different entries in the multi-index \(j\).
We conclude that \(L_M\) fulfils the properties assumed on \(L_{M+1}\):
\[b(\ker(j)\big)\ge M+1\quad\Rightarrow\quad\langle L_{M}(1),e_j\rangle=\langle T_q(1),e_j\rangle\]
\newline
\textbf{Case 2: \(b\big(\ker(j))= M\):} 
If \(\ker(j)\not\succeq q\), then both \(\langle L_M(1),e_j\rangle\) and \(\langle T_q(1),e_j\rangle\) are zero.
The arguments are the same as in the base case.
\newline
If \(\ker(j)\succeq q\), then \(\ker(j)\) coincides with one partition \(\hat{p}\in\{p\;|\;p\succeq q\,,\,b(p)=M\}\) and as in the base case one proves for \(p\) with \(b(p)=M\):
\[\langle T_p(1),e_j\rangle=
\begin{cases}
1&, p=\hat{p}\\
0&, p\neq \hat{p}
\end{cases}\]
Hence, by definition of \(\alpha_p\), it holds
\[\langle L_M(1),e_j\rangle=\langle L_{M+1}(1),e_j\rangle+\alpha_{\hat{p}}=1,\]
which is equal to \(\langle T_q(1),e_j\rangle=\langle e_j,e_j\rangle\), proving Implication \ref{eqn:compare_L_M_and_T_q} also in the case \(b\big(\ker(j)\big)= M\).
\end{proof}
\begin{rem}
While Proposition \ref{prop:linear_independence_if_blocks_do_not_exceed_N} obviously holds if we replace \(\mathcal{P}\) by any smaller category of partitions, the proof of Proposition \ref{prop:partitions_with_at_most_N_blocks_generate_everything} used the fact that, with every partition \(q\), the set \(\mathcal{P}(0,n)\) also contains all partitions \(p\succeq q\), i.e. any partition obtained from \(q\) by fusing different blocks.
\end{rem}
Using Proposition \ref{prop:linear_independence_if_blocks_do_not_exceed_N}, we can state the following result.
\begin{cor}\label{cor:different_categories_imply_different_easy_QGs_if_N_exceeds_number_of_points}
Consider two different categories of partitions \(\mathcal{C}_1\) and \(\mathcal{C}_2\).
Let \(M\in\N\) be the smallest integer such that \(\mathcal{C}_1(0,M)\neq \mathcal{C}_2(0,M)\).
Then \(G_N(\mathcal{C}_1) \neq G_N(\mathcal{C}_2)\) holds for all \(N\ge M\).
\newline
In particular, the sequences 
\[\big(G_N(\mathcal{C}_1)\big)_{N\in\N}\quad\quad\textnormal{and}\quad\quad\big(G_N(\mathcal{C}_1)\big)_{N\in\N}\]
differ.
\end{cor}
Note that \(M\) might not be the smallest value for \(N\) such that \(G_N(\mathcal{C}_1) \neq G_N(\mathcal{C}_2)\).
If we find  an \(n\in\N_0\) such that \(\mathcal{C}_1(0,n)\neq \mathcal{C}_2(0,n)\) and both \(\mathcal{C}_1(0,n)\) and \(\mathcal{C}_2(0,n)\) only contain partitions with at most \(M'\) blocks then we even have the result above for all \(N\ge M'\).
In particular, we can choose \(M'\) to be the maximal number of blocks of a partition \(p\) inside \(\mathcal{C}_1(0,M)\cup\mathcal{C}_2(0,M)\) (and this is at most \(M\)).
\section{Linear independence in the free case}
\label{sec:linear_independence_in_the_free_case}
For general categories of partitions we did not find in Section \ref{sec:linear_(in)dependencies_in_the_general_situation} a universal \(N\in\N\) such that different categories produce different easy quantum groups.
In the so-called \emph{free case}, where by definition only non-crossing partitions are considered, the situation is much more comfortable, at least if we assume \(N\ge 4\):
For each \(n\in\N_0\), the collection of maps \(\big(T_p\big)_{p\in\mathcal{NC}(0,n)}\) is linearly independent.
Hence, by Corollary \ref{cor:sufficient_condition_for_different_easy_quantum_groups}, every free easy quantum group with fundamental corepresentation matrix of size \(N\ge 4\) corresponds to a unique category of non-crossing partitions.
Conversely, for \(N\!\ge\!4\), two different non-crossing categories give rise to different free easy quantum groups.
\vspace{11pt}\newline
The structure of this section is as follows:
We first boil down the problem of linear independence to the question for a determinant of a special Gram matrix \(A(n,0)\).
The main part of this section is an adapted version of W. Tutte's work \cite{tutte}, where a formula for such a determinant is developed.
Afterwards, see Section \ref{subsec:conclusion:linear_independence_of_the_T_ps}, some easy observations show that the mentioned determinant is non-zero.
\newline
We recapitulate Tutte's work \cite{tutte} up to a recursion formula for the determinant \(\det\big(A(n,0)\big)\).
We do not change the principal ideas presented there, but we fixed errors in some of Tutte's arguments.
In further consequence, more definitions and partial results have been changed.
In addition, the graphical notations for partitions as introduced in the preliminaries are integrated in definitions and proofs as well as further explanations and more detailed arguments to justify partial results.
In the end, this section presents a self-contained proof of the linear independence described above, starting with the initial problem and guiding the reader without gaps through the relevant steps of the proof.
\subsection{Boiling down the problem to the invertiblity of a matrix}
\label{subsec:boiling_down_the_problem}
The question of injectivity of the construction \(\mathcal{C}\mapsto G_N(\mathcal{C})\) has been tracked down (in the sense of a sufficient condition) to the question of linear independence of the collections 
\begin{equation}\label{eqn:this_two-coloured_collection_should_be_l.i.}
\big(T_p\big)_{p\in\mathcal{C}(0,n)},
\end{equation}
see Corollary \ref{cor:sufficient_condition_for_different_easy_quantum_groups} and its proof.
For the rest of the chapter, we fix some  \(N\ge 4\) and we consider \(\mathcal{C}=\mathcal{NC}\), the category of all non-crossing partitions.
If the Collection \ref{eqn:this_uni-coloured_collection_should_be_l.i.} is linearly independent, so does every collection of the form \ref{eqn:this_two-coloured_collection_should_be_l.i.} with \(\mathcal{C}\subseteq\mathcal{NC}\) , proving the claim.
\vspace{11pt}\newline
As shown in the proof of Proposition \ref{prop:linear_independence_if_blocks_do_not_exceed_N}, every map \(T_p\) in Equation \ref{eqn:this_uni-coloured_collection_should_be_l.i.} is uniquely determined by the vector \(T_p(1)\in \big(\C^{N}\big)^{\otimes n}\).
Hence, to prove linear independence of the \(T_p\)'s, we can show linear independence of the vectors \(\big(T_p(1)\big)_{p\in\mathcal{NC}(0,n)}\).
In other words, we have to prove that the determinant of the Gram matrix
\begin{align}\label{eqn:Gram-matrix_A(n,0)}
\begin{split}
A(n,0):=\big(\langle T_p(1),T_q(1)\rangle\big)_{p,q\in \mathcal{NC}(0,n)}&=\big(\langle (T_q ^*T_p)(1),1\rangle\big)_{p,q\in \mathcal{NC}(0,n)}\\
&=\big((T_q ^*T_p)(1)\big)_{p,q\in \mathcal{NC}(0,n)}\\
&=\big(N^{\textnormal{rl}(q^*,p)}\big)_{p,q\in \mathcal{NC}(0,n)}
\end {split}
\end{align}
is non-zero.
Recall that \(\rl(q^*,p)\) are the numbers of remaining loops in the construction of \(q^*p\), see the preliminaries.
As both \(p\) and \(q\) have no upper points, we have \(q^*p\in\mathcal{NC}(0,0)\).
Hence, the number of remaining loops \(\textnormal{rl}(q^*,p)\) is the number of blocks after concatenating \(p\) and \(q^*\) vertically but before erasing all the blocks around the middle points, compare the preliminaries.
\vspace{11pt}\newline
For \(n\!=\!0\) we only have to consider the empty partition \(\emptyset\!\in\! \mathcal{NC}(0,0)\) and as \(T_{\emptyset}\!=\!\textnormal{id}_{\C}\!\neq\! 0\), the desired linear independence is given, so we can assume from now on that \(n\ge 1\).
Nonetheless we could check in all situations if our results also cover the special case \(n=0\).
\newline
Summing up these observations, our aim is to prove invertibility of the matrix \(A(n,0)\).
\subsection{Other proofs of the linear independence}
At this point, other works dealing with this (or similar) problems should be mentioned.
\newline
In \cite{banicaspeicherliberation}, where easy quantum groups are introduced, the linear independence of maps \(T_p\) for non-crossing partitions and for \(N\ge 4\) is already mentioned.
The authors refer to \cite{banica_symmetriesgenericcoaction} and \cite{banicacollins_integrationoverquantumpermutationgroups} where the basic idea of a proof is as follows:
\begin{itemize}
\item 
Using deep results from V. Jones's work on subfactors, \cite{jones_subfactors},  the dimension of \(\Hom(k,k)\) for \(k\in\N_0\) (and given \(N\!\ge\!4\)) is proved to be \(C_{2k}\), the \(2k\)-th Catalan number.
\item
It can easily be shown by induction that \(C_{2k}\) is also the number of non-crossing partitions on \(2k\) points, see for example \cite[Prop. 9.4]{nicaspeicher}.
\item
Combining both results, we have that the maps 
\[\big(T_p\big)_{p\in\mathcal{NC}(k,k)}\,,\]
that linearly generate by definition the intertwiner space \(\Hom(k,k)\), are linearly independent.
\item
By Frobenius reciprocity, see \cite[Prop. 3.1.11]{timmermann}, or simply by the fact that the rotation operator \(\textnormal{rot}_{\!\!\raisebox{2pt}{\rotatebox{270}{$^\curvearrowright$}}}\) is bijective, it also holds that the maps
\(\big(T_p\big)_{p\in\mathcal{NC}(0,2k)}\)
are linearly independent.
\item For \(\big(T_p\big)_{p\in\mathcal{NC}(0,2k+1)}\), the result follows from \(\Hom(0,2k+1)\otimes \id\subseteq \Hom(0,2k+2)\).
\end{itemize}

In \cite{kosmolinsky_recursionpaircase}, a matrix \(A_{\mathcal{NC}_2}(2n,0)\) as above is investigated, however it is in the case of \(\mathcal{NC}_2\), the category of all non-crossing pair-partitions and the investigation methods are different.
The main result there is a recursion formula for the determinant of this matrix and its zeros are identified.
Adapted to our purposes, it reads as follows:
\begin{thm*}[see {\cite[Cor. 2.10]{kosmolinsky_recursionpaircase}}] Let \((U_k(X))_{k\in\N}\) be the delated Chebyshev polynomials of the second kind from Definition \ref{defn:chebyshev_second_kind}.
Then it holds
\[\scalebox{0.9}{\(\displaystyle\det\big(A_{\mathcal{NC}_2}(2n,0)\big)=\prod_{i=1}^{\left\lfloor(n+1)/2\right\rfloor}\big(\det(A_{\mathcal{NC}_2}(n-i,0))\big)^{(-1)^{i}\binom{n-i+1}{i}}\prod_{i=1}^{\left\lfloor(n-1)/2\right\rfloor}\frac{U_{n-i}(N)}{U_i(N)}^{b(n-i,n-2i)}\)}\]
with
\[b(n,k)=\frac{k}{n}\binom{2n-k-1}{n-1}\]
\end{thm*}
As the Chebyshev polynomials have no roots greater or equal 2, this determinant is non-zero for \(N\ge 2\). 
A direct formula for this determinant is established in \cite{difrancesco_meandersandtemperleylieb}, where the central idea is to identify elements in \(\mathcal{NC}_2(n,n)\) with generating elements in a Temperley-Lieb algebra.
Using a suitable basis exchange matrix \(E\), the matrix \(A(n,0)\) is diagonalizable, i.e. it holds
\[A_{\mathcal{NC}_2}(2n,0)=EDE^*\]
for a suitable diagonal matrix \(D\).
Evidently, the structure of \(E\), and so its determinant, is encoded in the way the two bases are linked to each other.
Combining the results for both the determinants of \(D\) and \(E\), one obtains the following result.
\begin{thm*}[see {\cite[Eqn. 5.6]{difrancesco_meandersandtemperleylieb}}]
Let \((U_k(X))_{k\in\N}\) be the delated Chebyshev polynomials of the second kind from Definition \ref{defn:chebyshev_second_kind}. Then it holds
\[\det\big(A_{\mathcal{NC}_2}(2n,0)\big)=\prod_{i=1}^{n}U_i(N)^{a_ {n,i}}\]
with
\[a_{n,i}=\binom{2n}{n-i}-2\binom{2n}{n-i-1}+\binom{2n}{n-i-2}.\]
\end{thm*}
See also \cite{banicacurran_grammatrices}, where the same formula is proved by other means but not without using results and arguments from other works.
In \cite{banicacurran_grammatrices}, also a formula for the determinant of \(A(n,0)\), i.e. the case of \(\mathcal{NC}\), is proved and the formula follows from the result for \(\mathcal{NC}_2\).
In this virtue, the results for \(A_{\mathcal{NC}_2}(n,0)\) really contribute to the question in this article.
\newline
In contrast to that, Tutte's determinant formula is proved by elementary (combinatorial) arguments and these do not rely on results or theories from other sources.
\subsection{Main result and idea of its proof}\label{subsec:main_result_and_idea_of_its_proof}
We state again the initial problem of Section \ref{sec:linear_independence_in_the_free_case} in form of the following theorem to be proved:
\begin{thm}\label{thm:main_result:linear_independence_of_T_ps}
Consider \(N\ge 4\) and for a given partition  \(q\in\mathcal{P}(0,n)\) the map 
\[T_q:\big(\C^N\big)^{\otimes 0}=\C\rightarrow \big(\C^N\big)^{\otimes n},\]
as defined in Definition \ref{defn:T_p}.
Then the collection \(\big(T_p\big)_{p\in\mathcal{NC}(0,n)}\) is linearly independent. 
\end{thm}
By the previous consideration this answers the aforementioned question of this chapter in the non-crossing case:
The map \(\Psi:\mathcal{C}\mapsto R_N(\mathcal{C})\) from Equation \ref{eqn:from_partitions_to_easy_quantum_groups} is injective for non-crossing categories \(\mathcal{C}\) of partitions of sets and \(N\ge 4\).
\newline
As displayed in Section \ref{subsec:boiling_down_the_problem}, we prove this by showing the determinant of the Gram matrix \(A(n,0)\), see Equation \ref{eqn:Gram-matrix_A(n,0)}, to be non-zero.
\newline
In the following we will also need to consider matrices \(A(n,r)\) for \(1\le r<n\), compare Definition \ref{defn:A(n,r)}.
These matrices are obtained from \(A(n,0)\) by deleting certain rows and columns and by putting certain entries to zero.
\newline
Recall that the rows and columns of \(A(n,0)\) are indexed by the elements in \(\mathcal{NC}(0,n)\), the non-crossing partitions on \(n\) lower points.
The first step towards the main result is to divide \(\mathcal{NC}(0,n)\) into disjoint subsets, compare Definitions \ref{defn:Y(n,r)} and \ref{defn:W(n,r)}:
\[\mathcal{NC}(0,n)=Y(n,0)\,\dot{\cup}\, Y(n,1)\,\dot{\cup}\,\ldots\,\dot{\cup}\, Y(n,n-1)\]
We arrange the rows and columns of \(A(n,0)\) simultaneously such that the first \(\#Y(n,0)\) rows and columns are indexed by the elements in \(Y(n,0)\subset\mathcal{NC}(0,n)\):\vspace{1cm}
\[
\begin{array}{r}
\begin{matrix}
{\scriptstyle Y(n,0)}&\raisebox{0.5cm}{\rotatebox{270}{$\underbrace{\rule{0.6cm}{0cm}}$}}\\[10pt]
{\scriptstyle Y(n,1)}&\raisebox{0.5cm}{\rotatebox{270}{$\underbrace{\rule{0.6cm}{0cm}}$}}\\[0pt]
&\vdots\\[2pt]
{\scriptstyle Y(n,n-1)}&\raisebox{0.5cm}{\rotatebox{270}{$\underbrace{\rule{0.6cm}{0cm}}$}}
\end{matrix}
\begin{pmatrix}
&&&&\\[-6pt]
B(n,0)&\quad*\quad&\quad\cdots\;&\quad*\quad\\[15pt]
*&\quad*\quad&&\\[1pt]
\vdots&&\quad\ddots&\\[5pt]
*&&&\quad*\quad\\[-6pt]
&&&&
\end{pmatrix}=A(n,0)
\\
\underbrace{\rule{1.2cm}{0cm}}\rule{0.35cm}{0cm}\underbrace{\rule{0.8cm}{0cm}}\;\;\raisebox{-0.2cm}{$\;\cdots$}\;\;\rule{0.1cm}{0cm}\underbrace{\rule{0.8cm}{0cm}}\textnormal{\hspace{2.19cm}}\\[0pt]
\scriptstyle Y(n,0)\rule{0.6cm}{0cm} Y(n,1)\rule{1.1cm}{0cm} Y(n,n-1)\rule{2cm}{0cm}
\end{array}
\]

It turns out that multiplying the columns with suitable factors and adding the first \(\#Y(n,0)\) columns to the remaining ones in a suitable way puts the first block-row apart from \(B(n,0)\) to zero, compare Section \ref{subsec:recursion_formula}:
\[
\begin{pmatrix}
&&&&\\[-6pt]
B(n,0)&0&\cdots&0\\[10pt]
*&&&\\[6pt]
\vdots&&\alpha \tilde{M}&\\[6pt]
*&&&\\[-6pt]
&&&&
\end{pmatrix}.
\]
Here the factor \(\alpha\) is proved to be non-zero.
Having traced back the problem to the matrices \(B(n,0)\) and \(\tilde{M}\) would not be of any usage if we could not control these two submatrices.
Fortunately, we are able to show that both \(B(n,0)\) and \(\tilde{M}\) have structures similar to \(A(n,0)\):
They are given by the aforementioned \(A(n,r)\).
We can repeat the procedure for these submatrices and by induction we finally prove the determinant of \(A(n,0)\) to be non-zero.
More precisely, we prove the following, compare Section \ref{subsec:recursion_formula}.
\begin{thm*}
Consider the Gram marix \(A(n,0)\) as defined in Equation \ref{eqn:Gram-matrix_A(n,0)} as well as the matrices \(A(n,r)\) and \(B(n,r)\), see Definition \ref{defn:A(n,r)} and \ref{defn:B(n,r)}, respectively. 
Let further \((\beta_n)_{n\in\N_0}\) be the reversed Beraha polynomials from Definition \ref{defn:beraha_polynomials}.
Then the following holds.
\begin{itemize}
\item[(1)] For \(r\!\in\!\N\) and \(N\!\ge\! 4\), the number \(\beta_{r}(1/N)\) is non-zero.
\item[(2)]
For \(0\!\le\! r\!<\!n\!-\!1\) we have
\[
\det(A(n,r))=\alpha_r\cdot\det\begin{pmatrix}B(n,r)&0\\[6pt]D&A(n,r+1)\end{pmatrix}\]
where 
\[\alpha_r=\left(\frac{\beta_{r+3}(1/N)}{\beta_{r+2}(1/N)}\right)^c\neq 0\]
with a suitable positive integer \(c\).

\item[(3)] For \(0\!\le\! r\!<\!n\) we have
\[
B(n,r)=
\begin{cases}
A(n-1,r-1)&,\;r=2s+1.\\
N\cdot A(n-1,r-1)&,\;r=2s\textnormal{ and }r>0\\
N\cdot A(n-1,0)&,\;r=0.
\end{cases}
\]
\item[(4)]
For \(n\in\N\) we have
\[A(n,n-1)=N^{\left\lceil\frac{n}{2}\right\rceil}\in M_1(\C)=\C\]
\end{itemize}
\end{thm*}
A combination of these four results shows that the determinant of the Gram matrix \(A(n,0)\) is a product of non-zero quotients \(\alpha_r\) and a suitable power of \(N\), proving Theorem \ref{thm:main_result:linear_independence_of_T_ps}.
\subsection{Definitions and preparatory results}\label{subsec:definitions}
Sections \ref{subsec:definitions} and \ref{subsec:recursion_formula} are based on Tutte's article \cite{tutte}.
See also Section \ref{subsec:a_comparison_to_tutte}, where we collect the structure of Tutte's arguments and ideas as well as the differences to this article.
\subsubsection{Reversed Beraha polynomials and Chebyshev polynomials}
At the end of this chapter we deal with the so-called \emph{reversed Beraha polynomials}.
We define them here and show that they are closely related to the dilated Chebyshev polynomials of the second kind.
The definition of reversed Beraha polynomials can be found in \cite{tutte}.
Chebyshev polynomials and their properties presented here are standard knowledge, see for example \cite{rivlin_chebychevpolynomials}.
\begin{defn}\label{defn:beraha_polynomials}
The series of reversed Beraha poynomials \(\big(\beta_n(X)\big)_{n\in\N_0}\) is defined by the recursion formula
\[\beta_0(X)=0\quad,\quad \beta_1(X)=1\]
\[\beta_{n+1}(X)=\beta(n,X)-X\beta_{n-1}(X)\quad,\forall n\ge 1.\]
\end{defn}
\begin{defn}\label{defn:chebyshev_second_kind}
The series of dilated Chebyshev polynomials of the second kind \(\big(U_n(X)\big)_{n\in\N_0}\) is defined by the recursion formula
\[U_0(X)=1\quad,\quad U_1(X)=X\]
\[U_{n+1}(X)=XU_n(X)-U_{n-1}(X)\quad,\forall n\ge 1.\]
\end{defn}
\begin{defn}\label{defn:undilated_chebyshevs}
The (undilated) Chebyshev polynomial of the second kind \(\mathscr{U}_n(X)\) of degree \(n\in\N_0\) is defined as the unique polynomial of degree \(n\) that fulfils 
\begin{equation}\label{eqn:defn_of_undilated_chebychev}
\mathscr{U}_n\big(\cos(t)\big)=\frac{\sin\big((n+1)t\big)}{\sin(t)}\quad, \forall t\in(0,\pi).
\end{equation}
\end{defn}
Uniqueness follows from the fact that a polynomial (of degree \(n\)) that satisfies Equation \ref{eqn:defn_of_undilated_chebychev} must have the (simple) roots \(\frac{\pi}{n+1},\frac{2\pi}{n+1},\ldots,\frac{n\pi}{n+1}\).
Existence is easily proved by induction:
The polynomials defined by the following recursion formula fulfil Equation \ref{eqn:defn_of_undilated_chebychev}.
\[\mathscr{U}_0(X)=1\quad,\quad \mathscr{U}_1(X)=2X\]
\[\mathscr{U}_{n+1}(X)=X\mathscr{U}_n(X)-\mathscr{U}_{n-1}(X)\quad,\forall n\ge 1.\]
\newline
With a recursion formula for both series of Chebyshev polynomials at hand, one proves without any effort
\[\mathscr{U}_n(X)=U_n(\frac{1}{2}X).\]
In particular, we can locate the roots of the dilated Chebyshev polynomials.
\begin{lem}\label{lem:roots_of_chebychev}
For every \(n\in\N_0\), the roots of the dilated Chebyshev polynomial of the second kind \(U_n(X)\) are in the open interval \((-2,2)\).
\end{lem}
The reversed Beraha polynomials are closely related to the dilated Chebyshev polynomials of the second kind, compare for example \cite[p. 454]{copelandetall_berahachebychev}.
\begin{lem}\label{lem:relation_between_q_j_and_U_j-1}
Let \(N,j\in\N\) and write \(\frac{1}{N}=:z\).
Then we have the following relations between the reversed Beraha polynomials and the dilated Chebyshev polynomials of the second kind:
\begin{itemize}
\item[(i)] If \(j=2s\) is even, then it holds \[\displaystyle N^{s}\beta_j(z)=\sqrt{N}U_{j-1}(\sqrt{N}).\]
\item[(ii)] If \(j=2s\!+\!1\) is odd, then it holds \[\displaystyle N^{s}\beta_j(z)=U_{j-1}(\sqrt{N}).\]
\end{itemize}
\end{lem}
\begin{proof}
We use induction on \(j\in\N\).
To keep the notations short, we write \(\beta_i:=\beta_i(z)\) and \(U_i:=U_i(\sqrt{N})\).
For \(j\!=\!1\) we have \(\beta_1\!=\!1\!=\!U_0\), so the statement is true in the base case.
Assume the lemma to be true for all \(j'\in\N\) smaller than some \(j\in\N+1\).
In case (i), i.e. an even \(j\), we compute with the help of the induction assumption
\begin{align*}
N^s\beta_j=N^s\left(\beta_{j-1}-\frac{1}{N}\beta_{j-2}\right)&=N\cdot N^{s-1}\beta_{j-1}-N^{s-1}\beta_{j-2}\\
&=NU_{j-2}-\sqrt{N}U_{j-3}\\
&=\sqrt{N}\left(\sqrt{N}U_{j-2}-U_{j-3}\right)\\
&=\sqrt{N}U_{j-1}.
\end{align*}
Analogously, we find in the case of an odd \(j\)
\begin{align*}
N^s\beta_j=N^s\left(\beta_{j-1}-\frac{1}{N}\beta_{j-2}\right)&=N^s\beta_{j-1}-N^{s-1}\beta_{j-2}\\
&=\sqrt{N}U_{j-2}-U_{j-3}\\
&=U_{j-1}.
\end{align*}
\end{proof}
Combining Lemmata \ref{lem:roots_of_chebychev} and \ref{lem:relation_between_q_j_and_U_j-1} gives us the result that is used later on in order to prove Theorem \ref{thm:main_result:linear_independence_of_T_ps}.
\begin{lem}\label{lem:roots_of_beraha}
Writing \(z:=\frac{1}{N}\) it holds
\[\beta_n(z)\neq 0\]
for all \(n\in\N\) and \(N\in\N_{\ge4}\).
\end{lem} 
\subsubsection{The sets of partitions \(W(n,r)\) and \(Y(n,r)\) }
Related to \(n\!\in\!\N\), the number of (lower) points in the partitions \(\mathcal{NC}(0,n)\), we fix for all the remaining part of the chapter a few other numbers:
\newline
Let \(r\) be  a number with \(0\le r<n\) and let \(s\in\N_0\) be such that either \(2s=r\) or \(2s+1=r\). We start by defining subsets \(W(n,r)\subseteq \mathcal{NC}(0,n)\).
\begin{defn}\label{defn:W(n,r)}
Let \(0\le r<n\) and \(s\in\N_0\) such that \(r=2s\) or \(r=2s+1\). We define \(W(n,r)\subseteq \mathcal{NC}(0,n)\) by the following properties:
\begin{itemize}
\item[(i)] If \(r=2s\), then the \(s\) leftmost points of a partition in \(W(n,r)\) are no singletons and the \(s+1\) leftmost points belong to pairwise different blocks.
\item[(ii)] If \(r=2s+1\), then the \(s+1\) leftmost points of a partition in \(W(n,r)\) are no singletons and belong to pairwise different blocks.
\end{itemize}
We further define \(W(n,n):=\emptyset\).
\end{defn}
Note that for \(r=n\) the conditions (i) and (ii) from above cannot be fulfilled, so we can extend Definition \ref{defn:W(n,r)} by \(W(n,n):=\emptyset\subseteq\mathcal{NC}(0,n)\).
\newline
For \(r=0\) there is no condition at all, so \(W(n,0)=\mathcal{NC}(0,n)\).
Obviously, it holds \(W(n,n)\subseteq W(n,n-1)\subseteq\ldots\subseteq W(n,1)\subseteq W(n,0)\).
\newline
If \(r=2s\!+\!1\), the basic structure of an element \(p\in W(n,r)\) is the following:
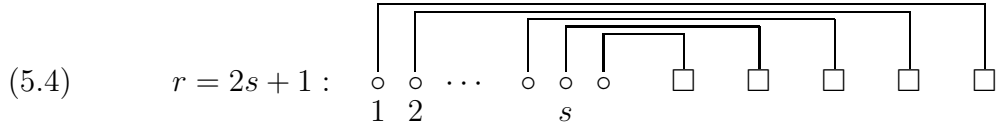
\begin{equation}\label{eqn:structure_in_W(n,r=2s+1)}
r=2s+1:\quad
\setlength{\unitlength}{0.5cm}
\begin{picture}(16,2.4)
\put(0,0) {$\circ$}
\put(1,0) {$\circ$}
\put(2,0) {$\cdots$}
\put(4,0) {$\circ$}
\put(5,0) {$\circ$}
\put(6,0) {$\circ$}
\put(8,0) {$\square$}
\put(10,0) {$\square$}
\put(12,0) {$\square$}
\put(14,0) {$\square$}
\put(16,0) {$\square$}
\put(0,-0.8) {$1$}
\put(1,-0.8) {$2$}
\put(5,-0.8) {$s$}

\put(0.2,0.5){\line(0,1){1.8}}
\put(1.2,0.5){\line(0,1){1.6}}
\put(4.2,0.5){\line(0,1){1.4}}
\put(5.2,0.5){\line(0,1){1.2}}
\put(6.2,0.5){\line(0,1){1}}

\put(8.35,0.6){\line(0,1){0.9}}
\put(10.35,0.6){\line(0,1){1.1}}
\put(12.35,0.6){\line(0,1){1.3}}
\put(14.35,0.6){\line(0,1){1.5}}
\put(016.35,0.6){\line(0,1){1.7}}

\put(6.2,1.5){\line(1,0){2.15}}
\put(5.2,1.7){\line(1,0){5.15}}
\put(4.2,1.9){\line(1,0){8.15}}
\put(1.2,2.1){\line(1,0){13.15}}
\put(0.2,2.3){\line(1,0){16.15}}

\end{picture}
\vspace{8pt}
\end{equation}
Recall from the preliminaries that the \(\square\)'s symbolize arbitrary non-empty substructures that we cannot or at least do not specify.
For example, the rightmost square might actually be given by
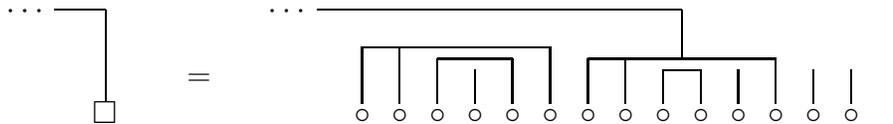
\[
\setlength{\unitlength}{0.5cm}
\raisebox{-0.5cm}{
\begin{picture}(2,3)
\put(0,0) {$\square$}
\put(0.35,0.6){\line(0,1){2.4}}
\put(-1,3){\line(1,0){1.35}}
\put(-2.3,2.78){$\cdots$}
\end{picture}
}
=\quad\quad\quad\quad
\raisebox{-0.5cm}{
\begin{picture}(14,4.3)
\put(0,0) {$\circ$}
\put(1,0) {$\circ$}
\put(2,0) {$\circ$}
\put(3,0) {$\circ$}
\put(4,0) {$\circ$}
\put(5,0) {$\circ$}
\put(6,0) {$\circ$}
\put(7,0) {$\circ$}
\put(8,0) {$\circ$}
\put(9,0) {$\circ$}
\put(10,0) {$\circ$}
\put(11,0) {$\circ$}
\put(12,0) {$\circ$}
\put(13,0) {$\circ$}

\put(0.2,0.5){\line(0,1){1.5}}
\put(1.2,0.5){\line(0,1){1.5}}
\put(2.2,0.5){\line(0,1){1.2}}
\put(3.2,0.5){\line(0,1){0.9}}
\put(4.2,0.5){\line(0,1){1.2}}
\put(5.2,0.5){\line(0,1){1.5}}
\put(6.2,0.5){\line(0,1){1.2}}
\put(7.2,0.5){\line(0,1){1.2}}
\put(8.2,0.5){\line(0,1){0.9}}
\put(9.2,0.5){\line(0,1){0.9}}
\put(10.2,0.5){\line(0,1){0.9}}
\put(11.2,0.5){\line(0,1){1.2}}
\put(12.2,0.5){\line(0,1){0.9}}
\put(13.2,0.5){\line(0,1){0.9}}

\put(0.2,2){\line(1,0){5}}
\put(2.2,1.7){\line(1,0){2}}
\put(6.2,1.7){\line(1,0){5}}
\put(8.2,1.4){\line(1,0){1}}

\put(8.7,1.7){\line(0,1){1.3}}
\put(-1,3){\line(1,0){9.7}}
\put(-2.3,2.78){$\cdots$}
\end{picture}
}.
\]
We only know that one of the blocks inside a square is connected to exactly one of the \(s+1\) leftmost points.
By non-crossingness and the properties of \(p\in W(n,r)\), this point amongst the \(s\!+\!1\) leftmost points is uniquely determined and it is of course the one indicated in picture \ref{eqn:structure_in_W(n,r=2s+1)}.
Note that there is some ambiguity in the definition of the substructures \(\square\) as we might have different possibilities where one square ends and the next square begins.
\newline
If \(r\!=\!2s\), the only difference to the structure above is the possibility that the point \(s\!+\!1\) is allowed to be a singleton:

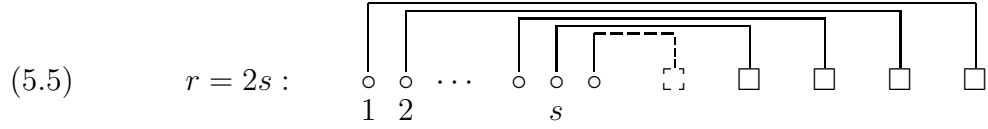
\begin{equation}\label{eqn:structure_in_W(n,r=2s)}
r=2s:\quad\quad
\setlength{\unitlength}{0.5cm}
\begin{picture}(16,2.4)
\put(0,0) {$\circ$}
\put(1,0) {$\circ$}
\put(2,0) {$\cdots$}
\put(4,0) {$\circ$}
\put(5,0) {$\circ$}
\put(6,0) {$\circ$}
\put(8,0) {$\square$}
\put(7.95,0.3) {\linethickness{0.1cm}\color{white}\line(1,0){0.7}}
\put(8.35,-0.05) {\linethickness{0.1cm}\color{white}\line(0,1){0.7}}
\put(10,0) {$\square$}
\put(12,0) {$\square$}
\put(14,0) {$\square$}
\put(16,0) {$\square$}
\put(0,-0.8) {$1$}
\put(1,-0.8) {$2$}
\put(5,-0.8) {$s$}
\put(0.2,0.5){\line(0,1){1.8}}
\put(1.2,0.5){\line(0,1){1.6}}
\put(4.2,0.5){\line(0,1){1.4}}
\put(5.2,0.5){\line(0,1){1.2}}
\put(6.2,0.5){\line(0,1){1}}
\multiput(8.35,0.6)(0,0.3){3}{\line(0,1){0.2}}
\put(10.35,0.6){\line(0,1){1.1}}
\put(12.35,0.6){\line(0,1){1.3}}
\put(14.35,0.6){\line(0,1){1.5}}
\put(016.35,0.6){\line(0,1){1.7}}
\multiput(6.2,1.5)(0.45,0){5}{\line(1,0){0.3}}
\put(5.2,1.7){\line(1,0){5.15}}
\put(4.2,1.9){\line(1,0){8.15}}
\put(1.2,2.1){\line(1,0){13.15}}
\put(0.2,2.3){\line(1,0){16.15}}
\end{picture}
\vspace{8pt}
\end{equation}
Note that the dashed structure might be empty, which is the case of \(s\!+\!1\) being a singleton.
\vspace{11pt}\newline
It is easy to check that the cardinality of \(W(n,n-1)\) is one:
\begin{lem}\label{lem:W(n,n-1)_contains_only_one_element}
For every \(n\!\in\!\N\) and \(r\!=\!n-1\) the set \(W(n,r)\) contains only the following element (depending on the parity of \(r\)):
\begin{align}
r=2s+1&:\quad\label{eqn:the_element_in_W(n,n-1=2s+1)}
\setlength{\unitlength}{0.5cm}
\begin{picture}(16,2.4)
\put(0,0) {$\circ$}
\put(1,0) {$\circ$}
\put(2,0) {$\cdots$}
\put(4,0) {$\circ$}
\put(5,0) {$\circ$}
\put(6,0) {$\circ$}
\put(7,0) {$\circ$}
\put(8,0) {$\circ$}
\put(9,0) {$\circ$}
\put(10,0) {$\cdots$}
\put(12,0) {$\circ$}
\put(13,0) {$\circ$}
\put(0,-0.8) {$1$}
\put(1,-0.8) {$2$}
\put(5,-0.8) {$s$}
\put(0.2,0.5){\line(0,1){1.8}}
\put(1.2,0.5){\line(0,1){1.6}}
\put(4.2,0.5){\line(0,1){1.4}}
\put(5.2,0.5){\line(0,1){1.2}}
\put(6.2,0.5){\line(0,1){1}}
\put(7.2,0.5){\line(0,1){1}}
\put(8.2,0.5){\line(0,1){1.2}}
\put(9.2,0.5){\line(0,1){1.4}}
\put(12.2,0.5){\line(0,1){1.6}}
\put(13.2,0.5){\line(0,1){1.8}}
\put(6.2,1.5){\line(1,0){1}}
\put(5.2,1.7){\line(1,0){3}}
\put(4.2,1.9){\line(1,0){5}}
\put(1.2,2.1){\line(1,0){11}}
\put(0.2,2.3){\line(1,0){13}}
\end{picture}
\\[19pt]
r=2s&:\quad\label{eqn:the_element_in_W(n,n-1=2s)}
\setlength{\unitlength}{0.5cm}
\begin{picture}(16,2.4)
\put(0,0) {$\circ$}
\put(1,0) {$\circ$}
\put(2,0) {$\cdots$}
\put(4,0) {$\circ$}
\put(5,0) {$\circ$}
\put(6,0) {$\circ$}
\put(7,0) {$\circ$}
\put(8,0) {$\circ$}
\put(9,0) {$\cdots$}
\put(11,0) {$\circ$}
\put(12,0) {$\circ$}
\put(0,-0.8) {$1$}
\put(1,-0.8) {$2$}
\put(5,-0.8) {$s$}
\put(0.2,0.5){\line(0,1){1.8}}
\put(1.2,0.5){\line(0,1){1.6}}
\put(4.2,0.5){\line(0,1){1.4}}
\put(5.2,0.5){\line(0,1){1.2}}
\put(6.2,0.5){\line(0,1){1}}
\put(7.2,0.5){\line(0,1){1.2}}
\put(8.2,0.5){\line(0,1){1.4}}
\put(11.2,0.5){\line(0,1){1.6}}
\put(12.2,0.5){\line(0,1){1.8}}
\put(5.2,1.7){\line(1,0){2}}
\put(4.2,1.9){\line(1,0){4}}
\put(1.2,2.1){\line(1,0){10}}
\put(0.2,2.3){\line(1,0){12}}
\end{picture}
\end{align}\vspace{8pt}
\end{lem}
\begin{defn}\label{defn:Y(n,r)}
For \(0\le r<n\) we define
\begin{equation}\label{eqn:def_of_Y(n,r)}
Y(n,r):=W(n,r)\backslash W(n,r+1).
\end{equation}
\end{defn}
Comparing the definitions and illustrating pictures for partitions in \(W(n,r)\) and \(W(n,r\!+\!1)\), we can easily describe the partitions inside the sets \(Y(n,r)\):
\begin{lem}\label{lem:elements_in_Y(n,r)}
Let \(0\le r<n\).
\begin{itemize}
\item[(i)] If \(r\!=\!2s\), then \(Y(n,r)\) contains all elements of \(W(n,r)\) such that \(s\!+\!1\) is a singleton:
 \begin{equation}\label{eqn:example_in_Y(n,2s)}
\setlength{\unitlength}{0.5cm}
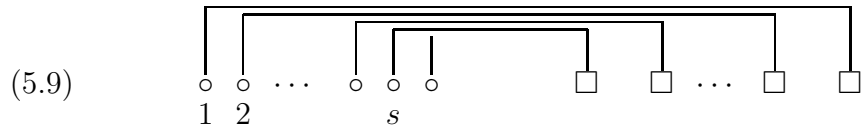
\begin{picture}(20,2.4)
\put(0,0) {$\circ$}
\put(1,0) {$\circ$}
\put(2,0) {$\cdots$}
\put(4,0) {$\circ$}
\put(5,0) {$\circ$}
\put(6,0) {$\circ$}
\put(10,0) {$\square$}
\put(12,0) {$\square$}
\put(13,0) {$\;\cdots$}
\put(15,0) {$\square$}
\put(17,0) {$\square$}
\put(0,-0.8) {$1$}
\put(1,-0.8) {$2$}
\put(5,-0.8) {$s$}
\put(0.2,0.5){\line(0,1){1.8}}
\put(1.2,0.5){\line(0,1){1.6}}
\put(4.2,0.5){\line(0,1){1.4}}
\put(5.2,0.5){\line(0,1){1.2}}
\put(6.2,0.5){\line(0,1){1}}
\put(10.35,0.6){\line(0,1){1.1}}
\put(12.35,0.6){\line(0,1){1.3}}
\put(15.35,0.6){\line(0,1){1.5}}
\put(17.35,0.6){\line(0,1){1.7}}
\put(5.2,1.7){\line(1,0){5.15}}
\put(4.2,1.9){\line(1,0){8.15}}
\put(1.2,2.1){\line(1,0){14.15}}
\put(0.2,2.3){\line(1,0){17.15}}
\end{picture}
\end{equation}
\vspace{6pt}
\item[(ii)] If \(r\!=\!2s\!+\!1\), then \(Y(n,r)\) contains all elements of \(W(n,r)\) such that \(s\!+\!1\) and \(s\!+\!2\) are connected:
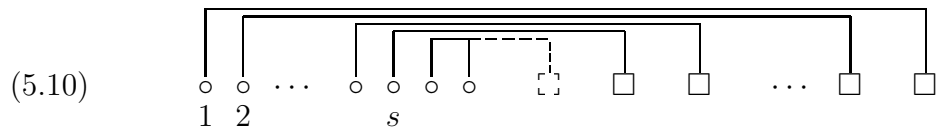
\begin{equation}\label{eqn:example_in_Y(n,2s+1)}
\setlength{\unitlength}{0.5cm}
\begin{picture}(20,2.4)
\put(0,0) {$\circ$}
\put(1,0) {$\circ$}
\put(2,0) {$\cdots$}
\put(4,0) {$\circ$}
\put(5,0) {$\circ$}
\put(6,0) {$\circ$}
\put(7,0) {$\circ$}
\put(9,0) {$\square$}
\put(8.95,0.3) {\linethickness{0.1cm}\color{white}\line(1,0){0.7}}
\put(9.35,-0.05) {\linethickness{0.1cm}\color{white}\line(0,1){0.7}}
\put(11,0) {$\square$}
\put(13,0) {$\square$}
\put(15,0) {$\;\cdots$}
\put(17,0) {$\square$}
\put(19,0) {$\square$}
\put(0,-0.8) {$1$}
\put(1,-0.8) {$2$}
\put(5,-0.8) {$s$}
\put(0.2,0.5){\line(0,1){1.8}}
\put(1.2,0.5){\line(0,1){1.6}}
\put(4.2,0.5){\line(0,1){1.4}}
\put(5.2,0.5){\line(0,1){1.2}}
\put(6.2,0.5){\line(0,1){1}}
\put(7.2,0.5){\line(0,1){1}}
\multiput(9.35,0.6)(0,0.3){3}{\line(0,1){0.2}}
\put(11.35,0.6){\line(0,1){1.1}}
\put(13.35,0.6){\line(0,1){1.3}}
\put(17.35,0.6){\line(0,1){1.5}}
\put(19.35,0.6){\line(0,1){1.7}}
\multiput(7.2,1.5)(0.45,0){5}{\line(1,0){0.3}}
\put(6.2,1.5){\line(1,0){1}}
\put(5.2,1.7){\line(1,0){6.15}}
\put(4.2,1.9){\line(1,0){9.15}}
\put(1.2,2.1){\line(1,0){16.15}}
\put(0.2,2.3){\line(1,0){19.15}}
\end{picture}\vspace{6pt}
\end{equation}
\end{itemize}

Furthermore, we have \(Y(n,n-1)=W(n,n-1)\), as \(W(n,n)\) is empty.
\end{lem}

\subsubsection{The graphs \(G(p,q)\), \(H_r(p,q)\) and \(r\)-flaws}
\label{subsubsec:The graphs G(p,q), H_r(p,q) and r-flaws}
In the following it will be helpful to look at the vertical concatenation of \(p\) and \(q^*\) (determining \(\langle T_p(1),T_q(1)\rangle=N^{\textnormal{rl}(q^*,p)}\)) as an (undirected) graph with \(2n\) points:
\begin{defn}\label{defn:G(p,q)}
Let \(p,q\in\mathcal{NC}(0,n)\) be two partitions.
Considering \(p,q\) as (undirected) graphs, i.e. the vertices are given by the points and two vertices are adjacent if they are in the same block of the partition, we define the graph \(G(p,q)\) as the direct sum of \(p\) and \(q\) where we additionally add edges between the \(i\)-th point of \(p\) and the \(i\)-th point of \(q\). 
\end{defn}
The graph \(G(p,q)\) exactly describes the situation of the concatenated partitions \(p\) and \(q^*\) \emph{before} erasing the points in the middle and \emph{before} erasing remaining loops, compare the preliminaries.
The number \(\rl(q^*,p)\) is the number of components of \(G(p,q)\).
We usually label the points of \(p\) from left to right by \(1,2,\ldots,n\) and the ones of \(q\) by \(1',2',\ldots,n'\).
\vspace{11pt}
\newline
As an example, consider the partitions 
\[
p=
\setlength{\unitlength}{0.5cm}
\begin{picture}(5,2.4)
\put(0,-0.5) {$\circ$}
\put(1,-0.5) {$\circ$}
\put(2,-0.5) {$\circ$}
\put(3,-0.5) {$\circ$}
\put(4,-0.5) {$\circ$}
\put(0.2,0){\line(0,1){1.4}}
\put(1.2,0){\line(0,1){1.4}}
\put(2.2,0){\line(0,1){1}}
\put(3.2,0){\line(0,1){1}}
\put(4.2,0){\line(0,1){1.4}}
\put(0.2,1.4){\line(1,0){4}}
\put(2.2,1){\line(1,0){1}}
\end{picture}
\quad,\quad
q=
\setlength{\unitlength}{0.5cm}
\begin{picture}(5,2.4)
\put(0,-0.5) {$\circ$}
\put(1,-0.5) {$\circ$}
\put(2,-0.5) {$\circ$}
\put(3,-0.5) {$\circ$}
\put(4,-0.5) {$\circ$}
\put(0.2,0){\line(0,1){1.4}}
\put(1.2,0){\line(0,1){1}}
\put(2.2,0){\line(0,1){1.4}}
\put(3.2,0){\line(0,1){1}}
\put(4.2,0){\line(0,1){1}}
\put(0.2,1.4){\line(1,0){2}}
\put(3.2,1){\line(1,0){1}}
\end{picture}.
\]
Identifying \(G(p,q)\) with its illustrations, we write 
\begin{equation}\label{eqn:example_of_G(p,q)}
G(p,q)=
\setlength{\unitlength}{0.5cm}
\begin{picture}(5,2.4)
\put(0,0.5) {$\circ$}
\put(1,0.5) {$\circ$}
\put(2,0.5) {$\circ$}
\put(3,0.5) {$\circ$}
\put(4,0.5) {$\circ$}
\put(0.2,1){\line(0,1){1.4}}
\put(1.2,1){\line(0,1){1.4}}
\put(2.2,1){\line(0,1){1}}
\put(3.2,1){\line(0,1){1}}
\put(4.2,1){\line(0,1){1.4}}
\put(0.2,2.4){\line(1,0){4}}
\put(2.2,2){\line(1,0){1}}
\put(0.2,-0.4){\line(0,1){0.7}}
\put(1.2,-0.4){\line(0,1){0.7}}
\put(2.2,-0.4){\line(0,1){0.7}}
\put(3.2,-0.4){\line(0,1){0.7}}
\put(4.2,-0.4){\line(0,1){0.7}}
\put(0,-1) {$\circ$}
\put(1,-1) {$\circ$}
\put(2,-1) {$\circ$}
\put(3,-1) {$\circ$}
\put(4,-1) {$\circ$}
\put(0.2,-2.6){\line(0,1){1.4}}
\put(1.2,-2.2){\line(0,1){1}}
\put(2.2,-2.6){\line(0,1){1.4}}
\put(3.2,-2.2){\line(0,1){1}}
\put(4.2,-2.2){\line(0,1){1}}
\put(0.2,-2.6){\line(1,0){2}}
\put(3.2,-2.2){\line(1,0){1}}
\end{picture}\vspace{33pt}.
\end{equation}
\begin{rem}
Recall that two points in an (undirected) graph are adjacent if there is an edge between them and two points are connected if they are in the same component of the graph, i.e. there is a path starting at one of the points and ending at the other.
\newline
Throughout this chapter we are only interested in the connected components of a graphs and not the actual adjacencies of its points. Hence, in many cases we do not distinguish precisely if a line drawn in an illustrating picture is an actual edge in the graph or if it just symbolizes the property that the corresponding points are connected. 
\newline
In this virtue we can say that the lines drawn in a picture like \ref{eqn:example_of_G(p,q)} only define the connectivity properties of a graph but not their realizations via actual edges.
This way each such picture describes an equivalence class of graphs (on the same points), equivalent by the property to have the same components.
\newline
Although we can always refer to Definition \ref{defn:G(p,q)} for a precise definition of the graph \(G(p,q)\), it does not cause any problems in the following to interpret any drawn line as the equivalence relation ``is connected to''.
\end{rem}
\begin{rem}
Note that the illustration of \(G(p,q)\) for non-crossing \(p\) and \(q\) can be drawn in a non-crossing way as well.
\end{rem}
\begin{defn}\label{defn:H_r(p,q)}
Considering the situation and the graph \(G(p,q)\) from Definition \ref{defn:G(p,q)}, we define the graphs \(\big(H_r(p,q)\big)_{1\le r< n}\) by starting with \(G(p,q)\) and erasing the edges \((i,i')\) for \(1\le i\le s\!+\!1\).
\end{defn}
Considering \(p\) and \(q\) as above, we have for example the illustrations
\[
H_3(p,q)=
\setlength{\unitlength}{0.5cm}
\begin{picture}(5,2.4)
\put(0,0.5) {$\circ$}
\put(1,0.5) {$\circ$}
\put(2,0.5) {$\circ$}
\put(3,0.5) {$\circ$}
\put(4,0.5) {$\circ$}
\put(0.2,1){\line(0,1){1.4}}
\put(1.2,1){\line(0,1){1.4}}
\put(2.2,1){\line(0,1){1}}
\put(3.2,1){\line(0,1){1}}
\put(4.2,1){\line(0,1){1.4}}
\put(0.2,2.4){\line(1,0){4}}
\put(2.2,2){\line(1,0){1}}
\put(2.2,-0.4){\line(0,1){0.7}}
\put(3.2,-0.4){\line(0,1){0.7}}
\put(4.2,-0.4){\line(0,1){0.7}}
\put(0,-1) {$\circ$}
\put(1,-1) {$\circ$}
\put(2,-1) {$\circ$}
\put(3,-1) {$\circ$}
\put(4,-1) {$\circ$}
\put(0.2,-2.6){\line(0,1){1.4}}
\put(1.2,-2.2){\line(0,1){1}}
\put(2.2,-2.6){\line(0,1){1.4}}
\put(3.2,-2.2){\line(0,1){1}}
\put(4.2,-2.2){\line(0,1){1}}
\put(0.2,-2.6){\line(1,0){2}}
\put(3.2,-2.2){\line(1,0){1}}
\end{picture}
\quad\quad,\quad\quad
H_4(p,q)=
\setlength{\unitlength}{0.5cm}
\begin{picture}(5,2.4)
\put(0,0.5) {$\circ$}
\put(1,0.5) {$\circ$}
\put(2,0.5) {$\circ$}
\put(3,0.5) {$\circ$}
\put(4,0.5) {$\circ$}
\put(0.2,1){\line(0,1){1.4}}
\put(1.2,1){\line(0,1){1.4}}
\put(2.2,1){\line(0,1){1}}
\put(3.2,1){\line(0,1){1}}
\put(4.2,1){\line(0,1){1.4}}
\put(0.2,2.4){\line(1,0){4}}
\put(2.2,2){\line(1,0){1}}
\put(3.2,-0.4){\line(0,1){0.7}}
\put(4.2,-0.4){\line(0,1){0.7}}
\put(0,-1) {$\circ$}
\put(1,-1) {$\circ$}
\put(2,-1) {$\circ$}
\put(3,-1) {$\circ$}
\put(4,-1) {$\circ$}
\put(0.2,-2.6){\line(0,1){1.4}}
\put(1.2,-2.2){\line(0,1){1}}
\put(2.2,-2.6){\line(0,1){1.4}}
\put(3.2,-2.2){\line(0,1){1}}
\put(4.2,-2.2){\line(0,1){1}}
\put(0.2,-2.6){\line(1,0){2}}
\put(3.2,-2.2){\line(1,0){1}}
\end{picture}\vspace{33pt}.
\]
Note that in the case \(H_3(p,q)\) we have \(r=3\) and \(s=1\), therefore we removed the \(s+1=2\) edges \((1,1')\) and \((2,2')\).
Likewise we have in the case \(H_4(p,q)\) the values \(r=4\) and \(s=2\), so \(s+1=3\) edges are removed.
\vspace{11pt}
\newline
To define the matrices \(A(n,r)\) from Section \ref{subsec:main_result_and_idea_of_its_proof}, we need the notion of a so called \emph{\(r\)-flaw}. 

\begin{defn}\label{defn:r-flawless_graph}
Let \(p,q\in\mathcal{NC}(0,n)\) and let \(G(p,q)\) be the graph associated to this pair of partitions as defined in Definition \ref{defn:H_r(p,q)}.
We define \(G(p,q)\) to be \emph{\(r\)-flawless} if  the following properties are fulfilled:
\begin{itemize}
\item[(1)] In \(H_r(p,q)\), the points \(1,\ldots,(s+1)\) are disconnected.
\item[(2)] In \(H_r(p,q)\), the points \(1',\ldots,(s+1)'\) are disconnected.
\item[(3)] In \(H_r(p,q)\), the points \(i\) and \(i'\) are connected for all \(1\!\le\!i\!\le\!s\).
\item[(4)] If \(r\) is odd, then we also have  in \(H_r(p,q)\) a path between \((s+1)\) and \((s+1)'\).
\end{itemize}
In all other cases we call \(G(p,q)\) to have an \emph{\(r\)-flaw}.
\end{defn}
A point in the graph \(G(p,q)\) is called a witness for an \(r\)-flaw if existence or absence of paths in \(H_r(p,q)\) between this point and other points contradicts one of the properties (1) -- (4).
\newline
Informally, it is easy to describe what an \(r\)-flawless graph should be:
Consider the picture
\begin{align*}
G(p,q)=\quad&
\setlength{\unitlength}{0.5cm}
\begin{picture}(16,2.4)
\put(0,0) {$\circ$}
\put(1,0) {$\circ$}
\put(2,0) {$\cdots$}
\put(4,0) {$\circ$}
\put(5,0) {$\circ$}
\put(10,0) {$\square$}
\put(12,0) {$\square$}
\put(13,0) {$\;\cdots$}
\put(15,0) {$\square$}
\put(17,0) {$\square$}
\put(0.2,0.5){\line(0,1){1.8}}
\put(1.2,0.5){\line(0,1){1.6}}
\put(4.2,0.5){\line(0,1){1.4}}
\put(5.2,0.5){\line(0,1){1.2}}
\put(0.2,-0.9){\line(0,1){0.8}}
\put(1.2,-0.9){\line(0,1){0.8}}
\put(4.4,-0.4){\scalebox{0.8}{s}}
\put(4.2,-0.9){\line(0,1){0.8}}
\put(5.2,-0.9){\line(0,1){0.8}}
\put(10.3,-0.9){\line(0,1){0.8}}
\put(12.3,-0.9){\line(0,1){0.8}}
\put(15.3,-0.9){\line(0,1){0.8}}
\put(17.3,-0.9){\line(0,1){0.8}}
\put(10.35,0.6){\line(0,1){1.1}}
\put(12.35,0.6){\line(0,1){1.3}}
\put(15.35,0.6){\line(0,1){1.5}}
\put(17.35,0.6){\line(0,1){1.7}}
\put(5.2,1.7){\line(1,0){5.15}}
\put(4.2,1.9){\line(1,0){8.15}}
\put(1.2,2.1){\line(1,0){14.15}}
\put(0.2,2.3){\line(1,0){17.15}}
\end{picture}
\hspace{-8.15cm}
\raisebox{-0.5cm}{\reflectbox{\scalebox{-1}{
\begin{picture}(16,2.4)
\put(0,0) {$\circ$}
\put(1,0) {$\circ$}
\put(2,0) {$\cdots$}
\put(4,0) {$\circ$}
\put(5,0) {$\circ$}
\put(10,0) {$\square$}
\put(12,0) {$\square$}
\put(13,0) {$\;\cdots$}
\put(15,0) {$\square$}
\put(17,0) {$\square$}
\put(0.2,0.5){\line(0,1){1.8}}
\put(1.2,0.5){\line(0,1){1.6}}
\put(4.2,0.5){\line(0,1){1.4}}
\put(5.2,0.5){\line(0,1){1.2}}
\put(10.35,0.6){\line(0,1){1.1}}
\put(12.35,0.6){\line(0,1){1.3}}
\put(15.35,0.6){\line(0,1){1.5}}
\put(17.35,0.6){\line(0,1){1.7}}
\put(5.2,1.7){\line(1,0){5.15}}
\put(4.2,1.9){\line(1,0){8.15}}
\put(1.2,2.1){\line(1,0){14.15}}
\put(0.2,2.3){\line(1,0){17.15}}
\end{picture}
}}}\quad\quad\quad.
\end{align*}
Firstly, no point \(1\le i\le (s+1)\) is connected to any of the other points \(1\le j\le (s+1)\).
It is easy to see that this property is equivalent to items (1) and (2) above.
Secondly, we have a ``path'' between \(i\) and \(i'\) that does not use the line \((i,i')\), proving item (3).
Even for odd \(r\) the graph above is \(r\)-flawless, compare item (4), as we also have such a path between \((s+1)\) and \((s+1)'\).
\vspace{11pt}\newline
\newline
Note that \(0\)-flaws do not exist in the sense that conditions (1) -- (4) from Definition \ref{defn:r-flawless_graph} are always fulfilled if \(r\!=\!s\!=\!0\);
excluding graphs with \(0\)-flaws is just a way to consider all possible \(G(p,q)\).

\subsubsection{The matrices \(A(n,r)\) and \(B(n,r)\)}
\begin{defn}\label{defn:A(n,r)}
Let \(p,q\in\mathcal{NC}(0,n)\). We define the matrix \(A(n,r)\) by
\[A(n,r):=\Big(e_r(p,q)\Big)_{p,q\in W(n,r)}\]
where
\[e_r(p,q):=
\begin{cases}
0&\textnormal{, \(G(p,q)\) has an \(r\)-flaw}\\
N^{\rl(q^*,p)}& \textnormal{, else.}
\end{cases}\]
\end{defn}
\begin{rem}
\begin{itemize}
\item[(a)] Note that for the definition of the matrix \(A(n,r)\) and its entries \(e_r(p,q)\) the basis \(N\) above is just a parameter.
It could be replaced by (more or less)  every other complex number, but, in order to establish the connection to our main result, Theorem \ref{thm:main_result:linear_independence_of_T_ps}, we have to choose this parameter to be the natural number \(N\ge 4\) fixed at the beginning of Section \ref{sec:linear_independence_in_the_free_case}.
\item[(b)]
Note that \(G(p,q)\) would have an \(r\)-flaw if \(\{p,q\}\not\subseteq W(n,r)\).
If conversely \(\{p,q\}\subseteq W(n,r)\), then, by definition, the points \(1,\ldots,s+1\) are pairwise disconnected in \(p\) and \(1',\ldots,(s+1)'\) are pairwise disconnected in \(q\).
In general, it is not clear whether the same holds after constructing \(G(p,q)\), but it is guaranteed by \(r\)-flawlessness.
\end{itemize}
\end{rem}
\begin{defn}\label{defn:B(n,r)}
Consider the set of partitions \(Y(n,r)\) as defined in Equation \ref{eqn:def_of_Y(n,r)}.
We define the matrix \(B(n,r)\) to be the submatrix of \(A(n,r)\) obtained by reducing \(A(n,r)\) to the rows and columns indexed by elements in \(Y(n,r)\).
\end{defn}
We now prove a connection between the matrices \(B(n,r)\) and suitable matrices \(A(n-1,r')\).
\begin{lem}[odd case]\label{lem:connection:between_B(n,r)-and_A(n-1,r-1)_r=2s+1}
Let \(0\!<\! r\!=\!2s\!+\!1\!<\!n\!-\!1\). Then, modulo row and column permutations, the matrix \(B(n,r)\) is equal to the matrix \(A(n-1,r-1)\).
\end{lem}
\begin{proof}
Recall that the rows and columns of \(A(n\!-\!1,r\!-\!1)\) are labelled by elements in \(W(n\!-\!1,r\!-\!1)\), while those of \(B(n,r)\) are labelled by elements in \(Y(n,r)=W(n,r+1)\backslash W(n,r)\).
See also Definition \ref{defn:W(n,r)} and Lemma \ref{lem:elements_in_Y(n,r)}, where the elements in \(W(n,r)\) and \(Y(n,r)\), respectively, are characterised.
\newline
We show the desired equality in three steps.
\begin{itemize}
\item [(1)]
In Step 1 we establish a bijection \(\alpha\) between the sets \(Y(n,r)\) and \(W(n-1,r-1)\) by removing suitable points in the partitions \(p\in Y(n,r)\).
Doing so, we can assume the rows and columns of both matrices to be labelled equally.
\item [(2)]
In Step 2 we investigate how the replacement \(p\mapsto \alpha(p)\) affects the number of components \(\rl(p,q)\) in the graph \(G(p,q)\).
\item [(3)]
In Step 3 we finally check that \(A(n\!-\!1,r\!-\!1)\) and \(B(n,r)\) have the same zero-entries, i.e. we prove that the application of \(\alpha\) replaces every  Graph \(G(p,q)\) which has an \(r\)-flaw with a graph \(G(\alpha(p),\alpha(q))\) which has an \((r\!-\!1)\)-flaw.
\end{itemize}
\textbf{Step 1: \(Y(n,r)\simeq W(n\!-\!1,r\!-\!1)\): }
Consider \(p,q\in Y(n,r)\).
For odd \(r\) they have a structure as displayed in Picture \ref{eqn:example_in_Y(n,2s+1)}.
In particular, we have in \(p\) that \(s+1\) is connected to \(s+2\) and in \(q\) we have that \((s+1)'\) and \((s+2)'\) are connected.
Deleting \(s+2\) and \((s+2)'\) from \(p\) and \(q\), respectively, defines two elements \(\alpha(p),\alpha(q)\in W(n-1,r-1)\).
Evidently, the map \(\alpha:Y(n,r)\rightarrow W(n\!-\!1,r\!-\!1)\) is bijective: 
We can start with  any partition \(p'\in W(n\!-\!1,r\!-\!1)\), add a point after \(s+1\) and connect it to \(s+1\).
The resulting partition as a preimage of \(p'\) under \(\alpha\).
\newline 
By definition, rows and columns of \(B(n,r)\) and \(A(n\!-\!1,r\!-\!1)\) are labelled by elements in \(Y(n,r)\) and \(W(n\!-\!1,r\!-\!1)\), respectively.
Identifying \(Y(n,r)\) and \(W(n\!-\!1,r\!-\!1)\) via the above \(\alpha\), we can, after a rearrangement, assume that the rows and columns of \(B(n,r)\) and \(A(n\!-\!1,r\!-\!1)\) are labelled equally.
It remains to show that their entries are the same, i.e. \(e_r(p,q)=e_{r-1}\big(\alpha(p),\alpha(q)\big)\), compare Definition \ref{defn:A(n,r)}.
\newline
\textbf{Step 2: \(\rl(q^*,p)=\rl\big(\alpha(q)^*,\alpha(p)\big)\):}
We compare the graphs \(G(p,q)\) and \(G\big(\alpha(p),\alpha(q)\big)\):
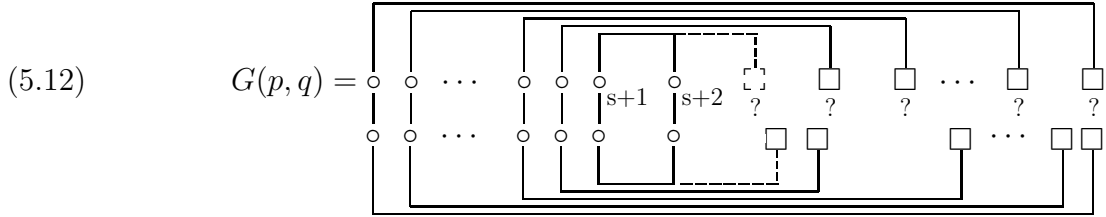
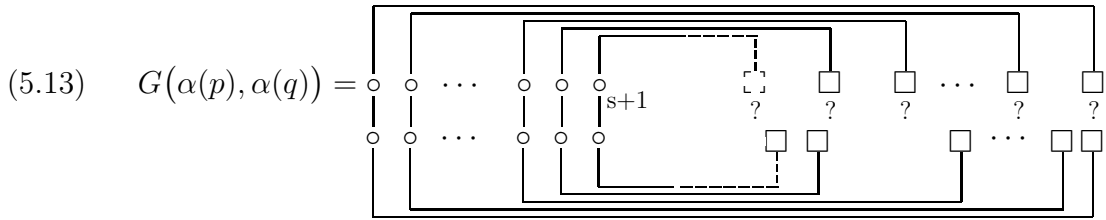
\begin{align}
\label{eqn:G(p,q)_for_W(n,r=2s+1)}
G(p,q)=\;&
\setlength{\unitlength}{0.5cm}
\begin{picture}(18,2.4)
\put(0,0) {$\circ$}
\put(1,0) {$\circ$}
\put(2,0) {$\cdots$}
\put(4,0) {$\circ$}
\put(5,0) {$\circ$}
\put(6,0) {$\circ$}
\put(8,0) {$\circ$}
\put(10,0) {$\square$}
\put(9.95,0.3) {\linethickness{0.1cm}\color{white}\line(1,0){0.7}}
\put(10.35,-0.05) {\linethickness{0.1cm}\color{white}\line(0,1){0.7}}
\put(12,0) {$\square$}
\put(14,0) {$\square$}
\put(15,0) {$\;\cdots$}
\put(17,0) {$\square$}
\put(19,0) {$\square$}
\put(0.2,0.5){\line(0,1){1.8}}
\put(1.2,0.5){\line(0,1){1.6}}
\put(4.2,0.5){\line(0,1){1.4}}
\put(5.2,0.5){\line(0,1){1.2}}
\put(6.2,0.5){\line(0,1){1}}
\put(8.2,0.5){\line(0,1){1}}
\put(0.2,-0.9){\line(0,1){0.8}}
\put(1.2,-0.9){\line(0,1){0.8}}
\put(4.2,-0.9){\line(0,1){0.8}}
\put(5.2,-0.9){\line(0,1){0.8}}
\put(6.2,-0.9){\line(0,1){0.8}}
\put(6.4,-0.4){\scalebox{0.8}{s+1}}
\put(8.4,-0.4){\scalebox{0.8}{s+2}}
\put(8.2,-0.9){\line(0,1){0.8}}
\put(10.2,-0.7){\scalebox{0.8}{?}}
\put(12.2,-0.7){\scalebox{0.8}{?}}
\put(14.2,-0.7){\scalebox{0.8}{?}}
\put(17.2,-0.7){\scalebox{0.8}{?}}
\put(19.2,-0.7){\scalebox{0.8}{?}}
\multiput(10.35,0.6)(0,0.3){3}{\line(0,1){0.2}}
\put(12.35,0.6){\line(0,1){1.1}}
\put(14.35,0.6){\line(0,1){1.3}}
\put(17.35,0.6){\line(0,1){1.5}}
\put(19.35,0.6){\line(0,1){1.7}}
\put(6.2,1.5){\line(1,0){2}}
\multiput(8.2,1.5)(0.383,0){6}{\line(1,0){0.3}}
\put(5.2,1.7){\line(1,0){7.15}}
\put(4.2,1.9){\line(1,0){10.15}}
\put(1.2,2.1){\line(1,0){16.15}}
\put(0.2,2.3){\line(1,0){19.15}}
\end{picture}
\hspace{-9.15cm}
\raisebox{-0.5cm}{\reflectbox{\scalebox{-1}{
\begin{picture}(18,2.4)
\put(0,0) {$\circ$}
\put(1,0) {$\circ$}
\put(2,0) {$\cdots$}
\put(4,0) {$\circ$}
\put(5,0) {$\circ$}
\put(6,0) {$\circ$}
\put(8,0) {$\circ$}
\put(10.6,0) {$\square$}
\put(9.95,0.3) {\linethickness{0.1cm}\color{white}\line(1,0){0.7}}
\put(10.35,-0.05) {\linethickness{0.1cm}\color{white}\line(0,1){0.7}}
\put(11.7,0) {$\square$}
\put(15.5,0) {$\square$}
\put(16.6,0) {$\cdots$}
\put(18.2,0) {$\square$}
\put(19,0) {$\square$}
\put(0.2,0.5){\line(0,1){1.8}}
\put(1.2,0.5){\line(0,1){1.6}}
\put(4.2,0.5){\line(0,1){1.4}}
\put(5.2,0.5){\line(0,1){1.2}}
\put(6.2,0.5){\line(0,1){1}}
\put(8.2,0.5){\line(0,1){1}}
\multiput(10.95,0.6)(0,0.3){3}{\line(0,1){0.2}}
\put(12.05,0.6){\line(0,1){1.1}}
\put(15.85,0.6){\line(0,1){1.3}}
\put(18.55,0.6){\line(0,1){1.5}}
\put(19.35,0.6){\line(0,1){1.7}}
\put(6.2,1.5){\line(1,0){2}}
\multiput(8.4,1.5)(0.38,0){7}{\line(1,0){0.3}}
\put(5.2,1.7){\line(1,0){6.85}}
\put(4.2,1.9){\line(1,0){11.65}}
\put(1.2,2.1){\line(1,0){17.35}}
\put(0.2,2.3){\line(1,0){19.15}}
\end{picture}
}}}\quad
\\[0.5cm]
\label{eqn:G(alpha(p),alpha(q))_for_W(n,r=2s+1)}
G\big(\alpha(p),\alpha(q)\big)=\;&
\setlength{\unitlength}{0.5cm}
\begin{picture}(18,2.4)
\put(0,0) {$\circ$}
\put(1,0) {$\circ$}
\put(2,0) {$\cdots$}
\put(4,0) {$\circ$}
\put(5,0) {$\circ$}
\put(6,0) {$\circ$}
\put(10,0) {$\square$}
\put(9.95,0.3) {\linethickness{0.1cm}\color{white}\line(1,0){0.7}}
\put(10.35,-0.05) {\linethickness{0.1cm}\color{white}\line(0,1){0.7}}
\put(12,0) {$\square$}
\put(14,0) {$\square$}
\put(15,0) {$\;\cdots$}
\put(17,0) {$\square$}
\put(19,0) {$\square$}
\put(0.2,0.5){\line(0,1){1.8}}
\put(1.2,0.5){\line(0,1){1.6}}
\put(4.2,0.5){\line(0,1){1.4}}
\put(5.2,0.5){\line(0,1){1.2}}
\put(6.2,0.5){\line(0,1){1}}
\put(0.2,-0.9){\line(0,1){0.8}}
\put(1.2,-0.9){\line(0,1){0.8}}
\put(4.2,-0.9){\line(0,1){0.8}}
\put(5.2,-0.9){\line(0,1){0.8}}
\put(6.2,-0.9){\line(0,1){0.8}}
\put(6.4,-0.4){\scalebox{0.8}{s+1}}
\put(10.2,-0.7){\scalebox{0.8}{?}}
\put(12.2,-0.7){\scalebox{0.8}{?}}
\put(14.2,-0.7){\scalebox{0.8}{?}}
\put(17.2,-0.7){\scalebox{0.8}{?}}
\put(19.2,-0.7){\scalebox{0.8}{?}}
\multiput(10.35,0.6)(0,0.3){3}{\line(0,1){0.2}}
\put(12.35,0.6){\line(0,1){1.1}}
\put(14.35,0.6){\line(0,1){1.3}}
\put(17.35,0.6){\line(0,1){1.5}}
\put(19.35,0.6){\line(0,1){1.7}}
\put(6.2,1.5){\line(1,0){2}}
\multiput(8.2,1.5)(0.383,0){6}{\line(1,0){0.3}}
\put(5.2,1.7){\line(1,0){7.15}}
\put(4.2,1.9){\line(1,0){10.15}}
\put(1.2,2.1){\line(1,0){16.15}}
\put(0.2,2.3){\line(1,0){19.15}}
\end{picture}
\hspace{-9.15cm}
\raisebox{-0.5cm}{\reflectbox{\scalebox{-1}{
\begin{picture}(18,2.4)
\put(0,0) {$\circ$}
\put(1,0) {$\circ$}
\put(2,0) {$\cdots$}
\put(4,0) {$\circ$}
\put(5,0) {$\circ$}
\put(6,0) {$\circ$}
\put(10.6,0) {$\square$}
\put(9.95,0.3) {\linethickness{0.1cm}\color{white}\line(1,0){0.7}}
\put(10.35,-0.05) {\linethickness{0.1cm}\color{white}\line(0,1){0.7}}
\put(11.7,0) {$\square$}
\put(15.5,0) {$\square$}
\put(16.6,0) {$\cdots$}
\put(18.2,0) {$\square$}
\put(19,0) {$\square$}
\put(0.2,0.5){\line(0,1){1.8}}
\put(1.2,0.5){\line(0,1){1.6}}
\put(4.2,0.5){\line(0,1){1.4}}
\put(5.2,0.5){\line(0,1){1.2}}
\put(6.2,0.5){\line(0,1){1}}
\multiput(10.95,0.6)(0,0.3){3}{\line(0,1){0.2}}
\put(12.05,0.6){\line(0,1){1.1}}
\put(15.85,0.6){\line(0,1){1.3}}
\put(18.55,0.6){\line(0,1){1.5}}
\put(19.35,0.6){\line(0,1){1.7}}
\put(6.2,1.5){\line(1,0){2}}
\multiput(8.4,1.5)(0.38,0){7}{\line(1,0){0.3}}
\put(5.2,1.7){\line(1,0){6.85}}
\put(4.2,1.9){\line(1,0){11.65}}
\put(1.2,2.1){\line(1,0){17.35}}
\put(0.2,2.3){\line(1,0){19.15}}
\end{picture}
}}}\quad
\end{align}
The symbols ? and the different positions of the \(\square\,\)-structures indicate that we do not know and do not care how the block structures of \(p\) and \(q\) actually look like in these parts of the graphs.
\newline
It is clear that \(G(p,q)\) and \(G\big(\alpha(p),\alpha(q)\big)\) have the same number of components as in \(p\) the point \(s+2\) was connected to \(s+1\) and in \(q\) the point \((s+2)'\) was connected to \((s+1)'\).
We conclude that the values \(\rl(q^*,p)\) and \(\rl\big(\alpha(q)^*,\alpha(p)\big)\) are the same.
\newline
\textbf{Step 3: \(r\)-flaws vs. \((r\!-\!1)\)-flaws:}
The proof is finished if we can show 
\[e_r(p,q)=e_{r-1}\big(\alpha(p),\alpha(q)\big),\]
but, despite Step 2, this is not yet guaranteed: On the left side we have to check whether \(G(p,q)\) has an \(r\)-flaw and on the right side we have to check whether \(G\big(\alpha(p),\alpha(q)\big)\) has an \((r\!-\!1)\)-flaw, see Definition \ref{defn:A(n,r)}.
We have to prove these two conditions to be equivalent.
The only problematic part of the above equivalence is the statement that an \(r\)-flaw in \(G(p,q)\) implies an \((r\!-\!1)\)-flaw in \(G\big(\alpha(p),\alpha(q)\big)\).
\newline 
Consider a graph \(G(p,q)\) with an \(r\)-flaw and assume \(G\big(\alpha(p),\alpha(q)\big)\) to be \((r\!-\!1)\)-flawless.
Let \(i\in\{1,\ldots,s+1\}\cup\{1',\ldots,(s+1)'\}\) be a point in \(G(p,q)\) that is a witness of the \(r\)-flaw of \(G(p,q)\).
The only possibility for \(i\) to be \emph{not} a witness of an \((r\!-\!1)\)-flaw in \(G\big(\alpha(p),\alpha(q)\big)\) is to be equal to \((s\!+\!1)\) and to be disconnected from \((s\!+\!1)'\) in \(H_r(p,q)\) (or the other way around).
But this is a contradiction, as by assumption \(p,q\in Y(n,r)\), so \((s\!+\!1)\) and \((s\!+\!1)'\) are connected via the points \((s\!+\!2)\) and \((s+2)'\), see Picture \ref{eqn:G(p,q)_for_W(n,r=2s+1)}.
We conclude that the assumption above was false and \(G\big(\alpha(p),\alpha(q)\big)\) has an \((r\!-\!1)\)-flaw.
\end{proof}
\begin{lem}[even case]\label{lem:connection:between_B(n,r)-and_A(n-1,r-1)_r=2s_and_notzero}
Let \(0\!<\! r\!=\!2s\!<\!n\!-\!1\).
Then, modulo row and column permutations, the matrix \(B(n,r)\) is equal to the matrix \(N\!\cdot\! A(n\!-\!1,r\!-\!1)\).
\end{lem}
\begin{proof}
We have to go through the same steps as in the proof of Lemma \ref{lem:connection:between_B(n,r)-and_A(n-1,r-1)_r=2s+1}.
\newline
\textbf{Step 1: \(Y(n,r)\simeq W(n\!-\!1,r\!-\!1)\): }
As \(r\) is even, the point \((s\!+\!1)\) in \(p\in Y(n,r)\) is a singleton, compare Lemma \ref{lem:elements_in_Y(n,r)}.
Erasing this point establishes the correspondence \(\alpha:Y(n,r)\rightarrow W(n\!-\!1,r\!-\!1)\).
It is bijective as obviously a preimage of \(p' \in W(n\!-\!1,r\!-\!1)\) can be constructed by adding a singleton in \(p'\) after the point \(s\).
As in the proof of Lemma \ref{lem:connection:between_B(n,r)-and_A(n-1,r-1)_r=2s+1}, Step 1, this allows us to label the rows and columns of \(B(n,r)\) and \(A(n\!-\!1,r\!-\!1)\) equally.
It remains to show that their entries are the same, i.e. \(e_r(p,q)=e_{r-1}\big(\alpha(p),\alpha(q)\big)\), compare Definition \ref{defn:A(n,r)}.
\newline
\textbf{Step 2: \(\rl(q^*,p)\!=1+\rl\big(\alpha(q)^*,\alpha(p)\big)\):}
We compare the graphs \(G(p,q)\) and \(G\big(\alpha(p),\alpha(q)\big)\):
\begin{align}
\label{eqn:G(p,q)_for_W(n,r=2s)}
G(p,q)=\quad&
\setlength{\unitlength}{0.5cm}
\begin{picture}(16,2.4)
\put(0,0) {$\circ$}
\put(1,0) {$\circ$}
\put(2,0) {$\cdots$}
\put(4,0) {$\circ$}
\put(5,0) {$\circ$}
\put(7,0) {$\circ$}
\put(10,0) {$\square$}
\put(12,0) {$\square$}
\put(13,0) {$\;\cdots$}
\put(15,0) {$\square$}
\put(17,0) {$\square$}
\put(0.2,0.5){\line(0,1){1.8}}
\put(1.2,0.5){\line(0,1){1.6}}
\put(4.2,0.5){\line(0,1){1.4}}
\put(5.2,0.5){\line(0,1){1.2}}
\put(7.2,0.5){\line(0,1){1}}
\put(0.2,-0.9){\line(0,1){0.8}}
\put(1.2,-0.9){\line(0,1){0.8}}
\put(4.2,-0.9){\line(0,1){0.8}}
\put(5.2,-0.9){\line(0,1){0.8}}
\put(5.4,-0.4){\scalebox{0.8}{s}}
\put(7.2,-0.9){\line(0,1){0.8}}
\put(7.4,-0.4){\scalebox{0.8}{s+1}}
\put(10.2,-0.7){\scalebox{0.8}{?}}
\put(12.2,-0.7){\scalebox{0.8}{?}}
\put(15.2,-0.7){\scalebox{0.8}{?}}
\put(17.2,-0.7){\scalebox{0.8}{?}}
\put(10.35,0.6){\line(0,1){1.1}}
\put(12.35,0.6){\line(0,1){1.3}}
\put(15.35,0.6){\line(0,1){1.5}}
\put(17.35,0.6){\line(0,1){1.7}}
\put(5.2,1.7){\line(1,0){5.15}}
\put(4.2,1.9){\line(1,0){8.15}}
\put(1.2,2.1){\line(1,0){14.15}}
\put(0.2,2.3){\line(1,0){17.15}}
\end{picture}
\hspace{-8.15cm}
\raisebox{-0.5cm}{\reflectbox{\scalebox{-1}{
\begin{picture}(16,2.4)
\put(0,0) {$\circ$}
\put(1,0) {$\circ$}
\put(2,0) {$\cdots$}
\put(4,0) {$\circ$}
\put(5,0) {$\circ$}
\put(7,0) {$\circ$}
\put(9.7,0) {$\square$}
\put(13.5,0) {$\square$}
\put(14.6,0) {$\cdots$}
\put(16.2,0) {$\square$}
\put(17,0) {$\square$}
\put(0.2,0.5){\line(0,1){1.8}}
\put(1.2,0.5){\line(0,1){1.6}}
\put(4.2,0.5){\line(0,1){1.4}}
\put(5.2,0.5){\line(0,1){1.2}}
\put(7.2,0.5){\line(0,1){1}}
\put(10.05,0.6){\line(0,1){1.1}}
\put(13.85,0.6){\line(0,1){1.3}}
\put(16.55,0.6){\line(0,1){1.5}}
\put(17.35,0.6){\line(0,1){1.7}}
\put(5.2,1.7){\line(1,0){4.85}}
\put(4.2,1.9){\line(1,0){9.65}}
\put(1.2,2.1){\line(1,0){15.35}}
\put(0.2,2.3){\line(1,0){17.15}}
\end{picture}
}}}
\\[0.5cm]
\label{eqn:G(alpha(p),alpha(q))_for_W(n,r=2s)}
G\big(\alpha(p),\alpha(q)\big)=\quad&
\setlength{\unitlength}{0.5cm}
\begin{picture}(16,2.4)
\put(0,0) {$\circ$}
\put(1,0) {$\circ$}
\put(2,0) {$\cdots$}
\put(4,0) {$\circ$}
\put(5,0) {$\circ$}
\put(10,0) {$\square$}
\put(12,0) {$\square$}
\put(13,0) {$\;\cdots$}
\put(15,0) {$\square$}
\put(17,0) {$\square$}
\put(0.2,0.5){\line(0,1){1.8}}
\put(1.2,0.5){\line(0,1){1.6}}
\put(4.2,0.5){\line(0,1){1.4}}
\put(5.2,0.5){\line(0,1){1.2}}
\put(0.2,-0.9){\line(0,1){0.8}}
\put(1.2,-0.9){\line(0,1){0.8}}
\put(4.2,-0.9){\line(0,1){0.8}}
\put(5.2,-0.9){\line(0,1){0.8}}
\put(5.4,-0.4){\scalebox{0.8}{s}}
\put(10.2,-0.7){\scalebox{0.8}{?}}
\put(12.2,-0.7){\scalebox{0.8}{?}}
\put(15.2,-0.7){\scalebox{0.8}{?}}
\put(17.2,-0.7){\scalebox{0.8}{?}}
\put(10.35,0.6){\line(0,1){1.1}}
\put(12.35,0.6){\line(0,1){1.3}}
\put(15.35,0.6){\line(0,1){1.5}}
\put(17.35,0.6){\line(0,1){1.7}}
\put(5.2,1.7){\line(1,0){5.15}}
\put(4.2,1.9){\line(1,0){8.15}}
\put(1.2,2.1){\line(1,0){14.15}}
\put(0.2,2.3){\line(1,0){17.15}}
\end{picture}
\hspace{-8.15cm}
\raisebox{-0.5cm}{\reflectbox{\scalebox{-1}{
\begin{picture}(16,2.4)
\put(0,0) {$\circ$}
\put(1,0) {$\circ$}
\put(2,0) {$\cdots$}
\put(4,0) {$\circ$}
\put(5,0) {$\circ$}
\put(9.7,0) {$\square$}
\put(13.5,0) {$\square$}
\put(14.6,0) {$\cdots$}
\put(16.2,0) {$\square$}
\put(17,0) {$\square$}
\put(0.2,0.5){\line(0,1){1.8}}
\put(1.2,0.5){\line(0,1){1.6}}
\put(4.2,0.5){\line(0,1){1.4}}
\put(5.2,0.5){\line(0,1){1.2}}
\put(10.05,0.6){\line(0,1){1.1}}
\put(13.85,0.6){\line(0,1){1.3}}
\put(16.55,0.6){\line(0,1){1.5}}
\put(17.35,0.6){\line(0,1){1.7}}
\put(5.2,1.7){\line(1,0){4.85}}
\put(4.2,1.9){\line(1,0){9.65}}
\put(1.2,2.1){\line(1,0){15.35}}
\put(0.2,2.3){\line(1,0){17.15}}
\end{picture}
}}}
\end{align}
It is clear that \(G(p,q)\) has one component more than \(G\big(\alpha(p),\alpha(q)\big)\), namely the one containing exactly the two points \((s\!+\!1)\) and \((s\!+\!1)'\), see Equation \ref{eqn:G(p,q)_for_W(n,r=2s)} and \ref{eqn:G(alpha(p),alpha(q))_for_W(n,r=2s)}.
We conclude that the values \(N^{\rl(q^*,p)}\) and \(N\!\cdot\! N^{\rl(\alpha(q)^*,\alpha(p))}\) are the same.
\newline
\textbf{Step 3: \(r\)-flaws vs. \((r\!-\!1)\)-flaws:}
Again, it remains to prove the equality
\[e_r(p,q)=N\cdot e_{r-1}\big(\alpha(p),\alpha(q)\big)\]
and due to Step 2 we only need to show that \(r\)-flaws in \(G(p,q)\) are equivalent to \((r\!-\!1)\)-flaws in \(G\big(\alpha(p),\alpha(q)\big)\).
Note that \(r\!=\!2s\) is even, so \(r\!-\!1=2(s\!-\!1)\!+\!1\) is associated to the increment of \(s\).
\newline
As above, the only non-trivial implication is that an \(r\)-flaw in \(G(p,q)\) implies an \((r\!-\!1)\)-flaw in \(G\big(\alpha(p),\alpha(q)\big)\).
To prove this, consider a witness \(i\in\{1,\ldots,s+1\}\cup\{1',\ldots,(s+1)'\}\) of an \(r\)-flaw in \(G(p,q)\).
Assuming \(G\big(\alpha(p),\alpha(q)\big)\) to be \((r-1)\)-flaw\-less requires \(i\) to be the point \(s+1\) (or \((s+1)'\)) and it must be connected in \(H_r(p,q)\) to another point amongst \(1,\ldots,s\) (or \(1',\ldots,s'\)).
This is again a contradiction, as both \((s+1)\) and \((s+1)'\)  are singletons in \(p\) and \(q\), respectively, so they are singletons in \(H_r(p,q)\), too.
Hence, the assumption above was false and \(G\big(\alpha(p),\alpha(q)\big)\) has an \((r\!-\!1)\)-flaw.
\end{proof}
\begin{lem}\label{lem:connection:between_B(n,0)-and_A(n-1,0)}
Let \(2\!\le\!n\in\N\).
Modulo row and column permutations, the matrix \(B(n,0)\) is equal to the matrix  \(N\!\cdot\! A(n\!-\!1,0)\).
\end{lem}
\begin{proof}
We consider the map \(\alpha\) as in the proof of Lemma \ref{lem:connection:between_B(n,r)-and_A(n-1,r-1)_r=2s_and_notzero} and observe that we establish the same one-to-one correspondence as before between \(Y(n,0)\) and \(W(n-1,0)\) by deleting the singletons \(1\) and \(1'\) from \(p\) and \(q\), respectively.
Again the graph \(G(\alpha(p),\alpha(q))\) has one component less then \(G(p,q)\), justifying the additional factor \(N\).
As there are no graphs with zero-flaws we finally have for all \(p,q\!\in\!Y(n,0)\):
\begin{align*} 
\big(B(n,0)\big)_{p,q}=&(A(n,0))_{p,q}\\
=&e_0(p,q)\\
=&N^{\rl(q^*,p)}\\
=&N\cdot N^{\rl(\alpha(q)^*,\alpha(p))}\\
=&N\cdot e_0(\alpha(p),\alpha(q))\\
=&N\cdot\big(A(n\!-\!1,0)\big)_{\alpha(p),\alpha(q)}.
\end{align*}
\end{proof}
%
%
%
%
%
We summarize the results of the last three lemmata in one proposition.
\begin{prop}\label{prop:connections_between_B(n,r)_and_A(n,r)}
It holds
\[B(n,r)=
\begin{cases}
A(n\!-\!1,r\!-\!1)&,0<r=2s\!+\!1<n-1\\
N\cdot A(n\!-\!1,r\!-\!1)&,0<r=2s<n-1\\
N\cdot A(n-1,0)&,r=0, n\ge2.
\end{cases}
\]
\end{prop}
In the case \(r=n\!-\!1\) the set \(Y(n,r)=Y(n,n\!-\!1)=W(n,n-1)\) contains only one element, namely the one displayed in Lemma \ref{lem:W(n,n-1)_contains_only_one_element}. So we directly get the following result:
\begin{prop}\label{prop:B(n,n-1)=A(n,n-1)_is_a_complex_number}
For \(n\in\N\) we have
\[B(n,n\!-\!1)=A(n,n\!-\!1)=N^{\left\lceil\frac{n}{2}\right\rceil}\!\in\! M_1(\C).\]
In particular, it holds 
\[\det\big(B(n,n\!-\!1)\big)=\det\big(A(n,n\!-\!1)\big)=N^{\left\lceil\frac{n}{2}\right\rceil}.\]
\end{prop}
Propositions \ref{prop:connections_between_B(n,r)_and_A(n,r)} and \ref{prop:B(n,n-1)=A(n,n-1)_is_a_complex_number} are ingredients to prove a recursion  formula for the determinant of \(A(n,0)\), see Section \ref{subsec:recursion_formula}.
\subsubsection{Partition and graph manipulations}

\begin{defn}\label{defn:X(i)and K(i)}
For a fixed partition \(p\!\in\! W(n,r)\) we denote by \(X(i)\) the block of \(p\) containing \(i\) and we define \(K(i):=X(i)\backslash\{i\}\).
\end{defn}
\begin{notation}\label{notation:displaying_connections_in_a_graph}
In order to display connections between points or sets of points in a graph we will use the following scheme:
We list the relevant (sets of) points and draw lines between them to indicate whether they are connected or not.
For example by Definition \ref{defn:X(i)and K(i)} a partition \(p\!\in\!W(n,r)\) fulfils the following:
\begin{equation}\label{eqn:scheme_for_the_K(i)}
\begin{tabular}{rclcrcl}
\(1\)&\textnormal{---}&\(K(1)\)\\
&\vdots&\\
\(i\)&\textnormal{---}&\(K(i)\)\\
&\vdots&\\
\(s\)&\textnormal{---}&\(K(s)\)\\[4pt]
\((s\!+\!1)\)&\textnormal{---}&\(K(s\!+\!1)\)
\end{tabular}
\end{equation}
Note that \(K(s\!+\!1)\) might be empty if \(r\) is even.
\end{notation}

We define now the following partition manipulations:
\begin{defn}\label{defn:f(i,q)}
Let \(0\le r<n\!-\!1\) and \(q\in W(n,r\!+\!1)\). For \(1\le i\le (s+1)\) we define the partition \(f(i,q)\) by  performing the following changes on \(q\):
\begin{itemize}
\item[(i)] A point \(i\le j\le s\) is not connected to \(K(j)\) any more, but to \(K(j\!+\!1)\).
\item[(ii)] If \(r\) is even, then the point \((s\!+\!1)\) is not connected to \(K(s\!+\!1)\) any more.
\item[(iii)] If \(r\) is odd, then the point \((s\!+\!1)\) is not connected to \(K(s\!+\!1)\) any more, but to \(X(s\!+\!2)\).
\end{itemize}
\end{defn}
Note that the mentioned sets \(K(1),\ldots, K(s+1)\) and \(K(1),\ldots, K(s+1), X(s+2)\), respectively, are disjoint and non-empty as \(q\in W(n,r+1)\).
\newline
Using Notation \ref{notation:displaying_connections_in_a_graph}, we could have defined \(f(i,q)\) by starting with \(q\) and defining the following changes in the connectivities, depending on the parity of \(r\):
\begin{equation}\label{eqn:scheme_for_f(i,q)}
\begin{tabular}{rclcrcl}
&\begin{picture}(0,0)\hspace{-0.5cm}\(r\) even:\end{picture}&&&&\begin{picture}(0,0)\hspace{-0.5cm}\(r\) odd:\end{picture}&\\[11pt]
\(i\)&\begin{picture}(0,0)\put(-0.3,0){{\line(1,-1){1}}}\end{picture}&\(K(i)\)&\(\quad\quad\quad\quad\quad\)&\(i\)&\begin{picture}(0,0)\put(-0.3,0){{\line(1,-1){1}}}\end{picture}&\(K(i)\)\\[4pt]
\vdots&&\(K(i\!+\!1)\)&&\vdots&&\(K(i\!+\!1)\)\\
&&\vdots&&&&\vdots\\
\(s\)&\begin{picture}(0,0)\put(-0.3,0){{\line(1,-1){1}}}\end{picture}&&&\(s\)&\begin{picture}(0,0)\put(-0.3,0){{\line(1,-1){1}}}\end{picture}&\\[8pt]
\((s\!+\!1)\)&&\(K(s\!+\!1)\)&&\((s\!+\!1)\)&\begin{picture}(0,0)\put(-0.3,0){{\line(1,-1){1}}}\end{picture}&\(K(s\!+\!1)\)\\[8pt]
&&&&&&\(X(s\!+\!2)\)\\[8pt]
\end{tabular}
\end{equation}
Here are the picture for this manipulation, once for a partition \(q\) in the case \(r=2s\) and once for a partition \(q'\) in the case \(r=2s+1\) (Recall that \(q,q'\in W(n,r+1)\)): \vspace{20pt}
{\allowdisplaybreaks
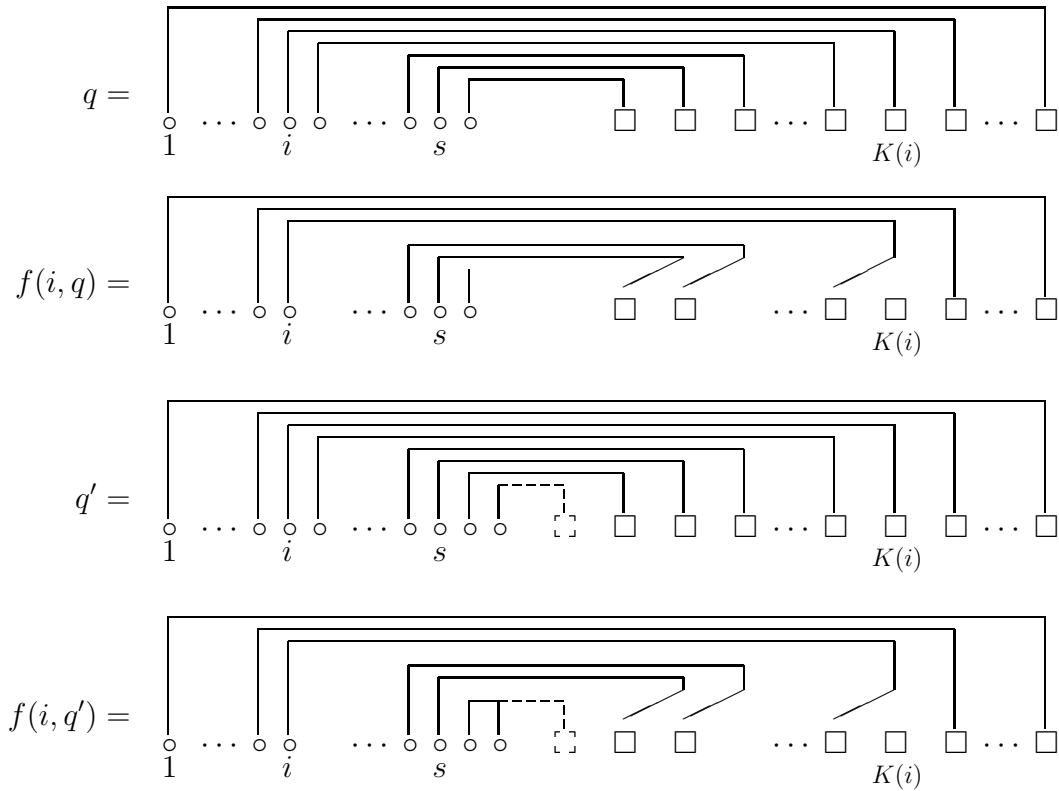
\begin{align*}
q=&\quad\parbox[c]{10cm}{
\setlength{\unitlength}{0.4cm}
\begin{picture}(28,2.4)
\put(0,0) {$\circ$}
\put(1,0) {$\;\cdots$}
\put(3,0) {$\circ$}
\put(4,0) {$\circ$}
\put(5,0) {$\circ$}
\put(6,0) {$\;\cdots$}
\put(8,0) {$\circ$}
\put(9,0) {$\circ$}
\put(10,0) {$\circ$}
\put(15,0) {$\square$}
\put(17,0) {$\square$}
\put(19,0) {$\square$}
\put(20,0) {$\;\cdots$}
\put(22,0) {$\square$}
\put(24,0) {$\square$}
\put(26,0) {$\square$}
\put(27,0) {$\;\cdots$}
\put(29,0) {$\square$}
\put(0,-0.8) {$1$}
\put(4,-0.8) {$i$}
\put(23.6,-1) {\scalebox{0.8}{$K(i)$}}
\put(9,-0.8) {$s$}
\put(0.2,0.6){\line(0,1){3.5}}
\put(3.2,0.6){\line(0,1){3.1}}
\put(4.2,0.6){\line(0,1){2.7}}
\put(5.2,0.6){\line(0,1){2.3}}
\put(8.2,0.6){\line(0,1){1.9}}
\put(9.2,0.6){\line(0,1){1.5}}
\put(10.2,0.6){\line(0,1){1.1}}
\put(15.35,0.8){\line(0,1){0.9}}
\put(17.35,0.8){\line(0,1){1.3}}
\put(19.35,0.8){\line(0,1){1.7}}
\put(22.35,0.8){\line(0,1){2.1}}
\put(24.35,0.8){\line(0,1){2.5}}
\put(26.35,0.8){\line(0,1){2.9}}
\put(29.35,0.8){\line(0,1){3.3}}
\put(10.2,1.7){\line(1,0){5.15}}
\put(9.2,2.1){\line(1,0){8.15}}
\put(8.2,2.5){\line(1,0){11.15}}
\put(5.2,2.9){\line(1,0){17.15}}
\put(4.2,3.3){\line(1,0){20.15}}
\put(3.2,3.7){\line(1,0){23.15}}
\put(0.2,4.1){\line(1,0){29.15}}
\end{picture}}\quad\quad\quad\quad\quad\;\;
\\[40pt]
f(i,q)=&\quad\parbox[c]{10cm}{
\setlength{\unitlength}{0.4cm}
\begin{picture}(28,2.4)
\put(0,0) {$\circ$}
\put(1,0) {$\;\cdots$}
\put(3,0) {$\circ$}
\put(4,0) {$\circ$}
\put(6,0) {$\;\cdots$}
\put(8,0) {$\circ$}
\put(9,0) {$\circ$}
\put(10,0) {$\circ$}
\put(15,0) {$\square$}
\put(17,0) {$\square$}
\put(20,0) {$\;\cdots$}
\put(22,0) {$\square$}
\put(24,0) {$\square$}
\put(26,0) {$\square$}
\put(27,0) {$\;\cdots$}
\put(29,0) {$\square$}
\put(0,-0.8) {$1$}
\put(4,-0.8) {$i$}
\put(23.6,-1) {\scalebox{0.8}{$K(i)$}}
\put(9,-0.8) {$s$}
\put(0.2,0.6){\line(0,1){3.5}}
\put(3.2,0.6){\line(0,1){3.1}}
\put(4.2,0.6){\line(0,1){2.7}}
\put(8.2,0.6){\line(0,1){1.9}}
\put(9.2,0.6){\line(0,1){1.5}}
\put(10.2,0.6){\line(0,1){1.1}}
\put(26.35,0.8){\line(0,1){2.9}}
\put(29.35,0.8){\line(0,1){3.3}}
\put(9.2,2.1){\line(1,0){8.15}}
\put(8.2,2.5){\line(1,0){11.15}}
\put(4.2,3.3){\line(1,0){20.15}}
\put(3.2,3.7){\line(1,0){23.15}}
\put(0.2,4.1){\line(1,0){29.15}}
\put(17.35,2.1){\line(-2,-1){2}}
\put(19.35,2.1){\line(-2,-1){2}}
\put(24.35,2.1){\line(-2,-1){2}}
\put(19.35,2.1){\line(0,1){0.4}}
\put(24.35,2.1){\line(0,1){1.2}}
\end{picture}}\quad\quad\quad\quad\quad\;\;
\\[50pt]
q'=&\quad\parbox[c]{10cm}{
\setlength{\unitlength}{0.4cm}
\begin{picture}(28,2.4)
\put(0,0) {$\circ$}
\put(1,0) {$\;\cdots$}
\put(3,0) {$\circ$}
\put(4,0) {$\circ$}
\put(5,0) {$\circ$}
\put(6,0) {$\;\cdots$}
\put(8,0) {$\circ$}
\put(9,0) {$\circ$}
\put(10,0) {$\circ$}
\put(11,0) {$\circ$}
\put(13,0) {$\square$}
\put(13,0.36) {\linethickness{0.1cm}\color{white}\line(1,0){0.8}}
\put(13.38,0) {\linethickness{0.1cm}\color{white}\line(0,1){0.8}}
\put(15,0) {$\square$}
\put(17,0) {$\square$}
\put(19,0) {$\square$}
\put(20,0) {$\;\cdots$}
\put(22,0) {$\square$}
\put(24,0) {$\square$}
\put(26,0) {$\square$}
\put(27,0) {$\;\cdots$}
\put(29,0) {$\square$}
\put(0,-0.8) {$1$}
\put(4,-0.8) {$i$}
\put(23.6,-1) {\scalebox{0.8}{$K(i)$}}
\put(9,-0.8) {$s$}
\put(0.2,0.6){\line(0,1){3.9}}
\put(3.2,0.6){\line(0,1){3.5}}
\put(4.2,0.6){\line(0,1){3.1}}
\put(5.2,0.6){\line(0,1){2.7}}
\put(8.2,0.6){\line(0,1){2.3}}
\put(9.2,0.6){\line(0,1){1.9}}
\put(10.2,0.6){\line(0,1){1.5}}
\put(11.2,0.6){\line(0,1){1.1}}
\multiput(13.38,0.8)(0,0.46){2}{\line(0,1){0.3}}
\put(15.35,0.8){\line(0,1){1.3}}
\put(17.35,0.8){\line(0,1){1.7}}
\put(19.35,0.8){\line(0,1){2.1}}
\put(22.35,0.8){\line(0,1){2.5}}
\put(24.35,0.8){\line(0,1){2.9}}
\put(26.35,0.8){\line(0,1){3.3}}
\put(29.35,0.8){\line(0,1){3.7}}
\multiput(11.2,1.7)(0.46,0){5}{\line(1,0){0.3}}
\put(10.2,2.1){\line(1,0){5.15}}
\put(9.2,2.5){\line(1,0){8.15}}
\put(8.2,2.9){\line(1,0){11.15}}
\put(5.2,3.3){\line(1,0){17.15}}
\put(4.2,3.7){\line(1,0){20.15}}
\put(3.2,4.1){\line(1,0){23.15}}
\put(0.2,4.5){\line(1,0){29.15}}
\end{picture}}\quad\quad\quad\quad\quad\;\;
\\[50pt]
f(i,q')=&\quad\parbox[c]{10cm}{
\setlength{\unitlength}{0.4cm}
\begin{picture}(28,2.4)
\put(0,0) {$\circ$}
\put(1,0) {$\;\cdots$}
\put(3,0) {$\circ$}
\put(4,0) {$\circ$}
\put(6,0) {$\;\cdots$}
\put(8,0) {$\circ$}
\put(9,0) {$\circ$}
\put(10,0) {$\circ$}
\put(11,0) {$\circ$}
\put(13,0) {$\square$}
\put(13,0.36) {\linethickness{0.1cm}\color{white}\line(1,0){0.8}}
\put(13.38,0) {\linethickness{0.1cm}\color{white}\line(0,1){0.8}}
\put(15,0) {$\square$}
\put(17,0) {$\square$}
\put(20,0) {$\;\cdots$}
\put(22,0) {$\square$}
\put(24,0) {$\square$}
\put(26,0) {$\square$}
\put(27,0) {$\;\cdots$}
\put(29,0) {$\square$}
\put(0,-0.8) {$1$}
\put(4,-0.8) {$i$}
\put(23.6,-1) {\scalebox{0.8}{$K(i)$}}
\put(9,-0.8) {$s$}
\put(0.2,0.6){\line(0,1){3.9}}
\put(3.2,0.6){\line(0,1){3.5}}
\put(4.2,0.6){\line(0,1){3.1}}
\put(8.2,0.6){\line(0,1){2.3}}
\put(9.2,0.6){\line(0,1){1.9}}
\put(10.2,0.6){\line(0,1){1.1}}
\put(11.2,0.6){\line(0,1){1.1}}
\multiput(13.38,0.8)(0,0.46){2}{\line(0,1){0.3}}
\put(26.35,0.8){\line(0,1){3.3}}
\put(29.35,0.8){\line(0,1){3.7}}
\multiput(11.2,1.7)(0.46,0){5}{\line(1,0){0.3}}
\put(10.2,1.7){\line(1,0){1}}
\put(9.2,2.5){\line(1,0){8.15}}
\put(8.2,2.9){\line(1,0){11.15}}
\put(4.2,3.7){\line(1,0){20.15}}
\put(3.2,4.1){\line(1,0){23.15}}
\put(0.2,4.5){\line(1,0){29.15}}
\put(17.35,2.1){\line(-2,-1){2}}
\put(19.35,2.1){\line(-2,-1){2}}
\put(24.35,2.1){\line(-2,-1){2}}
\put(17.35,2.1){\line(0,1){0.4}}
\put(19.35,2.1){\line(0,1){0.8}}
\put(24.35,2.1){\line(0,1){1.6}}
\end{picture}}\quad\quad\quad\quad\quad\quad
\end{align*}
}
\vspace{10pt}\newline
Note that in the case \(i\!=\!s+1\) we only separate \(s+1\) from \(K(s+1)\) and when \(r\!=\!2s+1\) we additionally join \(s\!+\!1\) with \(X(s\!+\!2)\).
\vspace{11pt}\newline
Similar to the construction of \(f(i,p)\) we define partitions \(g(i,p)\).
The only difference is that the point \(i\) stays connected to \(K(i)\).
\begin{defn}\label{defn:g(i,q)}
Let \(0\le r<n\!-\!1\) and \(q\in W(n,r\!+\!1)\). For \(1\le i\le s\) we define the partition \(g(i,q)\) by performing the following changes on \(q\):
\begin{itemize}
\item[(i)] The point \(i\) is connected additionally to \(K(i\!+\!1)\).
\item[(ii)] A point \(i<j\le s\) is not connected to \(K(j)\) any more, but to \(K(j\!+\!1)\)
\item[(iii)] If \(r\) is even, then the point \((s\!+\!1)\) is not connected to \(K(s\!+\!1)\) any more.
\item[(iv)] If \(r\) is odd, then the point \((s\!+\!1)\) is not connected to \(K(s\!+\!1)\) any more, but to \(X(s\!+\!2)\).
\end{itemize}
In the case \(r=2s\!+\!1\) we define in addition the case \(i\!=\!s\!+\!1\): \(g(s\!+\!1,q)\) is constructed out of \(q\) by joining \(X(s\!+\!1)\) with \(X(s\!+\!2)\), so the above changes are reduced to item (i).
\end{defn}
Using the Notations from \ref{notation:displaying_connections_in_a_graph},we could have defined \(g(i,q)\) by starting with \(q\) and defining the following changes in the connectivities, depending on the parity of \(r\):
\begin{equation}\label{eqn:scheme_for_g(i,q)}
\begin{tabular}{rclcrcl}
&\begin{picture}(0,0)\hspace{-0.5cm}\(r\) even:\end{picture}&&&&\begin{picture}(0,0)\hspace{-0.5cm}\(r\) odd:\end{picture}&\\[11pt]
\(i\)&\begin{picture}(0,0)\put(-0.3,0){\line(1,-1){1}}\put(-0.3,0.2){\line(1,0){0.9}}\end{picture}&\(K(i)\)&\(\quad\quad\quad\quad\quad\)&\(i\)&\begin{picture}(0,0)\put(-0.3,0){{\line(1,-1){1}}}\put(-0.3,0.2){\line(1,0){0.9}}\end{picture}&\(K(i)\)\\[4pt]
\vdots&&\(K(i\!+\!1)\)&&\vdots&&\(K(i\!+\!1)\)\\
&&\vdots&&&&\vdots\\
\(s\)&\begin{picture}(0,0)\put(-0.3,0){{\line(1,-1){1}}}\end{picture}&&&\(s\)&\begin{picture}(0,0)\put(-0.3,0){{\line(1,-1){1}}}\end{picture}&\\[8pt]
\((s\!+\!1)\)&&\(K(s\!+\!1)\)&&\((s\!+\!1)\)&\begin{picture}(0,0)\put(-0.3,0){{\line(1,-1){1}}}\end{picture}&\(K(s\!+\!1)\)\\[8pt]
&&&&&&\(X(s\!+\!2)\)\\[8pt]
\end{tabular}
\end{equation}
Considering \(q\) and \(q'\) as above we end up with the following partitions:\vspace{22pt}
{\allowdisplaybreaks
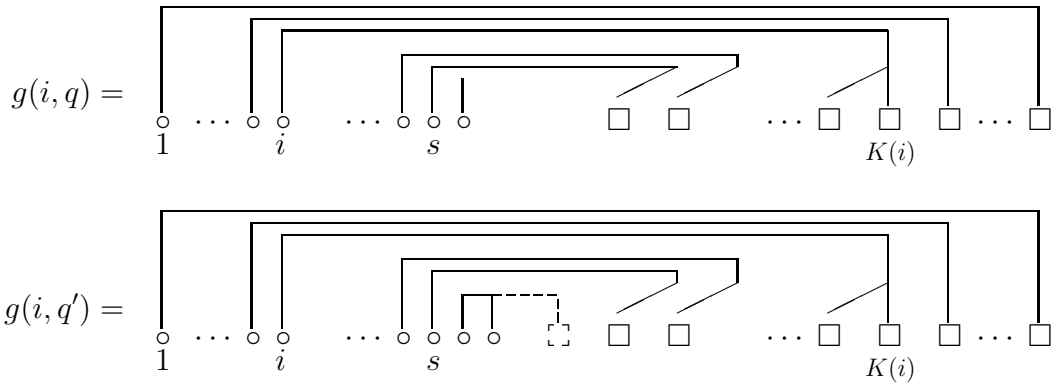
\begin{align*}
g(i,q)=&\quad\parbox[c]{10cm}{
\setlength{\unitlength}{0.4cm}
\begin{picture}(28,2.4)
\put(0,0) {$\circ$}
\put(1,0) {$\;\cdots$}
\put(3,0) {$\circ$}
\put(4,0) {$\circ$}
\put(6,0) {$\;\cdots$}
\put(8,0) {$\circ$}
\put(9,0) {$\circ$}
\put(10,0) {$\circ$}
\put(15,0) {$\square$}
\put(17,0) {$\square$}
\put(20,0) {$\;\cdots$}
\put(22,0) {$\square$}
\put(24,0) {$\square$}
\put(26,0) {$\square$}
\put(27,0) {$\;\cdots$}
\put(29,0) {$\square$}
\put(0,-0.8) {$1$}
\put(4,-0.8) {$i$}
\put(23.6,-1) {\scalebox{0.8}{$K(i)$}}
\put(9,-0.8) {$s$}
\put(0.2,0.6){\line(0,1){3.5}}
\put(3.2,0.6){\line(0,1){3.1}}
\put(4.2,0.6){\line(0,1){2.7}}
\put(8.2,0.6){\line(0,1){1.9}}
\put(9.2,0.6){\line(0,1){1.5}}
\put(10.2,0.6){\line(0,1){1.1}}
\put(26.35,0.8){\line(0,1){2.9}}
\put(29.35,0.8){\line(0,1){3.3}}
\put(9.2,2.1){\line(1,0){8.15}}
\put(8.2,2.5){\line(1,0){11.15}}
\put(4.2,3.3){\line(1,0){20.15}}
\put(3.2,3.7){\line(1,0){23.15}}
\put(0.2,4.1){\line(1,0){29.15}}
\put(17.35,2.1){\line(-2,-1){2}}
\put(19.35,2.1){\line(-2,-1){2}}
\put(24.35,2.1){\line(-2,-1){2}}
\put(19.35,2.1){\line(0,1){0.4}}
\put(24.35,2.1){\line(0,1){1.2}}
\put(24.35,0.8){\line(0,1){1.3}}
\end{picture}}\quad\quad\quad\quad\quad\;\;
\\[50pt]
g(i,q')=&\quad\parbox[c]{10cm}{
\setlength{\unitlength}{0.4cm}
\begin{picture}(28,2.4)
\put(0,0) {$\circ$}
\put(1,0) {$\;\cdots$}
\put(3,0) {$\circ$}
\put(4,0) {$\circ$}
\put(6,0) {$\;\cdots$}
\put(8,0) {$\circ$}
\put(9,0) {$\circ$}
\put(10,0) {$\circ$}
\put(11,0) {$\circ$}
\put(13,0) {$\square$}
\put(13,0.36) {\linethickness{0.1cm}\color{white}\line(1,0){0.8}}
\put(13.38,0) {\linethickness{0.1cm}\color{white}\line(0,1){0.8}}
\put(15,0) {$\square$}
\put(17,0) {$\square$}
\put(20,0) {$\;\cdots$}
\put(22,0) {$\square$}
\put(24,0) {$\square$}
\put(26,0) {$\square$}
\put(27,0) {$\;\cdots$}
\put(29,0) {$\square$}
\put(0,-0.8) {$1$}
\put(4,-0.8) {$i$}
\put(23.6,-1) {\scalebox{0.8}{$K(i)$}}
\put(9,-0.8) {$s$}
\put(0.2,0.6){\line(0,1){3.9}}
\put(3.2,0.6){\line(0,1){3.5}}
\put(4.2,0.6){\line(0,1){3.1}}
\put(8.2,0.6){\line(0,1){2.3}}
\put(9.2,0.6){\line(0,1){1.9}}
\put(10.2,0.6){\line(0,1){1.1}}
\put(11.2,0.6){\line(0,1){1.1}}
\multiput(13.38,0.8)(0,0.46){2}{\line(0,1){0.3}}
\put(26.35,0.8){\line(0,1){3.3}}
\put(29.35,0.8){\line(0,1){3.7}}
\multiput(11.2,1.7)(0.46,0){5}{\line(1,0){0.3}}
\put(10.2,1.7){\line(1,0){1}}
\put(9.2,2.5){\line(1,0){8.15}}
\put(8.2,2.9){\line(1,0){11.15}}
\put(4.2,3.7){\line(1,0){20.15}}
\put(3.2,4.1){\line(1,0){23.15}}
\put(0.2,4.5){\line(1,0){29.15}}
\put(17.35,2.1){\line(-2,-1){2}}
\put(19.35,2.1){\line(-2,-1){2}}
\put(24.35,2.1){\line(-2,-1){2}}
\put(17.35,2.1){\line(0,1){0.4}}
\put(19.35,2.1){\line(0,1){0.8}}
\put(24.35,2.1){\line(0,1){1.6}}
\put(24.35,0.8){\line(0,1){1.3}}
\end{picture}}\quad\quad\quad\quad\quad\quad
\end{align*}
}
\vspace{10pt}\newline
Taking a closer look at the illustrations above, the following can directly be checked:
\begin{lem}\label{lem:f(i,q)_and_g(i,q)_are_in_W(n,r)}
For \(q\in W(n,r+1)\) the partitions \(f(i,q)\) and \(g(i,q)\) as defined in Definition \ref{defn:f(i,q)} and \ref{defn:g(i,q)} are elements in \(W(n,r)\).
\end{lem}
Having \(p\in W(n,r)\) and \(q\in W(n,r+1)\), we know that in each partition the first \(s+1\) points are not connected to each other (and in the case \(r\!=\!2s\!+\!1\) even the first \(s\!+\!2\) points of \(q\)).
In contrast to that, we cannot in general guarantee this property in the graph \(G(p,q)\), as we do not know enough about the part of \(G(p,q)\) right of \(s\!+\!1\) and \((s\!+\!1)'\).
The next definition describes some special structures that might occur.
Again the descriptions themselves are quite technical, but see below for comprehensive illustrations of the structures.
\begin{defn}\label{defn:structures[i]_a.s.o.}
Let \(0\!\le\! r\!<\!n\!-\!1\), \(p\in W(n,r)\) and \(q\in W(n,r\!+\!1)\).
For even \(r\) we define on the graph \(H_r(p,q)\)
\begin{itemize}
\item[(1)] a structure \([i]\), for \(1\le i\le s+1\),
\item[(2)] a structure \([i,i\!+\!1]\) for \(1\le i\le s\) and
\item[(3)] a structure \([0]\)
\end{itemize}
if and only if the following respective connections are given:
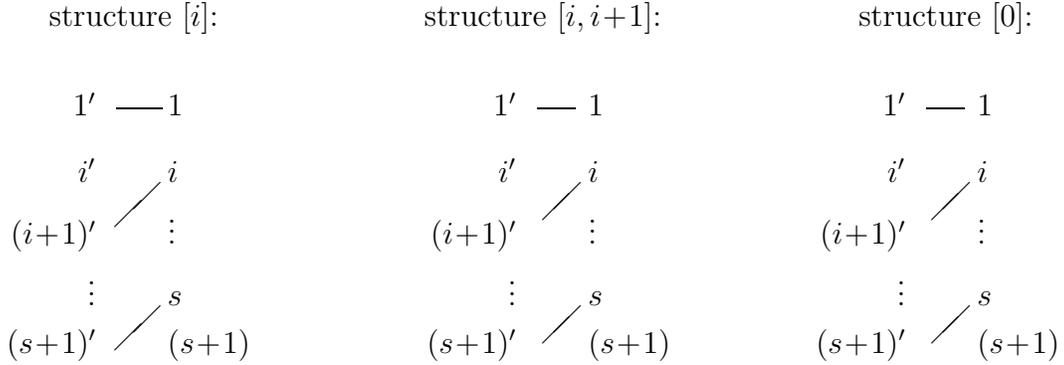
\begin{center}
\begin{tabular}{rclcrclcrcl}
&\begin{picture}(0,0)\hspace{-1.1cm}structure \([i]\):\end{picture}&&&
&\begin{picture}(0,0)\hspace{-1.7cm}structure \([i,i\!+\!1]\):\end{picture}&&&
&\begin{picture}(0,0)\hspace{-1.1cm}structure \([0]\):\end{picture}&\\[20pt]
\(1'\)&\begin{picture}(0,0)\put(-0.4,0.2){\line(1,0){1.2}}\end{picture}&1&&
\(1'\)&\begin{picture}(0,0)\put(-0.4,0.2){\line(1,0){1}}\end{picture}&1&&
\(1'\)&\begin{picture}(0,0)\put(-0.4,0.2){\line(1,0){1}}\end{picture}&1\\[11pt]
\(i'\)&\begin{picture}(0.5,0)\put(1,0){\line(-1,-1){1.2}}\end{picture}&\(i\)&\(\quad\quad\quad\quad\)&
\(i'\)&\begin{picture}(0.5,0)\put(1,0){\line(-1,-1){1}}\end{picture}&\(i\)&\(\quad\quad\quad\)&
\(i'\)&\begin{picture}(0.5,0)\put(1,0){\line(-1,-1){1}}\end{picture}&\(i\)\\[4pt]
\((i\!+\!1)'\)&&\vdots&&
\((i\!+\!1)'\)&&\vdots&&
\((i\!+\!1)'\)&&\vdots\\[4pt]
\vdots&\begin{picture}(0.5,0)\put(1,0){\line(-1,-1){1.2}}\end{picture}&\(s\)&&
\vdots&\begin{picture}(0.5,0)\put(1,0){\line(-1,-1){1}}\end{picture}&\(s\)&&
\vdots&\begin{picture}(0.5,0)\put(1,0){\line(-1,-1){1}}\end{picture}&\(s\)\\[4pt]
\((s\!+\!1)'\)&&\((s\!+\!1)\)&&
\((s\!+\!1)'\)&&\((s\!+\!1)\)&&
\((s\!+\!1)'\)&&\((s\!+\!1)\)
\vspace{22pt}
\end{tabular}
\end{center}
For odd \(r\) we define on the graph \(H_r(p,q)\)
\begin{itemize}
\item[(1)] a structure \([i]\), for \(1\le i\le s\!+\!1\),
\item[(2)] a structure \([i,i\!+\!1]\) for \(1\le i\le s\!+\!1\) and
\item[(3)] a structure \([0]\)
\end{itemize}
if and only if the following connections are given:
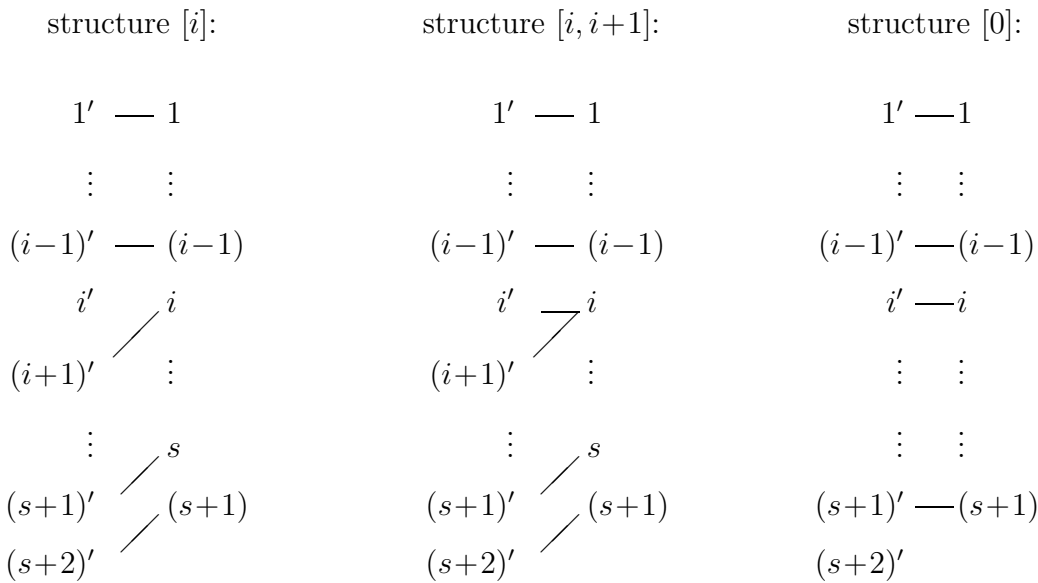
\begin{center}
\begin{tabular}{rclcrclcrcl}
&\begin{picture}(0,0)\hspace{-1.1cm}structure \([i]\):\end{picture}&&&
&\begin{picture}(0,0)\hspace{-1.7cm}structure \([i,i\!+\!1]\):\end{picture}&&&
&\begin{picture}(0,0)\hspace{-1.1cm}structure \([0]\):\end{picture}&\\[20pt]
\(1'\)&\begin{picture}(0,0)\put(-0.4,0.2){\line(1,0){1}}\end{picture}&1&&
\(1'\)&\begin{picture}(0,0)\put(-0.4,0.2){\line(1,0){1}}\end{picture}&1&&
\(1'\)&\begin{picture}(0,0)\put(-0.4,0.2){\line(1,0){1}}\end{picture}&1\\[8pt]
\vdots&&\vdots&&
\vdots&&\vdots&&
\vdots&&\vdots\\[8pt]
\((i\!-\!1)'\)&\begin{picture}(0,0)\put(-0.4,0.2){\line(1,0){1}}\end{picture}&\((i\!-\!1)\)&&
\((i\!-\!1)'\)&\begin{picture}(0,0)\put(-0.4,0.2){\line(1,0){1}}\end{picture}&\((i\!-\!1)\)&&
\((i\!-\!1)'\)&\begin{picture}(0,0)\put(-0.4,0.2){\line(1,0){1}}\end{picture}&\((i\!-\!1)\)\\[8pt]
\(i'\)&\begin{picture}(0.5,0)\put(1,0){\line(-1,-1){1.2}}\end{picture}&\(i\)&\(\quad\quad\quad\quad\)&
\(i'\)&\begin{picture}(0.5,0)\put(1,0){\line(-1,-1){1.2}}\put(1,0){\line(-1,0){1}}\end{picture}&\(i\)&\(\quad\quad\quad\)&
\(i'\)&\begin{picture}(0,0)\put(-0.4,0.2){\line(1,0){1}}\end{picture}&\(i\)\\[8pt]
\((i\!+\!1)'\)&&\vdots&&
\((i\!+\!1)'\)&&\vdots&&
\vdots&&\vdots\\[8pt]
\vdots&\begin{picture}(0.5,0)\put(1,0){\line(-1,-1){1}}\end{picture}&\(s\)&&
\vdots&\begin{picture}(0.5,0)\put(1,0){\line(-1,-1){1}}\end{picture}&\(s\)&&
\vdots&&\vdots\\[8pt]
\((s\!+\!1)'\)&\begin{picture}(0.5,0)\put(1,0){\line(-1,-1){1}}\end{picture}&\((s\!+\!1)\)&&
\((s\!+\!1)'\)&\begin{picture}(0.5,0)\put(1,0){\line(-1,-1){1}}\end{picture}&\((s\!+\!1)\)&&
\((s\!+\!1)'\)&\begin{picture}(0,0)\put(-0.4,0.2){\line(1,0){1}}\end{picture}&\((s\!+\!1)\)\\[8pt]
\((s\!+\!2)'\)&&&&
\((s\!+\!2)'\)&&&&
\((s\!+\!2)'\)&&
\vspace{22pt}
\end{tabular}
\end{center} 
\end{defn}

We note again that the schemes above also give information which points are \emph{not} connected in \(H_r(p,q)\):
Two points that are mentioned in the scheme and that are not connected there should not be connected in the graph \(H_r(p,q)\).
Conversely, the schemes say nothing about connection to points that are not mentioned.
Hence, in structure \([i]\), for example, the point \(i'\) is not allowed to be connected to any of the other drawn points but there are no conditions concerning links to points that are not explicitly mentioned.

\begin{rem}
\begin{itemize}
\item [(a)]
The structures described in Definition \ref{defn:structures[i]_a.s.o.} are incompatible, i.e. no Graph \(H_r(p,q)\) as above can be of more than one of the described structures.
\item[(b)]
Note that this definition is different from the corresponding structures defined in \cite{tutte}. 
There was an overlap in the structures defined there, resulting in false statements in the sequel.
\end{itemize}
\end{rem}
In virtue of the previous examples, the graphs as described in Definition \ref{defn:structures[i]_a.s.o.} look as follows:\vspace{10pt}\newline
For \(r\!=\!2s\!+\!1\) and structure \([i]\):\vspace{10pt}
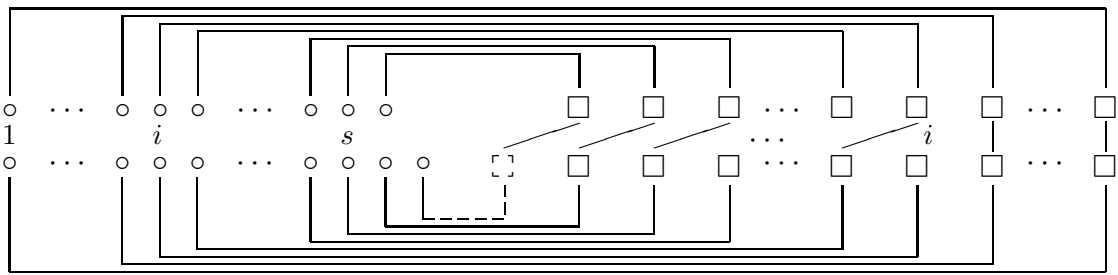
\[
\setlength{\unitlength}{0.5cm}
\begin{picture}(26,2.4)
\put(0,0) {$\circ$}
\put(1,0) {$\;\cdots$}
\put(3,0) {$\circ$}
\put(4,0) {$\circ$}
\put(5,0) {$\circ$}
\put(6,0) {$\;\cdots$}
\put(8,0) {$\circ$}
\put(9,0) {$\circ$}
\put(10,0) {$\circ$}
\put(15,0) {$\square$}
\put(17,0) {$\square$}
\put(19,0) {$\square$}
\put(20,0) {$\;\cdots$}
\put(22,0) {$\square$}
\put(24,0) {$\square$}
\put(26,0) {$\square$}
\put(27,0) {$\;\cdots$}
\put(29,0) {$\square$}
\put(0.2,0.6){\line(0,1){2.3}}
\put(3.2,0.6){\line(0,1){2.1}}
\put(4.2,0.6){\line(0,1){1.9}}
\put(5.2,0.6){\line(0,1){1.7}}
\put(8.2,0.6){\line(0,1){1.5}}
\put(9.2,0.6){\line(0,1){1.3}}
\put(10.2,0.6){\line(0,1){1.1}}
\put(15.35,0.8){\line(0,1){0.9}}
\put(17.35,0.8){\line(0,1){1.1}}
\put(19.35,0.8){\line(0,1){1.3}}
\put(22.35,0.8){\line(0,1){1.5}}
\put(24.35,0.8){\line(0,1){1.7}}
\put(26.35,0.8){\line(0,1){1.9}}
\put(29.35,0.8){\line(0,1){2.1}}
\put(26.35,-0.9){\line(0,1){0.8}}
\put(29.35,-0.9){\line(0,1){0.8}}
\put(13.35,-0.8){\line(3,1){2}}
\put(15.35,-0.8){\line(3,1){2}}
\put(17.35,-0.8){\line(3,1){2}}
\put(19.85,-0.8){$\cdots$}
\put(22.35,-0.8){\line(3,1){2}}
\put(0.,-0.7) {\small$1$}
\put(4,-0.7) {\small$i$}
\put(9,-0.7) {\small$s$}
\put(24.5,-0.7) {\small$i$}
\put(10.2,1.7){\line(1,0){5.15}}
\put(9.2,1.9){\line(1,0){8.15}}
\put(8.2,2.1){\line(1,0){11.15}}
\put(5.2,2.3){\line(1,0){17.15}}
\put(4.2,2.5){\line(1,0){20.15}}
\put(3.2,2.7){\line(1,0){23.15}}
\put(0.2,2.9){\line(1,0){29.15}}
\end{picture}
\hspace{-13.14cm}
\raisebox{-0.5cm}{\reflectbox{\scalebox{-1}{
\begin{picture}(26,2.4)
\put(0,0) {$\circ$}
\put(1,0) {$\;\cdots$}
\put(3,0) {$\circ$}
\put(4,0) {$\circ$}
\put(5,0) {$\circ$}
\put(6,0) {$\;\cdots$}
\put(8,0) {$\circ$}
\put(9,0) {$\circ$}
\put(10,0) {$\circ$}
\put(11,0) {$\circ$}
\put(13,0) {$\square$}
\put(13,0.32) {\linethickness{0.1cm}\color{white}\line(1,0){0.8}}
\put(13.34,0) {\linethickness{0.1cm}\color{white}\line(0,1){0.8}}
\put(15,0) {$\square$}
\put(17,0) {$\square$}
\put(19,0) {$\square$}
\put(20,0) {$\;\cdots$}
\put(22,0) {$\square$}
\put(24,0) {$\square$}
\put(26,0) {$\square$}
\put(27,0) {$\;\cdots$}
\put(29,0) {$\square$}
\put(0.2,0.6){\line(0,1){2.5}}
\put(3.2,0.6){\line(0,1){2.3}}
\put(4.2,0.6){\line(0,1){2.1}}
\put(5.2,0.6){\line(0,1){1.9}}
\put(8.2,0.6){\line(0,1){1.7}}
\put(9.2,0.6){\line(0,1){1.5}}
\put(10.2,0.6){\line(0,1){1.3}}
\put(11.2,0.6){\line(0,1){1.1}}
\multiput(13.38,0.8)(0,0.46){2}{\line(0,1){0.3}}
\put(15.35,0.8){\line(0,1){1.1}}
\put(17.35,0.8){\line(0,1){1.3}}
\put(19.35,0.8){\line(0,1){1.5}}
\put(22.35,0.8){\line(0,1){1.7}}
\put(24.35,0.8){\line(0,1){1.9}}
\put(26.35,0.8){\line(0,1){2.1}}
\put(29.35,0.8){\line(0,1){2.3}}
\multiput(11.2,1.7)(0.46,0){5}{\line(1,0){0.3}}
\put(10.2,1.9){\line(1,0){5.15}}
\put(9.2,2.1){\line(1,0){8.15}}
\put(8.2,2.3){\line(1,0){11.15}}
\put(5.2,2.5){\line(1,0){17.15}}
\put(4.2,2.7){\line(1,0){20.15}}
\put(3.2,2.9){\line(1,0){23.15}}
\put(0.2,3.1){\line(1,0){29.15}}
\end{picture}
}}}
\]\vspace{10pt}\newline
For \(r\!=\!2s\!+\!1\) and structure \([i,i+1]\):\vspace{10pt}
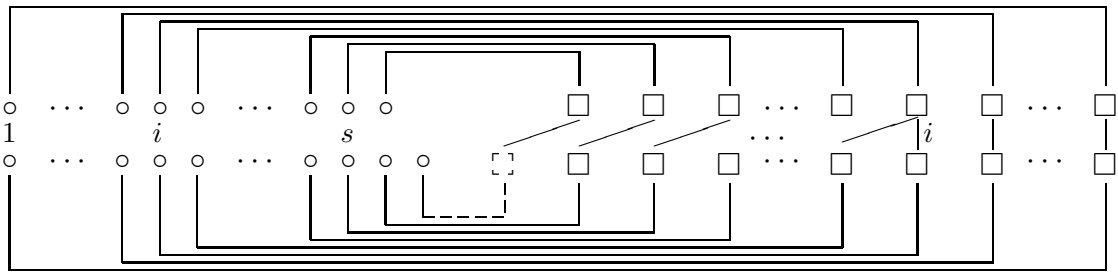
\[
\setlength{\unitlength}{0.5cm}
\begin{picture}(26,2.4)
\put(0,0) {$\circ$}
\put(1,0) {$\;\cdots$}
\put(3,0) {$\circ$}
\put(4,0) {$\circ$}
\put(5,0) {$\circ$}
\put(6,0) {$\;\cdots$}
\put(8,0) {$\circ$}
\put(9,0) {$\circ$}
\put(10,0) {$\circ$}
\put(15,0) {$\square$}
\put(17,0) {$\square$}
\put(19,0) {$\square$}
\put(20,0) {$\;\cdots$}
\put(22,0) {$\square$}
\put(24,0) {$\square$}
\put(26,0) {$\square$}
\put(27,0) {$\;\cdots$}
\put(29,0) {$\square$}
\put(0.2,0.6){\line(0,1){2.3}}
\put(3.2,0.6){\line(0,1){2.1}}
\put(4.2,0.6){\line(0,1){1.9}}
\put(5.2,0.6){\line(0,1){1.7}}
\put(8.2,0.6){\line(0,1){1.5}}
\put(9.2,0.6){\line(0,1){1.3}}
\put(10.2,0.6){\line(0,1){1.1}}
\put(15.35,0.8){\line(0,1){0.9}}
\put(17.35,0.8){\line(0,1){1.1}}
\put(19.35,0.8){\line(0,1){1.3}}
\put(22.35,0.8){\line(0,1){1.5}}
\put(24.35,0.8){\line(0,1){1.7}}
\put(26.35,0.8){\line(0,1){1.9}}
\put(29.35,0.8){\line(0,1){2.1}}
\put(24.35,-0.9){\line(0,1){0.8}}
\put(26.35,-0.9){\line(0,1){0.8}}
\put(29.35,-0.9){\line(0,1){0.8}}
\put(13.35,-0.8){\line(3,1){2}}
\put(15.35,-0.8){\line(3,1){2}}
\put(17.35,-0.8){\line(3,1){2}}
\put(19.85,-0.8){$\cdots$}
\put(22.35,-0.7){\line(3,1){2}}
\put(0.,-0.7) {\small$1$}
\put(4,-0.7) {\small$i$}
\put(9,-0.7) {\small$s$}
\put(24.5,-0.7) {\small$i$}
\put(10.2,1.7){\line(1,0){5.15}}
\put(9.2,1.9){\line(1,0){8.15}}
\put(8.2,2.1){\line(1,0){11.15}}
\put(5.2,2.3){\line(1,0){17.15}}
\put(4.2,2.5){\line(1,0){20.15}}
\put(3.2,2.7){\line(1,0){23.15}}
\put(0.2,2.9){\line(1,0){29.15}}
\end{picture}
\hspace{-13.14cm}
\raisebox{-0.5cm}{\reflectbox{\scalebox{-1}{
\begin{picture}(26,2.4)
\put(0,0) {$\circ$}
\put(1,0) {$\;\cdots$}
\put(3,0) {$\circ$}
\put(4,0) {$\circ$}
\put(5,0) {$\circ$}
\put(6,0) {$\;\cdots$}
\put(8,0) {$\circ$}
\put(9,0) {$\circ$}
\put(10,0) {$\circ$}
\put(11,0) {$\circ$}
\put(13,0) {$\square$}
\put(13,0.32) {\linethickness{0.1cm}\color{white}\line(1,0){0.8}}
\put(13.34,0) {\linethickness{0.1cm}\color{white}\line(0,1){0.8}}
\put(15,0) {$\square$}
\put(17,0) {$\square$}
\put(19,0) {$\square$}
\put(20,0) {$\;\cdots$}
\put(22,0) {$\square$}
\put(24,0) {$\square$}
\put(26,0) {$\square$}
\put(27,0) {$\;\cdots$}
\put(29,0) {$\square$}
\put(0.2,0.6){\line(0,1){2.5}}
\put(3.2,0.6){\line(0,1){2.3}}
\put(4.2,0.6){\line(0,1){2.1}}
\put(5.2,0.6){\line(0,1){1.9}}
\put(8.2,0.6){\line(0,1){1.7}}
\put(9.2,0.6){\line(0,1){1.5}}
\put(10.2,0.6){\line(0,1){1.3}}
\put(11.2,0.6){\line(0,1){1.1}}
\multiput(13.38,0.8)(0,0.46){2}{\line(0,1){0.3}}
\put(15.35,0.8){\line(0,1){1.1}}
\put(17.35,0.8){\line(0,1){1.3}}
\put(19.35,0.8){\line(0,1){1.5}}
\put(22.35,0.8){\line(0,1){1.7}}
\put(24.35,0.8){\line(0,1){1.9}}
\put(26.35,0.8){\line(0,1){2.1}}
\put(29.35,0.8){\line(0,1){2.3}}
\multiput(11.2,1.7)(0.46,0){5}{\line(1,0){0.3}}
\put(10.2,1.9){\line(1,0){5.15}}
\put(9.2,2.1){\line(1,0){8.15}}
\put(8.2,2.3){\line(1,0){11.15}}
\put(5.2,2.5){\line(1,0){17.15}}
\put(4.2,2.7){\line(1,0){20.15}}
\put(3.2,2.9){\line(1,0){23.15}}
\put(0.2,3.1){\line(1,0){29.15}}
\end{picture}
}}}
\]\vspace{10pt}\newline
For \(r\!=\!2s\) and structure \([i]\):\vspace{10pt}
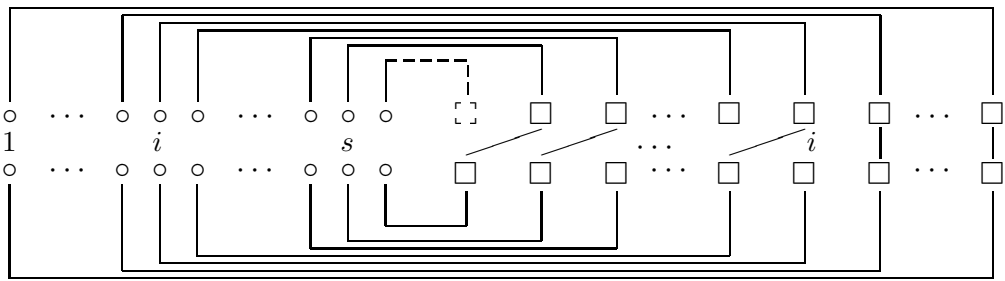
\[
\setlength{\unitlength}{0.5cm}
\begin{picture}(26,2.4)
\put(0,0) {$\circ$}
\put(1,0) {$\;\cdots$}
\put(3,0) {$\circ$}
\put(4,0) {$\circ$}
\put(5,0) {$\circ$}
\put(6,0) {$\;\cdots$}
\put(8,0) {$\circ$}
\put(9,0) {$\circ$}
\put(10,0) {$\circ$}
\put(12,0) {$\square$}
\put(12,0.3) {\linethickness{0.1cm}\color{white}\line(1,0){0.8}}
\put(12.33,0) {\linethickness{0.1cm}\color{white}\line(0,1){0.8}}
\put(14,0) {$\square$}
\put(16,0) {$\square$}
\put(17,0) {$\;\cdots$}
\put(19,0) {$\square$}
\put(21,0) {$\square$}
\put(23,0) {$\square$}
\put(24,0) {$\;\cdots$}
\put(26,0) {$\square$}
\put(23.35,-0.9){\line(0,1){0.8}}
\put(26.35,-0.9){\line(0,1){0.8}}
\put(12.35,-0.8){\line(3,1){2}}
\put(14.35,-0.8){\line(3,1){2}}
\put(16.85,-0.8){$\cdots$}
\put(19.35,-0.8){\line(3,1){2}}
\put(0.,-0.7) {\small$1$}
\put(4,-0.7) {\small$i$}
\put(9,-0.7) {\small$s$}
\put(21.4,-0.7) {\small$i$}
\put(0.2,0.6){\line(0,1){2.5}}
\put(3.2,0.6){\line(0,1){2.3}}
\put(4.2,0.6){\line(0,1){2.1}}
\put(5.2,0.6){\line(0,1){1.9}}
\put(8.2,0.6){\line(0,1){1.7}}
\put(9.2,0.6){\line(0,1){1.5}}
\put(10.2,0.6){\line(0,1){1.1}}
\multiput(12.38,0.8)(0,0.46){2}{\line(0,1){0.3}}
\put(14.35,0.8){\line(0,1){1.3}}
\put(16.35,0.8){\line(0,1){1.5}}
\put(19.35,0.8){\line(0,1){1.7}}
\put(21.35,0.8){\line(0,1){1.9}}
\put(23.35,0.8){\line(0,1){2.1}}
\put(26.35,0.8){\line(0,1){2.3}}
\multiput(10.2,1.7)(0.46,0){5}{\line(1,0){0.3}}
\put(9.2,2.1){\line(1,0){5.15}}
\put(8.2,2.3){\line(1,0){8.15}}
\put(5.2,2.5){\line(1,0){14.15}}
\put(4.2,2.7){\line(1,0){17.15}}
\put(3.2,2.9){\line(1,0){20.15}}
\put(0.2,3.1){\line(1,0){26.15}}
\end{picture}
\hspace{-13.14cm}
\raisebox{-0.5cm}{\reflectbox{\scalebox{-1}{
\begin{picture}(26,2.4)
\put(0,0) {$\circ$}
\put(1,0) {$\;\cdots$}
\put(3,0) {$\circ$}
\put(4,0) {$\circ$}
\put(5,0) {$\circ$}
\put(6,0) {$\;\cdots$}
\put(8,0) {$\circ$}
\put(9,0) {$\circ$}
\put(10,0) {$\circ$}
\put(12,0) {$\square$}
\put(14,0) {$\square$}
\put(16,0) {$\square$}
\put(17,0) {$\;\cdots$}
\put(19,0) {$\square$}
\put(21,0) {$\square$}
\put(23,0) {$\square$}
\put(24,0) {$\;\cdots$}
\put(26,0) {$\square$}
\put(0.2,0.6){\line(0,1){2.5}}
\put(3.2,0.6){\line(0,1){2.3}}
\put(4.2,0.6){\line(0,1){2.1}}
\put(5.2,0.6){\line(0,1){1.9}}
\put(8.2,0.6){\line(0,1){1.7}}
\put(9.2,0.6){\line(0,1){1.5}}
\put(10.2,0.6){\line(0,1){1.1}}
\put(12.35,0.8){\line(0,1){0.9}}
\put(14.35,0.8){\line(0,1){1.3}}
\put(16.35,0.8){\line(0,1){1.5}}
\put(19.35,0.8){\line(0,1){1.7}}
\put(21.35,0.8){\line(0,1){1.9}}
\put(23.35,0.8){\line(0,1){2.1}}
\put(26.35,0.8){\line(0,1){2.3}}
\put(10.2,1.7){\line(1,0){2.15}}
\put(9.2,2.1){\line(1,0){5.15}}
\put(8.2,2.3){\line(1,0){8.15}}
\put(5.2,2.5){\line(1,0){14.15}}
\put(4.2,2.7){\line(1,0){17.15}}
\put(3.2,2.9){\line(1,0){20.15}}
\put(0.2,3.1){\line(1,0){26.15}}
\end{picture}
}}}
\]\vspace{10pt}\newline
For \(r\!=\!2s\) and structure \([i,i+1]\):\vspace{10pt}
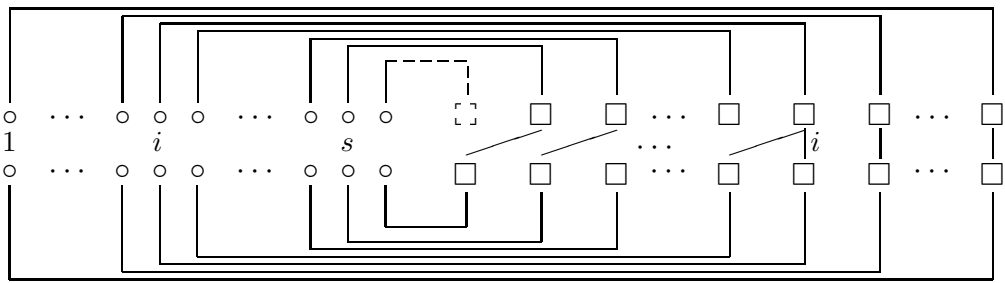
\[
\setlength{\unitlength}{0.5cm}
\begin{picture}(26,2.4)
\put(0,0) {$\circ$}
\put(1,0) {$\;\cdots$}
\put(3,0) {$\circ$}
\put(4,0) {$\circ$}
\put(5,0) {$\circ$}
\put(6,0) {$\;\cdots$}
\put(8,0) {$\circ$}
\put(9,0) {$\circ$}
\put(10,0) {$\circ$}
\put(12,0) {$\square$}
\put(12,0.3) {\linethickness{0.1cm}\color{white}\line(1,0){0.8}}
\put(12.33,0) {\linethickness{0.1cm}\color{white}\line(0,1){0.8}}
\put(14,0) {$\square$}
\put(16,0) {$\square$}
\put(17,0) {$\;\cdots$}
\put(19,0) {$\square$}
\put(21,0) {$\square$}
\put(23,0) {$\square$}
\put(24,0) {$\;\cdots$}
\put(26,0) {$\square$}
\put(21.35,-0.9){\line(0,1){0.8}}
\put(23.35,-0.9){\line(0,1){0.8}}
\put(26.35,-0.9){\line(0,1){0.8}}
\put(12.35,-0.8){\line(3,1){2}}
\put(14.35,-0.8){\line(3,1){2}}
\put(16.85,-0.8){$\cdots$}
\put(19.35,-0.8){\line(3,1){2}}
\put(0.,-0.7) {\small$1$}
\put(4,-0.7) {\small$i$}
\put(9,-0.7) {\small$s$}
\put(21.5,-0.7) {\small$i$}
\put(0.2,0.6){\line(0,1){2.5}}
\put(3.2,0.6){\line(0,1){2.3}}
\put(4.2,0.6){\line(0,1){2.1}}
\put(5.2,0.6){\line(0,1){1.9}}
\put(8.2,0.6){\line(0,1){1.7}}
\put(9.2,0.6){\line(0,1){1.5}}
\put(10.2,0.6){\line(0,1){1.1}}
\multiput(12.38,0.8)(0,0.46){2}{\line(0,1){0.3}}
\put(14.35,0.8){\line(0,1){1.3}}
\put(16.35,0.8){\line(0,1){1.5}}
\put(19.35,0.8){\line(0,1){1.7}}
\put(21.35,0.8){\line(0,1){1.9}}
\put(23.35,0.8){\line(0,1){2.1}}
\put(26.35,0.8){\line(0,1){2.3}}
\multiput(10.2,1.7)(0.46,0){5}{\line(1,0){0.3}}
\put(9.2,2.1){\line(1,0){5.15}}
\put(8.2,2.3){\line(1,0){8.15}}
\put(5.2,2.5){\line(1,0){14.15}}
\put(4.2,2.7){\line(1,0){17.15}}
\put(3.2,2.9){\line(1,0){20.15}}
\put(0.2,3.1){\line(1,0){26.15}}
\end{picture}
\hspace{-13.14cm}
\raisebox{-0.5cm}{\reflectbox{\scalebox{-1}{
\begin{picture}(26,2.4)
\put(0,0) {$\circ$}
\put(1,0) {$\;\cdots$}
\put(3,0) {$\circ$}
\put(4,0) {$\circ$}
\put(5,0) {$\circ$}
\put(6,0) {$\;\cdots$}
\put(8,0) {$\circ$}
\put(9,0) {$\circ$}
\put(10,0) {$\circ$}
\put(12,0) {$\square$}
\put(14,0) {$\square$}
\put(16,0) {$\square$}
\put(17,0) {$\;\cdots$}
\put(19,0) {$\square$}
\put(21,0) {$\square$}
\put(23,0) {$\square$}
\put(24,0) {$\;\cdots$}
\put(26,0) {$\square$}
\put(0.2,0.6){\line(0,1){2.5}}
\put(3.2,0.6){\line(0,1){2.3}}
\put(4.2,0.6){\line(0,1){2.1}}
\put(5.2,0.6){\line(0,1){1.9}}
\put(8.2,0.6){\line(0,1){1.7}}
\put(9.2,0.6){\line(0,1){1.5}}
\put(10.2,0.6){\line(0,1){1.1}}
\put(12.35,0.8){\line(0,1){0.9}}
\put(14.35,0.8){\line(0,1){1.3}}
\put(16.35,0.8){\line(0,1){1.5}}
\put(19.35,0.8){\line(0,1){1.7}}
\put(21.35,0.8){\line(0,1){1.9}}
\put(23.35,0.8){\line(0,1){2.1}}
\put(26.35,0.8){\line(0,1){2.3}}
\put(10.2,1.7){\line(1,0){2.15}}
\put(9.2,2.1){\line(1,0){5.15}}
\put(8.2,2.3){\line(1,0){8.15}}
\put(5.2,2.5){\line(1,0){14.15}}
\put(4.2,2.7){\line(1,0){17.15}}
\put(3.2,2.9){\line(1,0){20.15}}
\put(0.2,3.1){\line(1,0){26.15}}
\end{picture}
}}}
\]\vspace{10pt}\newline
Note that the pictures above are a little bit imprecise with respect to the meaning of a line between to squares \(\square\).
Such a line does not mean that there is just \emph{any} connection between these structures but there is a connection between the corresponding \(K(i)\) and \(K(j')\).
For example, the vertical line between the right-most squares says that the points \(1\) and \(1'\) are connected.
\newline
Note further that in the first two pictures the diagonal line to the dashed square  just means that \(s+1\) is in the same block as \((s+2)'\), so if the dashed structure is empty then this diagonal line has to end directly at point \((s+2)'\).
\vspace{11pt}\newline
In the next lemma we take a closer look at the matrix \(A(n,r)\) and ask under which conditions certain entries are non-zero.
Recall, see Remark \ref{lem:f(i,q)_and_g(i,q)_are_in_W(n,r)}, that for \(q\in W(n,r\!+\!1)\) the partitions \(f(i,q)\) and \(g(i,q)\) are elements in \(W(n,r)\).
Hence, there are rows and columns in \(A(n,r)\) labelled by \(f(i,q)\) and \(g(i,q)\).
\newline
Note further, compare Definition \ref{defn:r-flawless_graph}, that a graph \(G(p,q)\) is called \(r\)-flawless if the Graph \(H_r(p,q)\) has the following structure:
\begin{equation}\label{eqn:schemes_for_r-flawlessness}
\begin{tabular}{rclcrcl}
&\begin{picture}(0,0)\hspace{-0.5cm}\(r\) even:\end{picture}&&&&\begin{picture}(0,0)\hspace{-0.5cm}\(r\) odd:\end{picture}&\\[11pt]
\(1'\)&\begin{picture}(0,0)\put(-0.3,0.2){\line(1,0){0.9}}\end{picture}&\(1\)&\(\quad\quad\quad\quad\quad\)&\(1'\)&\begin{picture}(0,0)\put(-0.3,0.2){\line(1,0){0.9}}\end{picture}&\(1\)\\[4pt]
\vdots&&\vdots&&\vdots&&\vdots\\
\(s'\)&\begin{picture}(0,0)\put(-0.3,0.2){\line(1,0){0.9}}\end{picture}&\(s\)&&\(s'\)&\begin{picture}(0,0)\put(-0.3,0.2){\line(1,0){0.9}}\end{picture}&\(s\)\\[8pt]
\((s\!+\!1)'\)&\begin{picture}(0,0)\multiput(-0.3,0.2)(0.35,0){3}{\line(1,0){0.2}}\end{picture}&\((s\!+\!1)\)&&\((s\!+\!1)'\)&\begin{picture}(0,0)\put(-0.3,0.2){\line(1,0){0.9}}\end{picture}&\((s\!+\!1)\)
\end{tabular}
\end{equation}
The dashed line between \((s\!+\!1)'\) and \((s\!+\!1)\) means that a connection between these two points is allowed but not necessary.
The following observation will be used in the proof of Proposition \ref{prop:when_is_e_r(p,q)_non-zero?}.
\begin{obs}\label{obs:important_observation_about_drawing_f(i,q)}
Consider \(q\in W(n,r\!+\!1)\) and the illustrations of \(q^*\) and \(f(i,q^*)\):
{\allowdisplaybreaks
\vspace{-11pt}
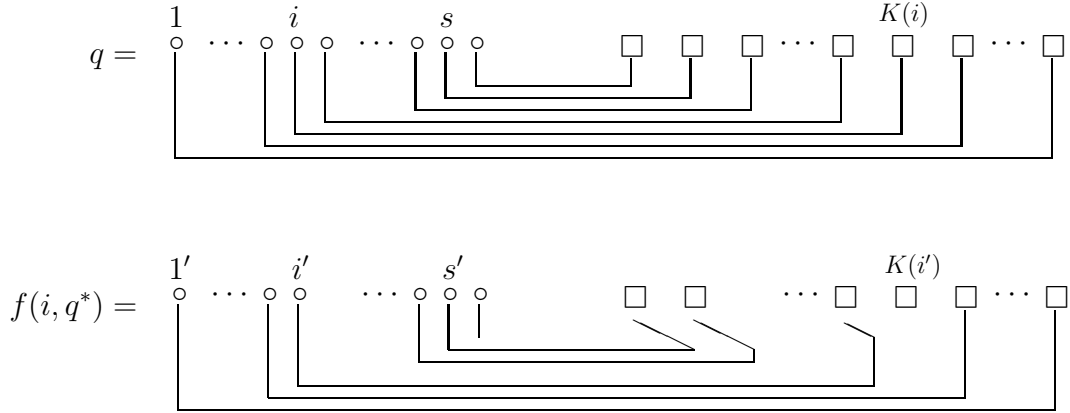
\begin{align*}
q=&\quad
\begin{picture}(28,2.4)\setlength{\unitlength}{0.4cm}
\put(0,1.2) {$1$}
\put(4,1.2) {$i$}
\put(23.6,1.4) {\scalebox{0.8}{$K(i)$}}
\put(9,1.2) {$s$}
\end{picture}
\hspace{-14cm}
\reflectbox{\scalebox{-1}{\parbox[c]{10cm}{
\setlength{\unitlength}{0.4cm}
\begin{picture}(28,2.4)
\put(0,0) {$\circ$}
\put(1,0) {$\;\cdots$}
\put(3,0) {$\circ$}
\put(4,0) {$\circ$}
\put(5,0) {$\circ$}
\put(6,0) {$\;\cdots$}
\put(8,0) {$\circ$}
\put(9,0) {$\circ$}
\put(10,0) {$\circ$}
\put(15,0) {$\square$}
\put(17,0) {$\square$}
\put(19,0) {$\square$}
\put(20,0) {$\;\cdots$}
\put(22,0) {$\square$}
\put(24,0) {$\square$}
\put(26,0) {$\square$}
\put(27,0) {$\;\cdots$}
\put(29,0) {$\square$}
\put(0.2,0.6){\line(0,1){3.5}}
\put(3.2,0.6){\line(0,1){3.1}}
\put(4.2,0.6){\line(0,1){2.7}}
\put(5.2,0.6){\line(0,1){2.3}}
\put(8.2,0.6){\line(0,1){1.9}}
\put(9.2,0.6){\line(0,1){1.5}}
\put(10.2,0.6){\line(0,1){1.1}}
\put(15.35,0.8){\line(0,1){0.9}}
\put(17.35,0.8){\line(0,1){1.3}}
\put(19.35,0.8){\line(0,1){1.7}}
\put(22.35,0.8){\line(0,1){2.1}}
\put(24.35,0.8){\line(0,1){2.5}}
\put(26.35,0.8){\line(0,1){2.9}}
\put(29.35,0.8){\line(0,1){3.3}}
\put(10.2,1.7){\line(1,0){5.15}}
\put(9.2,2.1){\line(1,0){8.15}}
\put(8.2,2.5){\line(1,0){11.15}}
\put(5.2,2.9){\line(1,0){17.15}}
\put(4.2,3.3){\line(1,0){20.15}}
\put(3.2,3.7){\line(1,0){23.15}}
\put(0.2,4.1){\line(1,0){29.15}}
\end{picture}}}}\quad\quad\quad\quad\quad\;\;
\\[40pt]
f(i,q^*)=&\quad
\begin{picture}(28,2.4)\setlength{\unitlength}{0.4cm}
\put(0,1.2) {$1'$}
\put(4.1,1.2) {$i'$}
\put(23.8,1.4) {\scalebox{0.8}{$K(i')$}}
\put(9.1,1.2) {$s'$}
\end{picture}
\hspace{-14.1cm}
\reflectbox{\scalebox{-1}{
\parbox[c]{10cm}{
\setlength{\unitlength}{0.4cm}
\begin{picture}(28,2.4)
\put(0,0) {$\circ$}
\put(1,0) {$\;\cdots$}
\put(3,0) {$\circ$}
\put(4,0) {$\circ$}
\put(6,0) {$\;\cdots$}
\put(8,0) {$\circ$}
\put(9,0) {$\circ$}
\put(10,0) {$\circ$}
\put(15,0) {$\square$}
\put(17,0) {$\square$}
\put(20,0) {$\;\cdots$}
\put(22,0) {$\square$}
\put(24,0) {$\square$}
\put(26,0) {$\square$}
\put(27,0) {$\;\cdots$}
\put(29,0) {$\square$}
\put(0.2,0.6){\line(0,1){3.5}}
\put(3.2,0.6){\line(0,1){3.1}}
\put(4.2,0.6){\line(0,1){2.7}}
\put(8.2,0.6){\line(0,1){1.9}}
\put(9.2,0.6){\line(0,1){1.5}}
\put(10.2,0.6){\line(0,1){1.1}}
\put(26.35,0.8){\line(0,1){2.9}}
\put(29.35,0.8){\line(0,1){3.3}}
\put(9.2,2.1){\line(1,0){8.15}}
\put(8.2,2.5){\line(1,0){11.15}}
\put(4.2,3.3){\line(1,0){19.15}}
\put(3.2,3.7){\line(1,0){23.15}}
\put(0.2,4.1){\line(1,0){29.15}}
\put(17.35,2.1){\line(-2,-1){2}}
\put(19.35,2.1){\line(-2,-1){2}}
\put(23.35,1.7){\line(-2,-1){1}}
\put(19.35,2.1){\line(0,1){0.4}}
\put(23.35,1.7){\line(0,1){1.6}}
\end{picture}}}}\quad\quad\quad\quad\quad\;\;
\end{align*}
}
\vspace{22pt}
\newline
We observe that in the illustration of \(f(i,q^*)\) the space below \(K(i')\) is only crossed by the lines which connect \(j'\) and \(K(j')\) for \(1\!\le\! j\!<\!i\).

\end{obs}
\begin{prop}\label{prop:when_is_e_r(p,q)_non-zero?}
Consider \(1\!\le\!r\!<\!n\!-\!1\) and partitions \(p\in W(n,r)\) and \newline \(q\in W(n,r\!+\!1)\).
\begin{itemize}
\item[(i)] The entry \(e_r(p,q)\) of \(A(n,r)\) is non-zero exactly in the following cases:
	\begin{itemize}
		\item[(1)] If \(r=2s\), then \(H_r(p,q)\) must have structure \([s+1]\) or \([0]\).
		\item[(2)] If \(r=2s+1\), then \(H_r(p,q)\) must have structure \([s+1,s+2]\) or \([0]\).
	\end{itemize}
\item[(ii)] For \(1\le i\le s+1\) the entry \(e_r\big(p,f(i,q)\big)\) in \(A(n,r)\) is non-zero if and only if \(H_r(p,q)\) is of structure \([i]\), \([i,i+1]\) or \([i-1,i]\).
Note that \([i-1,i]\) is excluded if \(i=1\) and \([i,i+1]\) is excluded if \(i=s+1\) and \(r=2s\), so in these cases there are only two possible structures if \(e_r\big(p,g(i,q)\big)\) should be non-zero.
\item[(iii)] For \(1\!\le\!i\le\!s\) the entry \(e_r\big(p,g(i,q)\big)\) in \(A(n,r)\) is non-zero if and only if \(H_r(p,q)\) has structure \([i]\), \([i\!+\!1]\) or \([i,i\!+\!1]\).
If \(i\!=\!s\!+\!1\) (only allowed if \(r\!=\!2s\!+\!1\)), then \(H_r(p,q)\) must be of structure \([i]\), \([i,i\!+\!1]\) or \([0]\).
\end{itemize}
\end{prop}
\begin{proof}
By definition of \(e_r(\cdot,\cdot)\), see Definition \ref{defn:A(n,r)}, we have to show that absence of \(r\)-flaws is equivalent to the respectively mentioned structures.
\newline
A detailed proof is lengthy because many different scenarios have to be considered, but each of them can be checked by comparing three kind of schemes:
\begin{itemize}
\item[(1)] The Picture \ref{eqn:schemes_for_r-flawlessness} describes absence of \(r\)-flaws.
\item[(2)] A comparison of the Picture \ref{eqn:scheme_for_the_K(i)}  with the Pictures \ref{eqn:scheme_for_f(i,q)} and \ref{eqn:scheme_for_g(i,q)}  tells us how to construct the partitions \(f(i,q)\) and \(g(i,q)\) from \(q\) (and vice versa).
\item[(3)] The schemes in Definition \ref{defn:structures[i]_a.s.o.}  define and illustrate the structures \([i]\), \([i,i\!+\!1]\) and \([0]\).
\end{itemize}
Going from (1) to (3), one sees that assuming absence of \(r\)-flaws in \(G_r(p,q)\), \(G_r(p,f(i,q))\) or \(G_r(p,g(i,q))\), respectively,  implies that \(H_r (p,q)\) has one of the respectively claimed structures.
Going backwards from (3) to (1) shows the conversion, and hence the desired equivalence.
\newline
To convince the reader, we consider the situation of an \(r\)-flawless \(G\big(p,f(i,q)\big)\) for odd \(r\) and show that \(H_r(p,q)\) must be of structure \([i]\), \([i,i\!+\!1]\) or \([i\!-\!1,i]\).
\newline
Picture \ref{eqn:schemes_for_r-flawlessness} describes the relevant structure of \(H_r\big(p,f(i,q)\big)\) and Picture \ref{eqn:scheme_for_f(i,q)} the relevant structure of \(f(i,q)\).
Combining them, we obtain the following scheme:
%
%
%
%
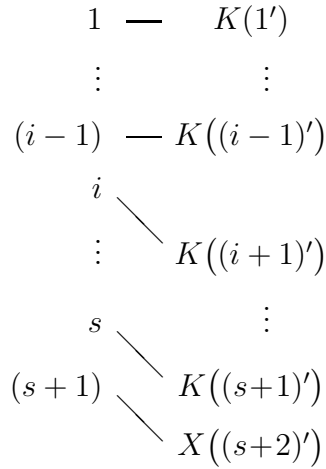
\[
\begin{tabular}{rcccl}
\(1\)&\begin{picture}(0,0)\put(-0.3,0.2){\line(1,0){0.9}}\end{picture}&\(K(1')\)\\[4pt]
\vdots&&\quad\vdots\\[6pt]
\((i-1)\)&\begin{picture}(0,0)\put(-0.3,0.2){\line(1,0){0.9}}\end{picture}&\(K\big((i-1)'\big)\)\\[6pt]
\(i\)&\begin{picture}(0.5,0)\put(-0.3,0){\line(1,-1){1.2}}\end{picture}&\\[6pt]
\(\vdots\)&&\(K\big((i+1)'\big)\)\\[6pt]
\(s\)&\begin{picture}(0.5,0)\put(-0.3,0){\line(1,-1){1.2}}\end{picture}&\(\quad\vdots\)\\[10pt]
\((s+1)\)&\begin{picture}(0.5,0)\put(-0.3,0){\line(1,-1){1.2}}\end{picture}&\(K\big((s\!+\!1)'\big)\)\\[8pt]
&&\(X\big((s\!+\!2)'\big)\)
\end{tabular}
\]
Note that there are no other links between the mentioned (sets of) points because the points \(1,\ldots,(s\!+\!1)\) are pairwisely disconnected by assumption on \(G\big(p,f(i,q)\big)\).
\newline
Replacing \(f(i,q)\) by \(q\), we obtain by definition of the sets \(K(j')\) and \(X\big((s\!+\!2)'\big)\) the following structure for \(H_r(p,q)\):
\begin{equation}\label{eqn:scheme_for_exemplary_proof}
\begin{tabular}{rcccl}
\(1\)&\begin{picture}(0,0)\put(-0.3,0.2){\line(1,0){0.9}}\end{picture}&\(K(1')\)\\[4pt]
\vdots&&\quad\vdots\\[6pt]
\((i-1)\)&\begin{picture}(0,0)\put(-0.3,0.2){\line(1,0){0.9}}\end{picture}&\((i-1)'\)\\[6pt]
\(i\)&\begin{picture}(0.5,0)\put(-0.3,0){\line(1,-1){1.2}}\end{picture}&\\[6pt]
\(\vdots\)&&\((i+1)'\)\\[6pt]
\(s\)&\begin{picture}(0.5,0)\put(-0.3,0){\line(1,-1){1.2}}\end{picture}&\(\quad\vdots\)\\[10pt]
\((s+1)\)&\begin{picture}(0.5,0)\put(-0.3,0){\line(1,-1){1.2}}\end{picture}&\((s\!+\!1)'\)\\[8pt]
&&\((s\!+\!2)'\)
\end{tabular}
\end{equation}
This scheme is compatible with the desired structures, but we have to prove that also the point \(i'\) is connected in the proper way:
\begin{itemize}
\item[(1)] If \(i'\) is connected to none of the points \(1\ldots,(s\!+\!1)\), then we have structure \([i]\).
\item[(2)] If \(i'\) is only connected to \(i\) but to no other point amongst \(1,\ldots,(s\!+\!1)\), then we have structure \([i,i\!+\!1]\).
\item[(3)] If \(i'\) is only connected to \((i\!-\!1)\) but to no other point amongst \(1,\ldots,(s\!+\!1)\), then we have structure \([i,i\!+\!1]\).
\end{itemize}
The prove is finished, once we have shown that no other situations occur.
\newline
It is clear that \(i'\) cannot be connected to more than one point of \(1,\ldots,(s+1)\) as this would contradict Scheme \ref{eqn:scheme_for_exemplary_proof}, so having another structure than the three situations above means that \(i'\) or, equivalently, \(K(i')\) is connected to any of the points \(1,\ldots,(s\!+\!1)\) other than \(i\) or \((i\!-\!1)\). 
We assume this to be true and lead it to a contradiction.
\newline
Drawing the illustration of \(H_r\big(p,f(i,q)\big)\) in a non-crossing way, the relevant part of \(H_r\big(p,f(i,q)\big)\) looks as follows:
\vspace{11pt}
\begin{equation}\label{eqn:relevant_part_of_H_r(p,f(i,q))}
\cdots\quad\quad\quad
\raisebox{0.3cm}{\setlength{\unitlength}{0.5cm}
\begin{picture}(15,2.4)
\put(0,0.5) {$\circ$}
\put(1,0.5) {$\circ$}
\put(4,0.5) {$\cdots$}
\put(10,0.5) {$\square$}
\put(13,0.5) {$\square$}
\put(0.2,1.1){\line(0,1){2.0}}
\put(1.2,1.1){\line(0,1){1.6}}
\put(10.35,1.3){\line(0,1){1.4}}
\put(13.35,1.3){\line(0,1){1.8}}
\put(0.2,3.1){\line(1,0){13.15}}
\put(1.2,2.7){\line(1,0){9.15}}
\put(7.35,-1){\line(3,1){3}}
\put(-1.8,-1.6) {\scalebox{0.7}{\((i-1)'\)}}
\put(-1.8,0.3) {\scalebox{0.7}{\((i-1)\)}}
\put(1.5,-1.6) {\scalebox{0.7}{\((i)'\)}}
\put(1.5,0.3) {\scalebox{0.7}{\((i)\)}}
\put(4.3,-2.4) {\scalebox{0.7}{\(K\big((i\!+\!1)'\big)\)}}
\put(10.6,-2.4) {\scalebox{0.7}{\(K\big((i)'\big)\)}}
\put(10.6,1.1) {\scalebox{0.7}{\(K\big(i\big)\)}}
\put(13.6,-2.4) {\scalebox{0.7}{\(K\big((i\!-\!1)'\big)\)}}
\put(13.6,1.1) {\scalebox{0.7}{\(K\big(i\!-\!1\big)\)}}
\put(13.35,-1.2){\line(0,1){1.3}}
\end{picture}
\hspace{-7.77cm}
\raisebox{-0.5cm}{\reflectbox{\scalebox{-1}{
\begin{picture}(15,2.4)
\put(0,0.5) {$\circ$}
\put(1,0.5) {$\circ$}
\put(4,0.5) {$\ldots$}
\put(7,0.5) {$\square$}
\put(10,0.5) {$\square$}
\put(13,0.5) {$\square$}
\put(0.2,1.1){\line(0,1){2.0}}
\put(1.2,1.1){\line(0,1){1.6}}
\put(7.35,1.3){\line(0,1){1.4}}
\put(13.35,1.3){\line(0,1){1.8}}
\put(0.2,3.1){\line(1,0){13.15}}
\put(1.2,2.7){\line(1,0){6.15}}
\end{picture}
}}}}
\quad\quad\cdots\vspace{22pt}
\end{equation}
See below, why the illustration in deed has to be of this structure.
Note that the diagonal line in the picture just symbolizes any path from a point inside \(K\big((i\!+\!1)'\) to a point inside \(K(i)\).
This path might cross the area between the two rows of points several times and also points ``left of \(K\big((i\!+\!1)'\)'' and ``right of \(K(i)\)'' might be visited.
The analogous remark holds for the vertical line between \(K\big((i\!-\!1)'\big)\) and \(K\big((i\!-\!1)\big)\).
\newline
The components containing the points \(i\) and \(i'\) form ``circles'' around \(K(i')\), so, by non-crossingness, a path starting inside \(K(i')\) and ending at one of the points \(1,\ldots,(i\!-\!2),(i\!+\!2),\ldots,(s\!+\!2)\) has to visit at least one of the circles.
We conclude that \(K\big((i)')\) would in addition be connected to \((i-1)\) or \(i\), contradicting our assumption.
\newline
In order to see that Illustration \ref{eqn:relevant_part_of_H_r(p,f(i,q))} does not contain inappropriate assumptions, we observe the following:
Upon first sight, it seems that the points \(K(i')\) do not have to be between the drawn ``circles'' but also the following constellations would be possible: 
\[
\cdots\quad\quad\quad
\raisebox{0.3cm}{\setlength{\unitlength}{0.5cm}
\begin{picture}(15,2.4)
\put(0,0.5) {$\circ$}
\put(1,0.5) {$\circ$}
\put(4,0.5) {$\cdots$}
\put(10,0.5) {$\square$}
\put(13,0.5) {$\square$}
\put(0.2,1.1){\line(0,1){2.0}}
\put(1.2,1.1){\line(0,1){1.6}}
\put(10.35,1.3){\line(0,1){1.4}}
\put(13.35,1.3){\line(0,1){1.8}}
\put(0.2,3.1){\line(1,0){13.15}}
\put(1.2,2.7){\line(1,0){9.15}}
\put(-1.8,-1.6) {\scalebox{0.7}{\((i-1)'\)}}
\put(-1.8,0.3) {\scalebox{0.7}{\((i-1)\)}}
\put(1.5,-1.6) {\scalebox{0.7}{\((i)'\)}}
\put(1.5,0.3) {\scalebox{0.7}{\((i)\)}}
\put(4.3,-2.4) {\scalebox{0.7}{\(K\big((i\!+\!1)'\big)\)}}
\put(10.6,-2.4) {\scalebox{0.7}{\(K\big(i'\big)\)}}
\put(10.6,1.1) {\scalebox{0.7}{\(K\big(i\big)\)}}
\put(13.6,-2.4) {\scalebox{0.7}{\(K\big((i\!-\!1)'\big)\)}}
\put(13.6,1.1) {\scalebox{0.7}{\(K\big(i\!-\!1\big)\)}}
\put(13.35,-1.2){\line(0,1){1.3}}
\put(11.85,-3.7){\line(0,1){1}}
\put(11.85,-1.8){\line(0,1){1.1}}
\put(10.35,-0.7){\line(1,0){1.5}}
\put(10.35,-0.7){\line(0,1){0.8}}
\put(11.85,-3.7){\line(-1,0){3}}
\put(8.85,-3.7){\line(0,1){2}}
\put(8.85,-1.7){\line(-1,0){1}}

\end{picture}
\hspace{-7.77cm}
\raisebox{-0.5cm}{\reflectbox{\scalebox{-1}{
\begin{picture}(15,2.4)
\put(0,0.5) {$\circ$}
\put(1,0.5) {$\circ$}
\put(4,0.5) {$\ldots$}
\put(7,0.5) {$\square$}
\put(10,0.5) {$\square$}
\put(13,0.5) {$\square$}
\put(0.2,1.1){\line(0,1){2.0}}
\put(1.2,1.1){\line(0,1){1.6}}
\put(7.35,1.3){\line(0,1){1.4}}
\put(13.35,1.3){\line(0,1){1.8}}
\put(0.2,3.1){\line(1,0){13.15}}
\put(1.2,2.7){\line(1,0){6.15}}
\end{picture}
}}}}
\quad\quad\cdots
\]
\vspace{33pt}
\[
\cdots\quad\quad\quad
\raisebox{0.3cm}{\setlength{\unitlength}{0.5cm}
\begin{picture}(15,2.4)
\put(0,0.5) {$\circ$}
\put(1,0.5) {$\circ$}
\put(4,0.5) {$\cdots$}
\put(10,0.5) {$\square$}
\put(13,0.5) {$\square$}
\put(0.2,1.1){\line(0,1){2.0}}
\put(1.2,1.1){\line(0,1){1.6}}
\put(10.35,1.3){\line(0,1){1.4}}
\put(13.35,1.3){\line(0,1){1.8}}
\put(0.2,3.1){\line(1,0){13.15}}
\put(1.2,2.7){\line(1,0){9.15}}
\put(7.35,-1){\line(3,1){3}}
\put(-1.8,-1.6) {\scalebox{0.7}{\((i-1)'\)}}
\put(-1.8,0.3) {\scalebox{0.7}{\((i-1)\)}}
\put(1.5,-1.6) {\scalebox{0.7}{\((i)'\)}}
\put(1.5,0.3) {\scalebox{0.7}{\((i)\)}}
\put(4.3,-2.4) {\scalebox{0.7}{\(K\big((i\!+\!1)'\big)\)}}
\put(8.2,-2.4) {\scalebox{0.7}{\(K\big(i'\big)\)}}
\put(10.6,1.1) {\scalebox{0.7}{\(K\big(i\big)\)}}
\put(13.6,-2.4) {\scalebox{0.7}{\(K\big((i\!-\!1)'\big)\)}}
\put(13.6,1.1) {\scalebox{0.7}{\(K\big(i\!-\!1\big)\)}}
\put(13.35,-0.8){\line(0,1){0.9}}
\put(9.35,-0.8){\line(1,0){4}}
\put(9.35,-1.8){\line(0,1){1}}
\put(9.35,-3.8){\line(0,1){1}}
\put(9.35,-3.8){\line(1,0){2.5}}
\put(11.85,-3.8){\line(0,1){2}}
\put(11.85,-1.8){\line(1,0){1}}
\end{picture}
\hspace{-7.77cm}
\raisebox{-0.5cm}{\reflectbox{\scalebox{-1}{
\begin{picture}(15,2.4)
\put(0,0.5) {$\circ$}
\put(1,0.5) {$\circ$}
\put(4,0.5) {$\ldots$}
\put(7,0.5) {$\square$}
\put(10,0.5) {$\square$}
\put(13,0.5) {$\square$}
\put(0.2,1.1){\line(0,1){2.0}}
\put(1.2,1.1){\line(0,1){1.6}}
\put(7.35,1.3){\line(0,1){1.4}}
\put(13.35,1.3){\line(0,1){1.8}}
\put(0.2,3.1){\line(1,0){13.15}}
\put(1.2,2.7){\line(1,0){6.15}}
\end{picture}
}}}}
\quad\quad\cdots\vspace{11pt}
\]
However, by Observation \ref{obs:important_observation_about_drawing_f(i,q)}, the only lines drawn below the part 
\begin{picture}(3,1)
\put(0,-0.2) {\scalebox{0.7}{\(K\big((i)'\big)\)}}
\put(2,0) {$\square$}
\end{picture}
can be chosen to be the lines between \(j'\) and \(K\big(j'\big)\) for \(1\!\le\! j\!<\!i\).
Hence, the two pictures above do not occur.
\end{proof}
Given a structure as defined in \ref{defn:structures[i]_a.s.o.}, Proposition \ref{prop:when_is_e_r(p,q)_non-zero?} tells us when an expression \(e_{(\cdot)}(\cdot,\cdot)\) is non-zero.
This will be used in the proof of Lemma \ref{lem:F_r(p,q)-formula}.
\begin{cor}\label{cor:corollary_from_prop:when_is_e_r(p,q)_non-zero?}
Let \(1\!\le\! r\!\le\! n\!-\!1\), \(p\in W(n,r)\) and \(q\in W(n,r\!+\!1)\).
For the statements below let \(i,j\in\N\) be such that the respective objects are well-defined.
\begin{itemize}
\item[(1)] If \(H_r(p,q)\) has structure \([i,i\!+\!1]\), then we have
\begin{itemize}
\item[(1.1)] \(e_r(p,f(j,q))\neq0\) only if \(j\in\{i,i\!+\!1\}\).
\item[(1.2)] \(e_r(p,g(j,q))\neq0\) only if \(j\!=\!i\).
\item[(1.3)] \(e_r(p,q)\neq0\) only if \(i\!=\!s\!+\!1\) (and \(r\!=\!2s\!+\!1\)).
\end{itemize}
\item[(2)] If \(H_r(p,q)\) has structure \([i]\), then we have
\begin{itemize}
\item[(2.1)] \(e_r(p,f(j,q))\neq0\) only if \(j\!=\!i\).
\item[(2.2)] \(e_r(p,g(j,q))\neq0\) only if \(j\in\{i\!-\!1,i\}\).
\item[(3.3)] \(e_r(p,q)\neq0\) only if \(i\!=\!s\!+\!1\) (and \(r\!=\!2s\)).
\end{itemize}
\item[(3)] If \(H_r(p,q)\) has structure \([0]\), then we have
\begin{itemize}
\item[(3.1)] \(e_r(p,f(j,q))=0\) for all \(j\).
\item[(3.2)] \(e_r(p,g(j,q))\neq0\) only if \(j\!=\!s\!+\!1\) (and \(r\!=\!2s\!+\!1\)).
\item[(3.3)] \(e_r(p,q)\neq0\).
\end{itemize}
\end{itemize}
\end{cor}
We compare now the number of components in \(G\big(p,f(i,q)\big)\) and \(G\big(p,g(i,q)\big)\) with those of \(G(p,q)\).
\begin{lem}\label{lem:components_of_G(p,f(i,q))_and_G(p,g(i,q))}
Let \(1\!\le\! r\!\le\! n\!-\!1\), \(p\in W(n,r)\) and \(q\in W(n,r\!+\!1)\).
If \(e_r\big(p,f(i,q)\big)\) is non-zero, then we have
\[
\rl\big(p,f(i,q)\big)=
\begin{cases}
\rl(p,q)+s-i+2&\textnormal{ if \(H_r(p,q)\) is of structure \([i]\) or \([i\!-\!1,i]\)}\\
\rl(p,q)+s-i+1&\textnormal{ if \(H_r(p,q)\) is of structure \([i,i\!+\!1]\).}
\end{cases}
\]
If \(e_r\big(p,g(i,q)\big)\) is non-zero, then we have
\[
\rl\big(p,g(i,q)\big)=
\begin{cases}
\rl(p,q)+s-i+1&\textnormal{ if }H_r(p,q)\textnormal{ is of structure }[i]\textnormal{ or }[i,i\!+\!1]\\
\rl(p,q)+s-i&\textnormal{ if }H_r(p,q)\textnormal{  is of structure }[i\!+\!1]\\
\rl(p,q)-1&\textnormal{ if }H_r(p,q)\textnormal{  is of structure }[0].
\end{cases}
\]

\end{lem}
\begin{proof}
By Proposition \ref{prop:when_is_e_r(p,q)_non-zero?} the  cases mentioned here are indeed all of the relevant ones.
We start with the situation of \(r\!=\!2s\), \(f(i,q)\) and structure \([i]\).
The relevant part of \(G(p,q)\) influenced by the application of \(f(i,\cdot)\) is the one ``surrounded by the component containing \(i\)'':\vspace{10pt}
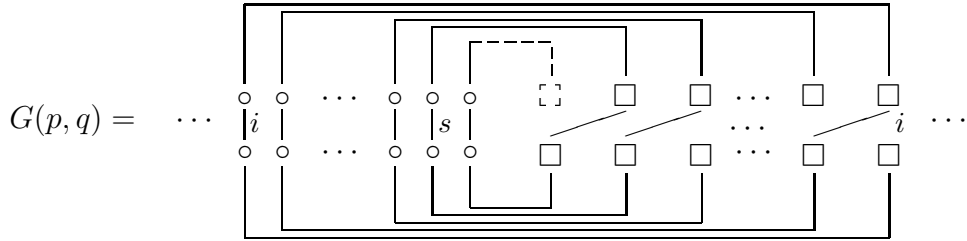
\[
G(p,q)=\quad\cdots\;\raisebox{0.3cm}{\textnormal{\hspace{-2cm}}\setlength{\unitlength}{0.5cm}
\begin{picture}(26,2.4)
\put(4,0) {$\circ$}
\put(5,0) {$\circ$}
\put(6,0) {$\;\cdots$}
\put(8,0) {$\circ$}
\put(9,0) {$\circ$}
\put(10,0) {$\circ$}
\put(12,0) {$\square$}
\put(12,0.32) {\linethickness{0.1cm}\color{white}\line(1,0){0.8}}
\put(12.34,0) {\linethickness{0.1cm}\color{white}\line(0,1){0.8}}
\put(14,0) {$\square$}
\put(16,0) {$\square$}
\put(17,0) {$\;\cdots$}
\put(19,0) {$\square$}
\put(21,0) {$\square$}
\put(4.2,0.6){\line(0,1){2.1}}
\put(5.2,0.6){\line(0,1){1.9}}
\put(8.2,0.6){\line(0,1){1.7}}
\put(9.2,0.6){\line(0,1){1.5}}
\put(9.2,0.6){\line(0,1){1.3}}
\put(10.2,0.6){\line(0,1){1.1}}
\multiput(12.38,0.8)(0,0.46){2}{\line(0,1){0.3}}
\put(14.35,0.8){\line(0,1){1.3}}
\put(16.35,0.8){\line(0,1){1.5}}
\put(19.35,0.8){\line(0,1){1.7}}
\put(21.35,0.8){\line(0,1){1.9}}
\put(12.35,-0.8){\line(3,1){2}}
\put(14.35,-0.8){\line(3,1){2}}
\put(16.85,-0.8){$\;\cdots$}
\put(19.35,-0.8){\line(3,1){2}}
\put(4.35,-0.7) {\small$i$}
\put(9.35,-0.7) {\small$s$}
\put(21.5,-0.7) {\small$i$}
\put(4.2,-0.9){\line(0,1){0.8}}
\put(5.2,-0.9){\line(0,1){0.8}}
\put(8.2,-0.9){\line(0,1){0.8}}
\put(9.2,-0.9){\line(0,1){0.8}}
\put(10.2,-0.9){\line(0,1){0.8}}
\multiput(10.2,1.7)(0.46,0){5}{\line(1,0){0.3}}
\put(9.2,2.1){\line(1,0){5.15}}
\put(8.2,2.3){\line(1,0){8.15}}
\put(5.2,2.5){\line(1,0){14.15}}
\put(4.2,2.7){\line(1,0){17.15}}
\end{picture}
\hspace{-13.27cm}
\raisebox{-0.5cm}{\reflectbox{\scalebox{-1}{
\begin{picture}(26,2.4)
\put(4,0) {$\circ$}
\put(5,0) {$\circ$}
\put(6,0) {$\;\cdots$}
\put(8,0) {$\circ$}
\put(9,0) {$\circ$}
\put(10,0) {$\circ$}
\put(12,0) {$\square$}
\put(14,0) {$\square$}
\put(16,0) {$\square$}
\put(17,0) {$\;\cdots$}
\put(19,0) {$\square$}
\put(21,0) {$\square$}
\put(4.2,0.6){\line(0,1){1.9}}
\put(5.2,0.6){\line(0,1){1.7}}
\put(8.2,0.6){\line(0,1){1.5}}
\put(9.2,0.6){\line(0,1){1.3}}
\put(10.2,0.6){\line(0,1){1.1}}
\put(12.35,0.8){\line(0,1){0.9}}
\put(14.35,0.8){\line(0,1){1.1}}
\put(16.35,0.8){\line(0,1){1.3}}
\put(19.35,0.8){\line(0,1){1.5}}
\put(21.35,0.8){\line(0,1){1.7}}
\put(10.2,1.7){\line(1,0){2.15}}
\put(9.2,1.9){\line(1,0){5.15}}
\put(8.2,2.1){\line(1,0){8.15}}
\put(5.2,2.3){\line(1,0){14.15}}
\put(4.2,2.5){\line(1,0){17.15}}
\end{picture}
}}}}
\textnormal{\hspace{-2cm}}\cdots
\]\vspace{10pt}\newline
As mentioned before, the squares symbolize a structure consisting possibly of more than one point and more than one component, but the points not connected to any of the points \(i,\ldots,(s\!+\!2)\) are not affected by the manipulations.
In this sense and as long as we are only interested in the alteration of the number of components, we treat the squares in the same way as points.
Having this in mind, we see that in the picture above there is just one component drawn.
Applying \(f(i,\cdot)\) to \(q\), we end up with\vspace{10pt}
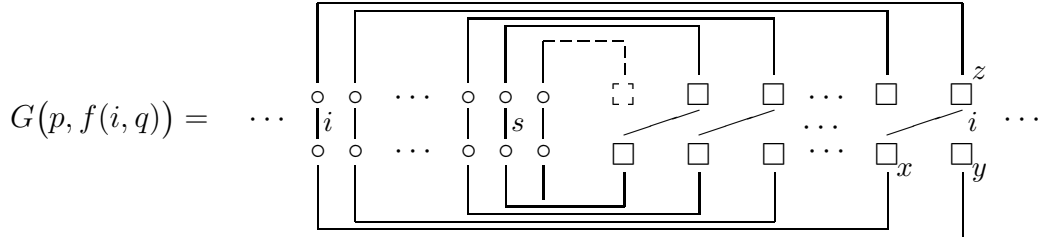
\[
G\big(p,f(i,q)\big)=\quad\cdots\;\raisebox{0.3cm}{\textnormal{\hspace{-2cm}}\setlength{\unitlength}{0.5cm}
\begin{picture}(26,2.4)
\put(4,0) {$\circ$}
\put(5,0) {$\circ$}
\put(6,0) {$\;\cdots$}
\put(8,0) {$\circ$}
\put(9,0) {$\circ$}
\put(10,0) {$\circ$}
\put(12,0) {$\square$}
\put(12,0.32) {\linethickness{0.1cm}\color{white}\line(1,0){0.8}}
\put(12.34,0) {\linethickness{0.1cm}\color{white}\line(0,1){0.8}}
\put(14,0) {$\square$}
\put(16,0) {$\square$}
\put(17,0) {$\;\cdots$}
\put(19,0) {$\square$}
\put(21,0) {$\square$}
\put(4.2,0.6){\line(0,1){2.1}}
\put(5.2,0.6){\line(0,1){1.9}}
\put(8.2,0.6){\line(0,1){1.7}}
\put(9.2,0.6){\line(0,1){1.5}}
\put(9.2,0.6){\line(0,1){1.3}}
\put(10.2,0.6){\line(0,1){1.1}}
\multiput(12.38,0.8)(0,0.46){2}{\line(0,1){0.3}}
\put(14.35,0.8){\line(0,1){1.3}}
\put(16.35,0.8){\line(0,1){1.5}}
\put(19.35,0.8){\line(0,1){1.7}}
\put(21.35,0.8){\line(0,1){1.9}}
\put(12.35,-0.8){\line(3,1){2}}
\put(14.35,-0.8){\line(3,1){2}}
\put(16.85,-0.8){$\;\cdots$}
\put(19.35,-0.8){\line(3,1){2}}
\put(4.35,-0.7) {\small$i$}
\put(9.35,-0.7) {\small$s$}
\put(21.5,-0.7) {\small$i$}
\put(19.6,-1.9) {\small$x$}
\put(21.6,-1.9) {\small$y$}
\put(21.6,0.7) {\small$z$}
\put(4.2,-0.9){\line(0,1){0.8}}
\put(5.2,-0.9){\line(0,1){0.8}}
\put(8.2,-0.9){\line(0,1){0.8}}
\put(9.2,-0.9){\line(0,1){0.8}}
\put(10.2,-0.9){\line(0,1){0.8}}
\multiput(10.2,1.7)(0.46,0){5}{\line(1,0){0.3}}
\put(9.2,2.1){\line(1,0){5.15}}
\put(8.2,2.3){\line(1,0){8.15}}
\put(5.2,2.5){\line(1,0){14.15}}
\put(4.2,2.7){\line(1,0){17.15}}
\end{picture}
\hspace{-13.27cm}
\raisebox{-0.5cm}{\reflectbox{\scalebox{-1}{
\begin{picture}(26,2.4)
\put(4,0) {$\circ$}
\put(5,0) {$\circ$}
\put(6,0) {$\;\cdots$}
\put(8,0) {$\circ$}
\put(9,0) {$\circ$}
\put(10,0) {$\circ$}
\put(12,0) {$\square$}
\put(14,0) {$\square$}
\put(16,0) {$\square$}
\put(17,0) {$\;\cdots$}
\put(19,0) {$\square$}
\put(21,0) {$\square$}
\put(4.2,0.6){\line(0,1){1.7}}
\put(5.2,0.6){\line(0,1){1.5}}
\put(8.2,0.6){\line(0,1){1.3}}
\put(9.2,0.6){\line(0,1){1.1}}
\put(10.2,0.6){\line(0,1){0.9}}
\put(12.35,0.8){\line(0,1){0.9}}
\put(14.35,0.8){\line(0,1){1.1}}
\put(16.35,0.8){\line(0,1){1.3}}
\put(19.35,0.8){\line(0,1){1.5}}
\put(21.35,0.8){\line(0,1){1.7}}
\put(9.2,1.7){\line(1,0){3.15}}
\put(8.2,1.9){\line(1,0){6.15}}
\put(5.2,2.1){\line(1,0){11.15}}
\put(4.2,2.3){\line(1,0){15.15}}

\end{picture}
}}}}
\textnormal{\hspace{-2cm}}\cdots
\]\vspace{10pt}\newline
Comparing the two situations, we see that \(G\big(p,f(i,q)\big)\) has \(s-i+2\) components more than \(G(p,q)\).
Starting with structure \([i,i\!+\!1]\), we directly see that the only difference is a connection between the squares \(y\) and \(z\), reducing the number of components in \(G\big(p,f(i,q)\big)\) by one.\newline
Considering \(g(i,q)\) instead of \(f(i,q)\) just connects in the second picture the squares \(x\) and \(y\) in the lower row.
For structure \([i]\), this decreases the number of components by 1 and in the case \([i,i\!+\!1]\) it leaves the components unchanged.
Of course, with the cases \([i]\) and \([i,i\!+\!1]\) we also proved the corresponding statements for \([i\!+\!1]\) and \([i\!-\!1,i]\) by shifting \(i\) by \(\pm1\).\newline
In the case \(r\!=\!2s\!+\!1\) we get the same results:
The component that contains \((s\!+\!1)\) looks a bit different, but that does not affect the difference in numbers of components when we replace \(q\) by \(f(i,q)\) or \(g(i,q)\):\vspace{10pt}
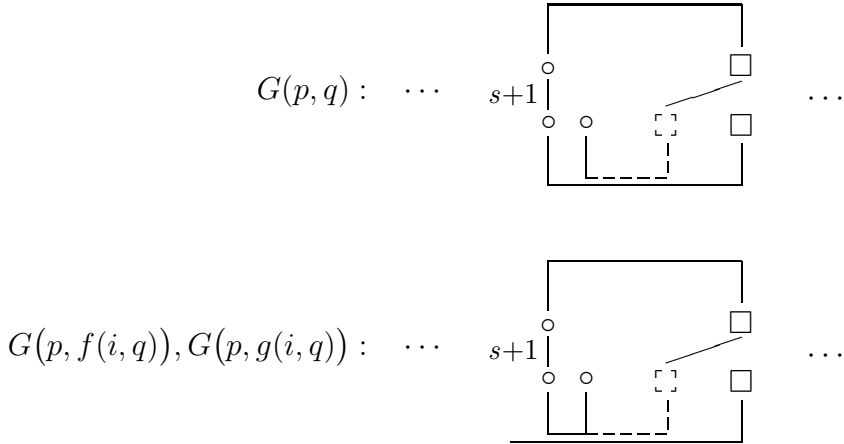
\begin{align*}
G(p,q):\quad&\cdots\;\raisebox{0.3cm}{\textnormal{\hspace{-4cm}}\setlength{\unitlength}{0.5cm}
\begin{picture}(26,2.4)
\put(10,0) {$\circ$}
\put(15,0) {$\square$}
\put(10.2,0.6){\line(0,1){1.3}}
\put(15.35,0.8){\line(0,1){1.1}}
\put(13.35,-0.8){\line(3,1){2}}
\put(16.85,-0.8){$\;\cdots$}
\put(8.6,-0.7) {\small \emph{s}+1}
\put(10.2,-0.9){\line(0,1){0.8}}
\put(10.2,1.9){\line(1,0){5.15}}
\end{picture}
\hspace{-13.27cm}
\raisebox{-0.5cm}{\reflectbox{\scalebox{-1}{
\begin{picture}(26,2.4)
\put(10,0) {$\circ$}
\put(11,0) {$\circ$}
\put(13,0) {$\square$}
\put(13,0.32) {\linethickness{0.1cm}\color{white}\line(1,0){0.8}}
\put(13.34,0) {\linethickness{0.1cm}\color{white}\line(0,1){0.8}}
\put(15,0) {$\square$}
\put(10.2,0.6){\line(0,1){1.3}}
\put(11.2,0.6){\line(0,1){1.1}}
\multiput(13.38,0.8)(0,0.46){2}{\line(0,1){0.3}}
\put(15.35,0.8){\line(0,1){1.1}}
\multiput(11.2,1.7)(0.46,0){5}{\line(1,0){0.3}}
\put(10.2,1.9){\line(1,0){5.15}}
\end{picture}
}}}}
\\[10pt]
G\big(p,f(i,q)\big),G\big(p,g(i,q)\big):\quad&\cdots\;\raisebox{0.3cm}{\textnormal{\hspace{-4cm}}\setlength{\unitlength}{0.5cm}
\begin{picture}(26,2.4)
\put(10,0) {$\circ$}
\put(15,0) {$\square$}
\put(10.2,0.6){\line(0,1){1.3}}
\put(15.35,0.8){\line(0,1){1.1}}
\put(13.35,-0.8){\line(3,1){2}}
\put(16.85,-0.8){$\;\cdots$}
\put(8.6,-0.7) {\small \emph{s}+1}
\put(10.2,-0.9){\line(0,1){0.8}}
\put(10.2,1.9){\line(1,0){5.15}}
\end{picture}
\hspace{-13.27cm}
\raisebox{-0.5cm}{\reflectbox{\scalebox{-1}{
\begin{picture}(26,2.4)
\put(10,0) {$\circ$}
\put(11,0) {$\circ$}
\put(13,0) {$\square$}
\put(13,0.32) {\linethickness{0.1cm}\color{white}\line(1,0){0.8}}
\put(13.34,0) {\linethickness{0.1cm}\color{white}\line(0,1){0.8}}
\put(15,0) {$\square$}
\put(10.2,0.6){\line(0,1){1.1}}
\put(11.2,0.6){\line(0,1){1.1}}
\multiput(13.38,0.8)(0,0.46){2}{\line(0,1){0.3}}
\put(15.35,0.8){\line(0,1){1.1}}
\multiput(11.2,1.7)(0.46,0){5}{\line(1,0){0.3}}
\put(10.2,1.7){\line(1,0){1}}
\put(9.2,1.9){\line(1,0){6.15}}
\end{picture}
}}}}
\end{align*}\vspace{6pt}\newline
Comparing this with the situation \(r\!=\!2s\), one sees that we just have to delete in the lower row 
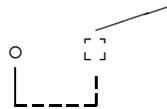
\[
\setlength{\unitlength}{0.5cm}
\begin{picture}(3,0)
\put(0,0) {$\circ$}
\put(2,0) {$\square$}
\put(2,0.32) {\linethickness{0.1cm}\color{white}\line(1,0){0.8}}
\put(2.34,0) {\linethickness{0.1cm}\color{white}\line(0,1){0.8}}
\put(0.2,-1.2){\line(0,1){1}}
\multiput(2.35,-1.2)(0,0.46){2}{\line(0,1){0.3}}
\multiput(0.2,-1.2)(0.46,0){5}{\line(1,0){0.3}}
\put(2.35,0.8){\line(3,1){2}}
\end{picture}\vspace{22pt}
\]
right beside \((s+1)'\).
However, this does not change the number of components, so the formulae do not change.
\newline
Note that in the case of \(g(s+1,q)\) and structure \([0]\) the graphs to compare are given by\vspace{10pt}
\begin{align*}
G(p,q)=\quad&\cdots\;\raisebox{0.3cm}{\textnormal{\hspace{-4cm}}\setlength{\unitlength}{0.5cm}
\begin{picture}(26,2.4)
\put(9,0) {$\circ$}
\put(10,0) {$\circ$}
\put(15,0) {$\square$}
\put(17,0) {$\square$}
\put(9.2,0.6){\line(0,1){1.5}}
\put(10.2,0.6){\line(0,1){1.3}}
\put(15.35,0.8){\line(0,1){1.1}}
\put(17.35,0.8){\line(0,1){1.3}}
\put(15.35,-0.9){\line(0,1){0.8}}
\put(17.35,-0.9){\line(0,1){0.8}}
\put(19.85,-0.8){$\cdots$}
\put(9.35,-0.7) {\small$s$}
\put(9.2,-0.9){\line(0,1){0.8}}
\put(10.2,-0.9){\line(0,1){0.8}}
\put(10.2,1.9){\line(1,0){5.15}}
\put(9.2,2.1){\line(1,0){8.15}}
\end{picture}
\hspace{-13.27cm}
\raisebox{-0.5cm}{\reflectbox{\scalebox{-1}{
\begin{picture}(26,2.4)
\put(9,0) {$\circ$}
\put(10,0) {$\circ$}
\put(11,0) {$\circ$}
\put(13,0) {$\square$}
\put(13,0.32) {\linethickness{0.1cm}\color{white}\line(1,0){0.8}}
\put(13.34,0) {\linethickness{0.1cm}\color{white}\line(0,1){0.8}}
\put(15,0) {$\square$}
\put(17,0) {$\square$}
\put(9.2,0.6){\line(0,1){1.5}}
\put(10.2,0.6){\line(0,1){1.3}}
\put(11.2,0.6){\line(0,1){1.1}}
\multiput(13.38,0.8)(0,0.46){2}{\line(0,1){0.3}}
\put(15.35,0.8){\line(0,1){1.1}}
\put(17.35,0.8){\line(0,1){1.3}}
\multiput(11.2,1.7)(0.46,0){5}{\line(1,0){0.3}}
\put(10.2,1.9){\line(1,0){5.15}}
\put(9.2,2.1){\line(1,0){8.15}}
\end{picture}
}}}}
\\[20pt]
G\big(p,g(s+1,q)\big)=\quad&\cdots\;\raisebox{0.3cm}{\textnormal{\hspace{-4cm}}\setlength{\unitlength}{0.5cm}
\begin{picture}(26,2.4)
\put(9,0) {$\circ$}
\put(10,0) {$\circ$}
\put(15,0) {$\square$}
\put(17,0) {$\square$}
\put(9.2,0.6){\line(0,1){1.5}}
\put(10.2,0.6){\line(0,1){1.3}}
\put(15.35,0.8){\line(0,1){1.1}}
\put(17.35,0.8){\line(0,1){1.3}}
\put(15.35,-0.9){\line(0,1){0.8}}
\put(17.35,-0.9){\line(0,1){0.8}}
\put(19.85,-0.8){$\cdots$}
\put(9.35,-0.7) {\small$s$}
\put(9.2,-0.9){\line(0,1){0.8}}
\put(10.2,-0.9){\line(0,1){0.8}}
\put(10.2,1.9){\line(1,0){5.15}}
\put(9.2,2.1){\line(1,0){8.15}}
\end{picture}
\hspace{-13.27cm}
\raisebox{-0.5cm}{\reflectbox{\scalebox{-1}{
\begin{picture}(26,2.4)
\put(9,0) {$\circ$}
\put(10,0) {$\circ$}
\put(11,0) {$\circ$}
\put(13,0) {$\square$}
\put(13,0.32) {\linethickness{0.1cm}\color{white}\line(1,0){0.8}}
\put(13.34,0) {\linethickness{0.1cm}\color{white}\line(0,1){0.8}}
\put(15,0) {$\square$}
\put(17,0) {$\square$}
\put(9.2,0.6){\line(0,1){1.5}}
\put(10.2,0.6){\line(0,1){1.3}}
\put(11.2,0.6){\line(0,1){1.3}}
\multiput(13.36,0.7)(0,0.42){3}{\line(0,1){0.28}}
\put(15.35,0.8){\line(0,1){1.1}}
\put(17.35,0.8){\line(0,1){1.3}}
\put(10.2,1.9){\line(1,0){5.15}}
\put(9.2,2.1){\line(1,0){8.15}}
\end{picture}
}}}}
\end{align*}\vspace{10pt}\newline
so \(\rl\big(p,g(s+1,q)\big)=\rl(p,q)\!-\!1\), as desired.
\end{proof}
If \(1\le r<n-1\) as well as \(p\in W(n,r)\) and \(q\in W(n,r+1)\) are given, we define now the following terms:
\begin{align}\label{eqn:F_r(p,q)}
\begin{split}
F_r(p,q)=&\sum_{j=1}^{s+1}N^{-(s-j+2)}\beta_{2j-1}(z)e_r\big(p,f(j,q)\big)\\
&-\sum_{j=1}^{t}N^{-(s-j+1)}\beta_{2j}(z)e_r\big(p,g(j,q)\big)
\end{split}
\end{align}
The index bound \(t\) above is either \(s\) or \(s\!+\!1\), depending on whether \(r\!=\!2s\) or \(r\!=\!2s\!+\!1\). 
For the sake of readability we used the abbreviation \(z:=\frac{1}{N}\).
Here \(\beta\) are the reversed Beraha polynomials from Definition \ref{defn:beraha_polynomials}, given by the recursion
\begin{align}\label{eqn:recursion_behara_second_time}
\begin{split}
\beta_0(X)=0\quad&,\quad\beta_1(X)=1\\
\beta_{n+1}(X)=\beta_{n}(X)-&X\beta_{n-1}(X)\quad,\forall n\ge 1.
\end{split}
\end{align}
Corollary \ref{cor:corollary_from_prop:when_is_e_r(p,q)_non-zero?} and Lemma \ref{lem:components_of_G(p,f(i,q))_and_G(p,g(i,q))} are preparatory for the following result about the expressions \(F_r(p,q)\).
\begin{lem}\label{lem:F_r(p,q)-formula}
Let \(1\!\le\! r\!<\!n\!-\!1\), \(p\in W(n,r)\) and \(q\in W(n,r\!+\!1)\). 
Using the abbreviation \(z:=\frac{1}{N}\), it holds 
\begin{equation}\label{eqn:lem:F_r(p,q)-formula}
(-1)^rF_r(p,q)=\beta_{r+2}(z)\,e_r(p,q)-\beta_{r+3}(z)\,e_{r+1}(p,q)
\end{equation}
\end{lem}
\begin{proof}
By definition of \(F_r(p,q)\), see Equation \ref{eqn:F_r(p,q)},  we can write out Equation \ref{eqn:lem:F_r(p,q)-formula} as
\begin{align}\label{eqn:lem:F_r(p,q)-formula_extended}
\begin{split}
&\sum_{j=1}^{s+1}N^{-(s-j+2)}\beta_{2j-1}(z)e_r\big(p,f(j,q)\big)\\
&-\sum_{j=1}^{t}N^{-(s-j+1)}\beta_{2j}(z)e_r\big(p,g(j,q)\big)\\
=&\,(-1)^r\left(\beta_{r+2}(z)\,e_r(p,q)-\beta_{r+3}(z)\,e_{r+1}(p,q)\right)
\end{split}
\end{align}

We prove this formula for all different structures of \(H_r(p,q)\) which might appear.
It will turn out that, in order to have at least one non-vanishing \(e_{\cdot}(\cdot,\cdot)\), the graph \(H_r(p,q)\) has to be of one of the structures from Definition \ref{defn:structures[i]_a.s.o.} and, even in these cases, most summands on the left side turn out to be zero.
\newline
\textbf{Case 1: \(H_r(p,q)\) has structure \([i,i\!+\!1]\):}
If \(i\!\neq\!s\!+\!1\), then by Corollary \ref{cor:corollary_from_prop:when_is_e_r(p,q)_non-zero?} only three summands survive in the sums in Equation \ref{eqn:lem:F_r(p,q)-formula_extended}: The summands for \(j\!=\!i\) and \(j\!=\!i\!+\!1\) in the first sum and the summand for \(j\!=\!i\) in the second sum.
Hence, Equation \ref{eqn:lem:F_r(p,q)-formula_extended} reads
\begin{align}\label{eqn:lem:F_r(p,q)-formula_extended_case1}
\begin{split}
&N^{-(s-i+2)}\beta_{2i-1}(z)e_r\big(p,f(i,q)\big)+N^{-(s-i+1)}\beta_{2i+1}(z)e_r\big(p,f(i+1,q)\big)\\
&-N^{-(s-i+1)}\beta_{2i}(z)e_r\big(p,g(i,q)\big)\\
=&\,(-1)^r\left(\beta_{r+2}(z)\,e_r(p,q)-\beta_{r+3}(z)\,e_{r+1}(p,q)\right)
\end{split}
\end{align}
By Lemma \ref{lem:components_of_G(p,f(i,q))_and_G(p,g(i,q))} it holds
\[
e_r\big(p,f(i,q)=e_r\big(p,f(i\!+\!1),q\big)=e_r\big(p,g(i,q)\big)=N^{\rl(q^*,p)+s-i+1},
\]
so we can rewrite the left side of Equation \ref{eqn:lem:F_r(p,q)-formula_extended_case1} to obtain
\[N^{\rl(q^*,p)}\underbrace{\Big(z\beta_{2i-1}(z)-\beta_{2i}(z)+\beta_{2i+1}(z)\Big)}_{=0}=(-1)^r\left(\beta_{r+2}(z)\,e_r(p,q)-\beta_{r+3}(z)\,e_{r+1}(p,q)\right).\]
The bracket vanishes by the recursion formula for the reversed Beraha polynomials, see Equation \ref{eqn:recursion_behara_second_time}.
Because of structure \([i,i\!+\!1]\) it holds
\[e_{r}(p,q)=e_{r+1}(p,q)=0,\]
and also the right side vanishes, as desired.
\newline
The situation \(i\!=\!s\!+\!1\) is only defined and relevant if \(r\!=\!2s\!+\!1\) is odd.
In this case we have to omit in Equation \ref{eqn:lem:F_r(p,q)-formula_extended_case1} the term with \(\beta_{2i+1}(z)\)  as the summations over \(j\) in Equation \ref{eqn:lem:F_r(p,q)-formula_extended} both end at \(j=s+1\).
Hence, Equation \ref{eqn:lem:F_r(p,q)-formula_extended} reads
\begin{equation}\label{eqn:lem:F_r(p,q)-formula_extended_case1-special}
N^{\rl(q^*,p)}\Big(z\beta_{2i-1}(z)-\beta_{2i}(z)\Big)=(-1)^r\left(\beta_{r+2}(z)\,e_r(p,q)-\beta_{r+3}(z)\,e_{r+1}(p,q)\right).
\end{equation}
We compute for the left side
\begin{align*}
&N^{\rl(q^*,p)}\Big(z\beta_{2i-1}(z)-\beta_{2i}(z)\Big)\\
=&N^{\rl(q^*,p)}\Big(-\beta_{2i+1}(z)\Big)\\
=&N^{\rl(q^*,p)}\Big(-z\beta_{r+2}(z)\Big)\\
=&(-1)^r\Big(\beta_{r+2}(z)\Big)e_r(p,q)\\
\end{align*}
as \(-1=(-1)^r\) and \(e_r(p,q)=N^{\rl(q^*,p)}\) by Proposition \ref{prop:when_is_e_r(p,q)_non-zero?}.
This is equal to the right side of Equation \ref{eqn:lem:F_r(p,q)-formula_extended_case1-special} because \(e_{r+1}(p,q)\) is zero:
Structure  \([s\!+\!1,s\!+\!2]\) in particular tells us that \((s+1)'\) and \((s+2)'\) are connected in \(H_r(p,q)\), so there is an \((r\!+\!1)\)-flaw in \(G(p,q)\) and \(e_{r+1}(p,q)=0\).
Again this shows the claim.
\newline
\textbf{Case 2: \(H_r(p,q)\) has structure \([i]\):}
We exclude at first the situation \(i\!=\!1\) as well as \(i\!=\!s\!+\!1\) in the case \(r\!=\!2s\).
By Corollary \ref{cor:corollary_from_prop:when_is_e_r(p,q)_non-zero?} only three summands survive in the sums in Equation \ref{eqn:lem:F_r(p,q)-formula_extended}: The summand for \(j\!=\!i\) in the first sum and the summands for \(j\!=\!i\!-\!1\) and \(j\!=\!i\) in the second sum.
Hence, Equation \ref{eqn:lem:F_r(p,q)-formula_extended} reads in this situation
\begin{align}\label{eqn:lem:F_r(p,q)-formula_extended_case2}
\begin{split}
&N^{-(s-i+2)}\beta_{2i-1}(z)e_r\big(p,f(i,q)\big)\\
&-N^{-(s-i+2)}\beta_{2i-2}(z)e_r\big(p,g(i\!-\!1,q)\big)-N^{-(s-i+1)}\beta_{2i}(z)e_r\big(p,g(i,q)\big)\\
=&\,(-1)^r\left(\beta_{r+2}(z)\,e_r(p,q)-\beta_{r+3}(z)\,e_{r+1}(p,q)\right)
\end{split}
\end{align}
By Lemma \ref{lem:components_of_G(p,f(i,q))_and_G(p,g(i,q))} it holds
\[
e_r\big(p,f(i,q)=e_r\big(p,g(i,q)\big)=N^{\rl(q^*,p)+s-i+1}\quad,\quad 
e_r\big(p,g(i\!-\!1,q)\big)=N^{\rl(q^*,p)+s-i+2},
\]
so we can rewrite the left side of Equation \ref{eqn:lem:F_r(p,q)-formula_extended_case1} to obtain
\begin{align*}
N^{\rl(q^*,p)}\underbrace{\Big(-z\beta_{2i-2}(z)+\beta_{2i-1}(z)-\beta_{2i}(z)\Big)}_{=0}&\\
=(-1)^r\big(\beta_{r+2}(z)&\,e_r(p,q)-\beta_{r+3}(z)\,e_{r+1}(p,q)\big).
\end{align*}
The bracket vanishes again by the recursion formula for the reversed Beraha polynomials, see Equation \ref{eqn:recursion_behara_second_time}.
If \(i\!=\!1\), then in Equation \ref{eqn:F_r(p,q)} we do not have a summation index \(j=0\) and so we have to omit at first sight the term \(z\beta_{2i-2}(z)=z\beta_{0}(z)\).
However, as \(\beta_{0}(z)=0\), we still get the same result.
It remains to show that also the right side vanished:
\newline
We note that having structure \([i]\) (and not the situation of \(i=s+1\) in combination with \(r=2s\)) says that the point \(i\) is the witness of an \(r\)-flaw of \(G(p,q)\) as \(i\) is not connected to \(i'\) in \(H_r(p,q)\).
Hence both \(e_{r}(p,q)\) and \(e_{r+1}(p,q)\) are zero.\newline
Finally, we consider the case of \(i\!=\!s\!+\!1\) and \(r\!=\!2s\).
The second sum on the left side of Equation \ref{eqn:lem:F_r(p,q)-formula_extended} ends at \(t\!=\!s\), so there is no term \(\beta_{2i}(z)\) to consider and Equation \ref{eqn:lem:F_r(p,q)-formula_extended_case2} reads
\begin{align}\label{eqn:lem:F_r(p,q)-formula_extended_case2-special}
N^{\rl(q^*,p)}\Big(-z\beta_{2i-2}(z)+\beta_{2i-1}(z)\Big)
=(-1)^r\big(\beta_{r+2}(z)\,e_r(p,q)-\beta_{r+3}(z)\,e_{r+1}(p,q)\big).
\end{align}
For the left side we compute
\begin{align*}
N^{\rl(q^*,p)}\Big(-z\beta_{2i-2}(z)+\beta_{2i-1}(z)\Big)
=&N^{\rl(q^*,p)}\Big(\beta_{2i}(z)\Big)\\
=&N^{\rl(q^*,p)}\Big(\beta_{r+2}(z)\Big)\\
=&(-1)^r\Big(\beta_{r+2}(z)\Big)e_r(p,q)\\
\end{align*}
as \((-1)^r=(-1)^{2s}=1\) and \(e_r(p,q)=N^{\rl(q^*,p)}\) by Proposition \ref{prop:when_is_e_r(p,q)_non-zero?}.
For the right side of Equation \ref{eqn:lem:F_r(p,q)-formula_extended_case2-special} we see again that \(e_{r+1}(p,q)\) vanishes:
Structure \([s+1]\) tells us that the point \((s+1)\) witnesses an \((r\!+\!1)\)-flaw of \(G(p,q)\) as it is not connected to \((s\!+\!1)'\) in \(H_r(p,q)\).
Hence, we proved the claim also in this case.
\newline
\textbf{Case 3: \(H_r(p,q)\) has structure \([0]\):}
If we assume structure \([0]\), then there is only one possibility for a non-trivial summand on the left side of Equation \ref{eqn:lem:F_r(p,q)-formula_extended}, namely if \(r\!=\!2s\!+\!1\) and then only the summand for \(j\!=\!s\!+\!1\) in the second sum survives.
It can directly be checked by the definition of structure \([0]\), that \(e_{r+1}(p,q)\) is non-zero and so Equation \ref{eqn:lem:F_r(p,q)-formula_extended} reads in this case
\begin{align}\label{eqn:lem:F_r(p,q)-formula_extended_case3}
\beta_{2s+2}(z)e_r\big(p,g(s+1,q)\big)=(-1)^r\big(\beta_{r+2}(z)\,e_r(p,q)-\beta_{r+3}(z)\,e_{r+1}(p,q)\big)
\end{align}
By Lemma \ref{lem:components_of_G(p,f(i,q))_and_G(p,g(i,q))} it holds
\[e_r\big(p,g(s+1,q)=N^{\rl(q^*,p)-1}\]
and we compute
\begin{align*}
\beta_{2s+2}(z)e_r\big(p,g(s+1,q)\big)\\
=&N^{\rl(q^*,p)}z\beta_{r+1}(z)\\
=&N^{\rl(q^*,p)}\Big(\beta_{r+2}(z)-\beta_{r+3}(z)\Big)\\
=&\beta_{r+2}(z)e_r(p,q)-\beta_{r+3}(z)e_{r+1}(p,q),
\end{align*}
proving Equation \ref{eqn:lem:F_r(p,q)-formula_extended_case3}.
\newline
\textbf{Case 4: All other structures:}
For all other structures apart from the three cases above we have that \(e_r\big(p,f(i,q)\big)\),  \(e_r\big(p,g(i,q)\big)\) and \(e_r\big(p,q)\big)\) are zero by Proposition \ref{prop:when_is_e_r(p,q)_non-zero?}.
We have not proved this for \(e_{r+1}(p,q)\) yet, but it can easily be seen that structure \([0]\) on \(H_r(p,q)\) is a necessary condition for \(e_{r+1}(p,q)\) to be non-zero.
So also \(e_{r+1}(p,q)\) vanishes for all structures not explicitly considered above. Hence, both sides of Equation \ref{eqn:lem:F_r(p,q)-formula} are zero.
This finishes the proof.
\end{proof}

\subsection{A recursion formula}\label{subsec:recursion_formula}
%
Consider the matrix \(A(n,r)\) for \(0\!\le\! r\!<\! n\!-\!1\).
Assume that the first \(\#Y(n,r)\) rows and columns are labelled by the elements in \(Y(n,r)\).
In this case, \(A(n,r)\) has a block structure given by
\[A(n,r)=\begin{pmatrix}B(n,r)&C\\[6pt]D&E\end{pmatrix}\]
We denote the column of \(A(n,r)\) associated to a partition \(q\) by \(\col(q)\).\newline
Now for every \(q\in W(n,r+1)\), we consider the linear combination of columns
\[F_r(q):=\sum_{j=1}^{s+1}N^{-(s-j+2)}\beta_{2j-1}(z)\col\big(f(j,q)\big)-\sum_{j=1}^{t}N^{-(s-j+2)}\beta{2j}(z)\col\big(g(j,q)\big)\]
where we again defined \(z:=\frac{1}{N}\) and the summation bound \(t\) depends on the parity of \(r\): 
\[t:=\begin{cases}s&,r\textnormal{ even}\\s+1&,r\textnormal{ odd}\end{cases}\]
Note that \(F_r(q)\) is just the column vector \(\big(F_r(p,q)\big)_{p\in W(n,r)}\).
Multiplying now \(\col(q)\) by \(\beta_{r+2}(z)\) and adding \((-1)^rF_r(q)\) to it changes \(\col(q)\) by Equation \ref{eqn:lem:F_r(p,q)-formula} to
\[\beta_{r+3}(z)\Big(e_{r+1}(p,q)\Big)_{p\in W(n,r)}.\vspace{4pt}\]
For the first \(\#Y(n,r)\) rows this is zero, as for \(p\in Y(n,r)\) the graph \(G(p,q)\) has an \((r+1)\)-flaw, so \(e_{r+1}(p,q)=0\).
Hence we have
\begin{align*}
\begin{split}
&\det\begin{pmatrix}B(n,r)&C\\[6pt]D&E\end{pmatrix}\\
=&\left(\frac{\beta_{r+3}(z)}{\beta_{r+2}(z)}\right)^{\#(W(n,r+1))}\det\begin{pmatrix}B(n,r)&0\\[6pt]D&A(n,r+1)\end{pmatrix}\\
=&\left(\frac{\beta_{r+3}(z)}{\beta_{r+2}(z)}\right)^{\#(W(n,r+1))}\det\big(B(n,r)\big)\det(A(n,r+1))
\end{split}
\end{align*}
By Proposition \ref{prop:connections_between_B(n,r)_and_A(n,r)}, the matrix \(B(n,r)\) is again given by a matrix of type \(A(n-1,r')\), so we can summarize the observations above in the following Proposition:
\begin{prop}\label{prop:recursion_formula}
For \(0\!\le\! r\!<\!n\!-\!1\) we have the following recursion formula for the determinant of \(A(n,r)\):

\begin{equation}\label{eqn:recursion_formula}
\det\big(A(n,r)\big)=\left(\frac{\beta_{r+3}(z)}{\beta_{r+2}(z)}\right)^{\#(W(n,r+1))}\det\big(B(n,r)\big)\det(A(n,r+1))
\end{equation}
where
\begin{equation}\label{eqn:rewriting_B(n,r)}
\det\big(B(n,r)\big)=
\begin{cases}
\det\big(A(n-1,r-1)\big)&,\;r=2s+1.\\
N\cdot\det\big(A(n-1,r-1)\big)&,\;r=2s\textnormal{ and }r>0\\
N\cdot\det\big(A(n-1,0)\big)&,\;r=0.
\end{cases}
\end{equation}
\end{prop}
\begin{rem}\label{rem:beraha_quotients_are_well-defined}
Note that by Lemma \ref{lem:roots_of_beraha} the expression \(\beta_n(z)\) is non-zero for \(n\in\N\), \(z=\frac{1}{N}\) and \(N\in\N_{\ge 4}\).
Hence, also the fractions
\[\frac{\beta_{r+3}(z)}{\beta_{r+2}(z)}\]
are  non-zero and, in addition, the whole recursion above is guaranteed to be well-defined.
\end{rem}
\subsection{Conclusion: Linear independence of the maps \(T_p\)}
\label{subsec:conclusion:linear_independence_of_the_T_ps}
With the results of Section \ref{subsec:recursion_formula}, we can finally prove invertibility of the matrix \(A(n,0)\) and so Theorem \ref{thm:main_result:linear_independence_of_T_ps}.
\begin{thm}\label{thm:A(n,0)_invertible}
Let \(N\in\N_{\ge4}, n\in\N\) and consider the Gram matrix
\[A(n,0)=\left(\langle T_p(1),T_q(1)\rangle\right)_{p,q\in\mathcal{NC}(0,n)}.\]
Then \(A(n,0)\) is invertible.
\end{thm}
\begin{proof}
By the considerations in Section \ref{subsec:boiling_down_the_problem} and Section \ref{subsubsec:The graphs G(p,q), H_r(p,q) and r-flaws}, the Gram matrix above coincides with the matrix \(A(n,0)\) as defined in Definition \ref{defn:A(n,r)}.
Using repeatedly the recursion formula from Proposition \ref{prop:recursion_formula}, we see that the determinant of \(A(n,0)\) is given by a product whose factors are
\begin{itemize}
\item[(1)] quotients \(\displaystyle\frac{\beta_{r+3}(z)}{\beta_{r+2}(z)}\),
\item[(2)] powers of \(N\) or
\item [(3)] determinants of matrices \(A(n,n\!-\!1)\) with \(1\le n\le N\).
\end{itemize}
The quotients in (1) are non-zero as the reversed Beraha polynomials have no roots at \(z=\frac{1}{N}\), see Remark \ref{rem:beraha_quotients_are_well-defined} and Lemma \ref{lem:roots_of_beraha}.
Concerning the determinants in (3), it holds 
\[\det\big(A(n,n\!-\!1)\big)=N^{\left\lceil\frac{n}{2}\right\rceil}\neq 0\]
by Proposition \ref{prop:B(n,n-1)=A(n,n-1)_is_a_complex_number}.
Hence, we conclude for \(N\!\in\!\N_{\ge 4}\) and all \(n\!\in\!\N\) that \(\det\big(A(n,0)\big)\) is non-zero, i.e \(A(n,0)\) is invertible.
\end{proof}
Theorem \ref{thm:main_result:linear_independence_of_T_ps} is a direct consequence of this result:
\begin{cor*}[compare {Theorem \ref{thm:main_result:linear_independence_of_T_ps}}]
For any given \(N\!\in\!\N_{\ge 4}\) and \(n\!\in\!\N\), the collection of linear maps
\[\left(T_p\right)_{p\in\mathcal{NC}(0,n)},\]
as defined in Definition \ref{defn:T_p}, is linearly independent.
\end{cor*}
\begin{proof}
Given the linear independence of the vectors \(\left(T_p(1)\right)_{p\in\mathcal{NC}(0,n)}\), this result follows directly from the fact that the linear maps \(T_p\) above are maps from the complex numbers (into some Hilbert space).
\end{proof}
As displayed in Section \ref{sec:reformulating_the_injectivity_of_Psi_as_a_problem_of_linear_independence}, the linear independence of maps \(\big(T_p\big)_{p\in\mathcal(0,n)}\) (for \(N\ge 4\)) guarantees the following result:
\begin{cor}\label{cor:injectivity_of_Psi}
Let \(N\ge 4\) and consider the functor \(\Psi\) from Equation \ref{eqn:from_partitions_to_easy_quantum_groups}, given by
\[\mathcal{C}\overset{\Psi}{\longmapsto}R_N(\mathcal{C}).\]
Restricted to non-crossing categories of partitions, \(\Psi\) is injective.
\end{cor} 
\subsection{A comparison to Tutte's work}\label{subsec:a_comparison_to_tutte}
\label{subsec:comments_on_tuttes_work}
Sections \ref{subsec:definitions} and \ref{subsec:recursion_formula} are heavily based on W. Tutte's work ``The matrix of chromatic joints'', \cite{tutte}.
This section is used to stress again the changes and corrections we performed compared to the ideas in the original work.
\newline
The aforementioned two sections are an adaption of \cite{tutte} in the following sense: All objects defined in this work appear (in similar or the same form) in \cite{tutte} and the logical steps in order to establish a recursion formula for the determinant of the matrix \(A(n,0)\) are adopted from \cite{tutte}:
\begin{itemize}
\item [(1)] The sets of partitions \(W(n,r)\) and \(Y(n,r)\) as well as the partition manipulations \(f(i,q)\) and \(g(i,q)\).
\item [(2)]The graphs \(G(p,q)\) and \(H_r(p,q)\) and \(r\)-flaws.
\item [(3)]The matrices \(A(n,r)\) and \(B(n,r)\).
\item [(4)]The structures \([i]\),  \([i,i\!+\!1]\) and \([0]\).
\item [(5)]The expression \(F_r(p,q)\) that describes the column manipulations to be performed in the matrix \(A(n,r)\).
\item [(6)] The comparison of matrices \(B(n,r)\) and \(A(n\!-\!1,r')\).
\item [(7)] The statements about absence of \(r\)-flaws in the graphs \(G(p,q)\), \(G\big(p,f(i,q)\big)\) and \(G\big(p,g(i,q)\big)\), respectively.
\item [(8)] The statement about the number of components in the graphs \(G(p,q)\), \(G\big(p,f(i,q)\big)\) and \(G\big(p,g(i,q)\big)\).
\item [(9)] The statement about \((-1)^rF_r(p,q)\), describing the outcome of the column maniplations in Section \ref{subsec:recursion_formula}.
\item [(10)] The recursion formula from Section \ref{subsec:recursion_formula} itself.
\end{itemize}
In \cite{tutte}, item (7) and (8), two of the most important ingredients in order to establish in the end the recursion formula from Section \ref{subsec:recursion_formula}, were detected to have errors.
More precisely, the statement of item (7) is still the same, compare Proposition \ref{prop:when_is_e_r(p,q)_non-zero?} and \cite[p. 278]{tutte}, but in Tutte's work it was not compatible with the definition of the structures \([i]\), \([i,i\!+\!1]\) and \([0]\), see \cite[p. 277]{tutte}.
\begin{itemize}
\item[(a)]
In order to keep Tutte's ideas usable (and the statement behind item (7), Proposition \ref{prop:when_is_e_r(p,q)_non-zero?}, true), we changed the definition of the structures \([i]\),  \([i,i\!+\!1]\) and \([0]\), compare Definition \ref{defn:structures[i]_a.s.o.} and \cite[pp. 277]{tutte}, and the statement referring to item (8), compare Lemma \ref{lem:components_of_G(p,f(i,q))_and_G(p,g(i,q))} and \cite[Thm. 5.1]{tutte}.
\item[(b)] Although the result behind item (9) itself was not wrong in Tutte's work, compare Lemma \ref{lem:F_r(p,q)-formula} and \cite[Thm. 5.2]{tutte}, its proof needed to be adapted as it heavily relied on item (7) and (8).
\end{itemize}
In addition to the corrective adaptation above, we differ in the following aspects from \cite{tutte}:
\begin{itemize}
\item [(i)] The definitions of the graphs \(G(p,q)\) and \(H_r(p,q)\) in this article, Definitions \ref{defn:G(p,q)} and \ref{defn:H_r(p,q)}, are equivalent to those in \cite[p. 270, p. 277]{tutte}, but our definitions of vertices and edges are different.
\item[(ii)] As a consequence, also our definition of an \(r\)-flawless graph, Definition \ref{defn:r-flawless_graph}, uses other formulations.
\item [(iii)] The pictures of partitions as used in Equation \ref{eqn:structure_in_W(n,r=2s+1)} and thereafter do not appear in \cite{tutte} and neither do the illustrations of graphs as in Equation \ref{eqn:example_of_G(p,q)} or Equation \ref{eqn:G(p,q)_for_W(n,r=2s+1)}.
\item [(iv)] The schemes as introduced in Notation \ref{notation:displaying_connections_in_a_graph}, describing or defining connections in a graph, are added in order to improve the readability of definitions and proofs.
\end{itemize}

\bibliography{linear_independences}
\bibliographystyle{alpha}

\end{document}

%% file: PartitionCommands.tex
\usepackage{calc}
\newcounter{PartitionDepth}
\newcounter{PartitionLength}
\newsavebox{\boxpaarpart}
   \savebox{\boxpaarpart}
   { \begin{picture}(0.7,0.5)
     \put(0,0){\line(0,1){0.4}}
     \put(0.5,0){\line(0,1){0.4}}
     \put(0,0.4){\line(1,0){0.5}}
     \end{picture}}
\newsavebox{\boxbaarpart}
   \savebox{\boxbaarpart}
   { \begin{picture}(0.7,0.5)
     \put(0,0){\line(0,1){0.4}}
     \put(0.5,0){\line(0,1){0.4}}
     \put(0,0){\line(1,0){0.5}}
     \end{picture}}
\newsavebox{\boxdreipart}
   \savebox{\boxdreipart}
   { \begin{picture}(1,0.5)
     \put(0,0){\line(0,1){0.4}}
     \put(0.4,0){\line(0,1){0.4}}
     \put(0.8,0){\line(0,1){0.4}}
     \put(0,0.4){\line(1,0){0.8}}
     \end{picture}}
\newsavebox{\boxvierpart}
   \savebox{\boxvierpart}
   { \begin{picture}(1.4,0.5)
     \put(0,0){\line(0,1){0.4}}
     \put(0.4,0){\line(0,1){0.4}}
     \put(0.8,0){\line(0,1){0.4}}
     \put(1.2,0){\line(0,1){0.4}}
     \put(0,0.4){\line(1,0){1.2}}
     \end{picture}}
\newsavebox{\boxvierpartrot}
   \savebox{\boxvierpartrot}
   { \begin{picture}(0.5,0.8)
     \put(0,-0.2){\line(0,1){0.3}}
     \put(0.4,-0.2){\line(0,1){0.3}}
     \put(0,0.1){\line(1,0){0.4}}
     \put(0,0.4){\line(0,1){0.3}}
     \put(0.4,0.4){\line(0,1){0.3}}
     \put(0,0.4){\line(1,0){0.4}}
     \put(0.2,0.1){\line(0,1){0.3}}
     \end{picture}}
\newsavebox{\boxvierpartrotdrei}
   \savebox{\boxvierpartrotdrei}
   { \begin{picture}(1,0.8)
     \put(0,-0.2){\line(0,1){0.3}}
     \put(0.4,-0.2){\line(0,1){0.8}}
     \put(0.8,-0.2){\line(0,1){0.3}}
     \put(0,0.1){\line(1,0){0.8}}
     \end{picture}}
\newsavebox{\boxcrosspart}
   \savebox{\boxcrosspart}
   { \begin{picture}(0.5,0.8)
     \put(0,-0.2){\line(1,2){0.4}}
     \put(0.4,-0.2){\line(-1,2){0.4}}
     \end{picture}}
\newsavebox{\boxhalflibpart}
   \savebox{\boxhalflibpart}
   { \begin{picture}(1,0.8)
     \put(0,-0.2){\line(1,1){0.8}}
     \put(0.8,-0.2){\line(-1,1){0.8}}
     \put(0.4,-0.2){\line(0,1){0.8}}
     \end{picture}}
\newsavebox{\boxpositioner} 
   \savebox{\boxpositioner}
   { \begin{picture}(1.4,0.5)
     \put(0,0){\line(0,1){0.3}}
     \put(0.4,0){\line(0,1){0.6}}
     \put(0.8,0){\line(0,1){0.3}}
     \put(1.2,0){\line(0,1){0.6}}
     \put(0.4,0.6){\line(1,0){0.8}}
     \end{picture}}
\newsavebox{\boxfatcross} 
   \savebox{\boxfatcross}
   { \begin{picture}(1.4,1.3)
     \put(0,-0.2){\line(0,1){0.3}}
     \put(0.4,-0.2){\line(0,1){0.3}}
     \put(0.8,-0.2){\line(0,1){0.3}}
     \put(1.2,-0.2){\line(0,1){0.3}}
     \put(0,0.1){\line(1,0){0.4}}
     \put(0.8,0.1){\line(1,0){0.4}}
     \put(0,0.9){\line(0,1){0.3}}
     \put(0.4,0.9){\line(0,1){0.3}}
     \put(0.8,0.9){\line(0,1){0.3}}
     \put(1.2,0.9){\line(0,1){0.3}}
     \put(0,0.9){\line(1,0){0.4}}
     \put(0.8,0.9){\line(1,0){0.4}}
     \put(0.2,0.1){\line(1,1){0.8}}
     \put(0.2,0.9){\line(1,-1){0.8}}
     \end{picture}}
\newsavebox{\boxprimarypart} 
   \savebox{\boxprimarypart}
   { \begin{picture}(1,1.3)
     \put(0,-0.2){\line(0,1){0.3}}
     \put(0.4,-0.2){\line(0,1){0.3}}
     \put(0.8,-0.2){\line(0,1){0.3}}
     \put(0.4,0.1){\line(1,0){0.4}}
     \put(0,0.9){\line(0,1){0.3}}
     \put(0.4,0.9){\line(0,1){0.3}}
     \put(0.8,0.9){\line(0,1){0.3}}
     \put(0,0.9){\line(1,0){0.4}}
     \put(0,0.1){\line(1,1){0.8}}
     \put(0.2,0.9){\line(1,-2){0.4}}
     \end{picture}}

\newsavebox{\boxidpartww}
   \savebox{\boxidpartww}
   { \begin{picture}(0.5,1.2)
     \put(0.2,0.15){\line(0,1){0.7}}
     \put(0,-0.2){$\circ$}
     \put(0,0.8){$\circ$}
     \end{picture}}
\newsavebox{\boxidpartbw}
   \savebox{\boxidpartbw}
   { \begin{picture}(0.5,1.2)
     \put(0.2,0.15){\line(0,1){0.7}}
     \put(0,-0.2){$\circ$}
     \put(0,0.8){$\bullet$}
     \end{picture}}
\newsavebox{\boxidpartwb}
   \savebox{\boxidpartwb}
   { \begin{picture}(0.5,1.2)
     \put(0.2,0.15){\line(0,1){0.7}}
     \put(0,-0.2){$\bullet$}
     \put(0,0.8){$\circ$}
     \end{picture}}
\newsavebox{\boxidpartbb}
   \savebox{\boxidpartbb}
   { \begin{picture}(0.5,1.2)
     \put(0.2,0.15){\line(0,1){0.7}}
     \put(0,-0.2){$\bullet$}
     \put(0,0.8){$\bullet$}
     \end{picture}}

\newsavebox{\boxidpartsingletonbb}
   \savebox{\boxidpartsingletonbb}
   { \begin{picture}(0.5,1.2)
     \put(0.2,0.15){\line(0,1){0.2}}
     \put(0.2,0.65){\line(0,1){0.2}}
     \put(0,-0.2){$\bullet$}
     \put(0,0.8){$\bullet$}
     \end{picture}}
\newsavebox{\boxidpartsingletonww}
   \savebox{\boxidpartsingletonww}
   { \begin{picture}(0.5,1.2)
     \put(0.2,0.15){\line(0,1){0.2}}
     \put(0.2,0.65){\line(0,1){0.2}}
     \put(0,-0.2){$\circ$}
     \put(0,0.8){$\circ$}
     \end{picture}}
     
\newsavebox{\boxpaarpartbb}
   \savebox{\boxpaarpartbb}
   { \begin{picture}(1,1)
     \put(0.2,0.2){\line(0,1){0.4}}
     \put(0.7,0.2){\line(0,1){0.4}}
     \put(0.2,0.6){\line(1,0){0.5}}
     \put(0.5,-0.2){$\bullet$}
     \put(0,-0.2){$\bullet$}
     \end{picture}}
\newsavebox{\boxpaarpartww}
   \savebox{\boxpaarpartww}
   { \begin{picture}(1,1)
     \put(0.2,0.2){\line(0,1){0.4}}
     \put(0.7,0.2){\line(0,1){0.4}}
     \put(0.2,0.6){\line(1,0){0.5}}
     \put(0.5,-0.2){$\circ$}
     \put(0,-0.2){$\circ$}
     \end{picture}}
\newsavebox{\boxpaarpartbw}
   \savebox{\boxpaarpartbw}
   { \begin{picture}(1,1)
     \put(0.2,0.2){\line(0,1){0.4}}
     \put(0.7,0.2){\line(0,1){0.4}}
     \put(0.2,0.6){\line(1,0){0.5}}
     \put(0.5,-0.2){$\circ$}
     \put(0,-0.2){$\bullet$}
     \end{picture}}
\newsavebox{\boxpaarpartwb}
   \savebox{\boxpaarpartwb}
   { \begin{picture}(1,1)
     \put(0.2,0.2){\line(0,1){0.4}}
     \put(0.7,0.2){\line(0,1){0.4}}
     \put(0.2,0.6){\line(1,0){0.5}}
     \put(0.5,-0.2){$\bullet$}
     \put(0,-0.2){$\circ$}
     \end{picture}}
\newsavebox{\boxbaarpartbb}
   \savebox{\boxbaarpartbb}
   { \begin{picture}(1,1)
     \put(0.2,-0.1){\line(0,1){0.4}}
     \put(0.7,-0.1){\line(0,1){0.4}}
     \put(0.2,-0.1){\line(1,0){0.5}}
     \put(0.5,0.3){$\bullet$}
     \put(0,0.3){$\bullet$}
     \end{picture}}
\newsavebox{\boxbaarpartww}
   \savebox{\boxbaarpartww}
   { \begin{picture}(1,1)
     \put(0.2,-0.1){\line(0,1){0.4}}
     \put(0.7,-0.1){\line(0,1){0.4}}
     \put(0.2,-0.1){\line(1,0){0.5}}
     \put(0.5,0.3){$\circ$}
     \put(0,0.3){$\circ$}
     \end{picture}}
\newsavebox{\boxbaarpartbw}
   \savebox{\boxbaarpartbw}
   { \begin{picture}(1,1)
     \put(0.2,-0.1){\line(0,1){0.4}}
     \put(0.7,-0.1){\line(0,1){0.4}}
     \put(0.2,-0.1){\line(1,0){0.5}}
     \put(0.5,0.3){$\circ$}
     \put(0,0.3){$\bullet$}
     \end{picture}}
\newsavebox{\boxbaarpartwb}
   \savebox{\boxbaarpartwb}
   { \begin{picture}(1,1)
     \put(0.2,-0.1){\line(0,1){0.4}}
     \put(0.7,-0.1){\line(0,1){0.4}}
     \put(0.2,-0.1){\line(1,0){0.5}}
     \put(0.5,0.3){$\bullet$}
     \put(0,0.3){$\circ$}
     \end{picture}}
\newsavebox{\boxpaarbaarpartwwww}
   \savebox{\boxpaarbaarpartwwww}
   { \begin{picture}(1,1.2)
     \put(0.2,0.15){\line(0,1){0.2}}
     \put(0.2,0.65){\line(0,1){0.2}}
     \put(0.2,0.65){\line(1,0){0.5}}
     \put(0.2,0.35){\line(1,0){0.5}}
     \put(0.7,0.15){\line(0,1){0.2}}
     \put(0.7,0.65){\line(0,1){0.2}}
     \put(0,0.8){$\circ$}
     \put(0.5,0.8){$\circ$}
     \put(0,-0.2){$\circ$}
     \put(0.5,-0.2){$\circ$}
     \end{picture}}     

\newsavebox{\boxsingletonw}
   \savebox{\boxsingletonw}
   { \begin{picture}(0.5, 1)
     \put(0,0.3){$\uparrow$}
     \put(0,-0.2){$\circ$}
     \end{picture}}
\newsavebox{\boxsingletonb}
   \savebox{\boxsingletonb}
   { \begin{picture}(0.5, 1)
     \put(0,0.3){$\uparrow$}
     \put(0,-0.2){$\bullet$}
     \end{picture}}
\newsavebox{\boxdownsingletonw}
   \savebox{\boxdownsingletonw}
   { \begin{picture}(0.5, 1)
     \put(0,0){$\downarrow$}
     \put(0,0.6){$\circ$}
     \end{picture}}
\newsavebox{\boxdownsingletonb}
   \savebox{\boxdownsingletonb}
   { \begin{picture}(0.5, 1)
     \put(0,0){$\downarrow$}
     \put(0,0.6){$\bullet$}
     \end{picture}}

\newsavebox{\boxvierpartwbwb}
   \savebox{\boxvierpartwbwb}
   { \begin{picture}(1.7,1)
     \put(0.2,0.2){\line(0,1){0.4}}
     \put(0.6,0.2){\line(0,1){0.4}}
     \put(1,0.2){\line(0,1){0.4}}
     \put(1.4,0.2){\line(0,1){0.4}}
     \put(0.2,0.6){\line(1,0){1.2}}
     \put(0,-0.2){$\circ$}
     \put(0.4,-0.2){$\bullet$}
     \put(0.8,-0.2){$\circ$}
     \put(1.2,-0.2){$\bullet$}
     \end{picture}}
\newsavebox{\boxvierpartbwbw}
   \savebox{\boxvierpartbwbw}
   { \begin{picture}(1.7,1)
     \put(0.2,0.2){\line(0,1){0.4}}
     \put(0.6,0.2){\line(0,1){0.4}}
     \put(1,0.2){\line(0,1){0.4}}
     \put(1.4,0.2){\line(0,1){0.4}}
     \put(0.2,0.6){\line(1,0){1.2}}
     \put(0,-0.2){$\bullet$}
     \put(0.4,-0.2){$\circ$}
     \put(0.8,-0.2){$\bullet$}
     \put(1.2,-0.2){$\circ$}
     \end{picture}}
\newsavebox{\boxvierpartwwbb}
   \savebox{\boxvierpartwwbb}
   { \begin{picture}(1.7,1)
     \put(0.2,0.2){\line(0,1){0.4}}
     \put(0.6,0.2){\line(0,1){0.4}}
     \put(1,0.2){\line(0,1){0.4}}
     \put(1.4,0.2){\line(0,1){0.4}}
     \put(0.2,0.6){\line(1,0){1.2}}
     \put(0,-0.2){$\circ$}
     \put(0.4,-0.2){$\circ$}
     \put(0.8,-0.2){$\bullet$}
     \put(1.2,-0.2){$\bullet$}
     \end{picture}}
\newsavebox{\boxvierpartwbbw}
   \savebox{\boxvierpartwbbw}
   { \begin{picture}(1.7,1)
     \put(0.2,0.2){\line(0,1){0.4}}
     \put(0.6,0.2){\line(0,1){0.4}}
     \put(1,0.2){\line(0,1){0.4}}
     \put(1.4,0.2){\line(0,1){0.4}}
     \put(0.2,0.6){\line(1,0){1.2}}
     \put(0,-0.2){$\circ$}
     \put(0.4,-0.2){$\bullet$}
     \put(0.8,-0.2){$\bullet$}
     \put(1.2,-0.2){$\circ$}
     \end{picture}}
\newsavebox{\boxvierpartbwwb}
   \savebox{\boxvierpartbwwb}
   { \begin{picture}(1.7,1)
     \put(0.2,0.2){\line(0,1){0.4}}
     \put(0.6,0.2){\line(0,1){0.4}}
     \put(1,0.2){\line(0,1){0.4}}
     \put(1.4,0.2){\line(0,1){0.4}}
     \put(0.2,0.6){\line(1,0){1.2}}
     \put(0,-0.2){$\bullet$}
     \put(0.4,-0.2){$\circ$}
     \put(0.8,-0.2){$\circ$}
     \put(1.2,-0.2){$\bullet$}
     \end{picture}}
\newsavebox{\boxvierpartrotwbwb}
   \savebox{\boxvierpartrotwbwb}
   { \begin{picture}(1,1.2)
     \put(0.2,0.15){\line(0,1){0.2}}
     \put(0.2,0.65){\line(0,1){0.2}}
     \put(0.2,0.65){\line(1,0){0.5}}
     \put(0.2,0.35){\line(1,0){0.5}}
     \put(0.45,0.35){\line(0,1){0.3}}
     \put(0.7,0.15){\line(0,1){0.2}}
     \put(0.7,0.65){\line(0,1){0.2}}
     \put(0,0.8){$\circ$}
     \put(0.5,0.8){$\bullet$}
     \put(0,-0.2){$\circ$}
     \put(0.5,-0.2){$\bullet$}
     \end{picture}}
\newsavebox{\boxvierpartrotbwwb}
   \savebox{\boxvierpartrotbwwb}
   { \begin{picture}(1,1.2)
     \put(0.2,0.15){\line(0,1){0.2}}
     \put(0.2,0.65){\line(0,1){0.2}}
     \put(0.2,0.65){\line(1,0){0.5}}
     \put(0.2,0.35){\line(1,0){0.5}}
     \put(0.45,0.35){\line(0,1){0.3}}
     \put(0.7,0.15){\line(0,1){0.2}}
     \put(0.7,0.65){\line(0,1){0.2}}
     \put(0,0.8){$\bullet$}
     \put(0.5,0.8){$\circ$}
     \put(0,-0.2){$\circ$}
     \put(0.5,-0.2){$\bullet$}
     \end{picture}}
\newsavebox{\boxvierpartrotbwbw}
   \savebox{\boxvierpartrotbwbw}
   { \begin{picture}(1,1.2)
     \put(0.2,0.15){\line(0,1){0.2}}
     \put(0.2,0.65){\line(0,1){0.2}}
     \put(0.2,0.65){\line(1,0){0.5}}
     \put(0.2,0.35){\line(1,0){0.5}}
     \put(0.45,0.35){\line(0,1){0.3}}
     \put(0.7,0.15){\line(0,1){0.2}}
     \put(0.7,0.65){\line(0,1){0.2}}
     \put(0,0.8){$\bullet$}
     \put(0.5,0.8){$\circ$}
     \put(0,-0.2){$\bullet$}
     \put(0.5,-0.2){$\circ$}
     \end{picture}}
\newsavebox{\boxvierpartrotwwww}
   \savebox{\boxvierpartrotwwww}
   { \begin{picture}(1,1.2)
     \put(0.2,0.15){\line(0,1){0.2}}
     \put(0.2,0.65){\line(0,1){0.2}}
     \put(0.2,0.65){\line(1,0){0.5}}
     \put(0.2,0.35){\line(1,0){0.5}}
     \put(0.45,0.35){\line(0,1){0.3}}
     \put(0.7,0.15){\line(0,1){0.2}}
     \put(0.7,0.65){\line(0,1){0.2}}
     \put(0,0.8){$\circ$}
     \put(0.5,0.8){$\circ$}
     \put(0,-0.2){$\circ$}
     \put(0.5,-0.2){$\circ$}
     \end{picture}}
\newsavebox{\boxvierpartrotbbbb}
   \savebox{\boxvierpartrotbbbb}
   { \begin{picture}(1,1.2)
     \put(0.2,0.15){\line(0,1){0.2}}
     \put(0.2,0.65){\line(0,1){0.2}}
     \put(0.2,0.65){\line(1,0){0.5}}
     \put(0.2,0.35){\line(1,0){0.5}}
     \put(0.45,0.35){\line(0,1){0.3}}
     \put(0.7,0.15){\line(0,1){0.2}}
     \put(0.7,0.65){\line(0,1){0.2}}
     \put(0,0.8){$\bullet$}
     \put(0.5,0.8){$\bullet$}
     \put(0,-0.2){$\bullet$}
     \put(0.5,-0.2){$\bullet$}
     \end{picture}}

\newsavebox{\boxdreipartwww}
   \savebox{\boxdreipartwww}
   { \begin{picture}(1.2,1)
     \put(0.2,0.2){\line(0,1){0.4}}
     \put(0.6,0.2){\line(0,1){0.4}}
     \put(1,0.2){\line(0,1){0.4}}
     \put(0.2,0.6){\line(1,0){0.8}}
     \put(0,-0.2){$\circ$}
     \put(0.4,-0.2){$\circ$}
     \put(0.8,-0.2){$\circ$}
     \end{picture}}
\newsavebox{\boxdreipartrotwww}
   \savebox{\boxdreipartrotwww}
   { \begin{picture}(1,1.2)
     \put(0.2,0.65){\line(0,1){0.2}}
     \put(0.2,0.65){\line(1,0){0.5}}
     \put(0.45,0.15){\line(0,1){0.5}}
     \put(0.7,0.65){\line(0,1){0.2}}
     \put(0,0.8){$\circ$}
     \put(0.5,0.8){$\circ$}
     \put(0.25,-0.2){$\circ$}
     \end{picture}}     
\newsavebox{\boxdowndreipartrotwww}
   \savebox{\boxdowndreipartrotwww}
   { \begin{picture}(1,1.2)
     \put(0.2,0.15){\line(0,1){0.2}}
     \put(0.2,0.35){\line(1,0){0.5}}
     \put(0.45,0.35){\line(0,1){0.5}}
     \put(0.7,0.15){\line(0,1){0.2}}
     \put(0,-0.2){$\circ$}
     \put(0.5,-0.2){$\circ$}
     \put(0.25,0.8){$\circ$}
     \end{picture}}          

\newsavebox{\boxsechspartwbwbwb}
   \savebox{\boxsechspartwbwbwb}
   { \begin{picture}(2.5,1)
     \put(0.2,0.2){\line(0,1){0.4}}
     \put(0.6,0.2){\line(0,1){0.4}}
     \put(1,0.2){\line(0,1){0.4}}
     \put(1.4,0.2){\line(0,1){0.4}}
     \put(1.8,0.2){\line(0,1){0.4}}
     \put(2.2,0.2){\line(0,1){0.4}}
     \put(0.2,0.6){\line(1,0){2}}
     \put(0,-0.2){$\circ$}
     \put(0.4,-0.2){$\bullet$}
     \put(0.8,-0.2){$\circ$}
     \put(1.2,-0.2){$\bullet$}
     \put(1.6,-0.2){$\circ$}
     \put(2.0,-0.2){$\bullet$}
     \end{picture}}
 
\newsavebox{\boxcrosspartwbbw}
   \savebox{\boxcrosspartwbbw}
   { \begin{picture}(1,1.4)
     \put(0.25,0.15){\line(1,2){0.4}}
     \put(0.65,0.15){\line(-1,2){0.4}}
     \put(0,0.9){$\circ$}
     \put(0.5,0.9){$\bullet$}
     \put(0,-0.2){$\bullet$}
     \put(0.5,-0.2){$\circ$}
     \end{picture}}
\newsavebox{\boxcrosspartbwwb}
   \savebox{\boxcrosspartbwwb}
   { \begin{picture}(1,1.4)
     \put(0.25,0.15){\line(1,2){0.4}}
     \put(0.65,0.15){\line(-1,2){0.4}}
     \put(0,0.9){$\bullet$}
     \put(0.5,0.9){$\circ$}
     \put(0,-0.2){$\circ$}
     \put(0.5,-0.2){$\bullet$}
     \end{picture}}
\newsavebox{\boxcrosspartwwww}
   \savebox{\boxcrosspartwwww}
   { \begin{picture}(1,1.4)
     \put(0.25,0.15){\line(1,2){0.4}}
     \put(0.65,0.15){\line(-1,2){0.4}}
     \put(0,0.9){$\circ$}
     \put(0.5,0.9){$\circ$}
     \put(0,-0.2){$\circ$}
     \put(0.5,-0.2){$\circ$}
     \end{picture}}
\newsavebox{\boxcrosspartbbbb}
   \savebox{\boxcrosspartbbbb}
   { \begin{picture}(1,1.4)
     \put(0.25,0.15){\line(1,2){0.4}}
     \put(0.65,0.15){\line(-1,2){0.4}}
     \put(0,0.9){$\bullet$}
     \put(0.5,0.9){$\bullet$}
     \put(0,-0.2){$\bullet$}
     \put(0.5,-0.2){$\bullet$}
     \end{picture}}

\newsavebox{\boxhalflibpartwwwwww}
   \savebox{\boxhalflibpartwwwwww}
   { \begin{picture}(1.5,1.4)
     \put(0.3,0.15){\line(1,1){0.8}}
     \put(1.1,0.15){\line(-1,1){0.8}}
     \put(0.7,0.15){\line(0,1){0.8}}     
     \put(0,0.9){$\circ$}
     \put(0.5,0.9){$\circ$}
     \put(1,0.9){$\circ$}
     \put(0,-0.2){$\circ$}
     \put(0.5,-0.2){$\circ$}
     \put(1,-0.2){$\circ$}     
     \end{picture}}

\newsavebox{\boxpositionerd} 
   \savebox{\boxpositionerd}
   { \begin{picture}(2.6,1.5)
     \put(0.9,0.2){\line(0,1){0.9}}
     \put(2.3,0.2){\line(0,1){0.9}}
     \put(0.9,1.1){\line(1,0){1.4}}
     \put(0.05,-0.05){$^\uparrow$}
     \put(0.2,0.4){\tiny$^{\otimes d}$}
     \put(1.35,-0.05){$^\uparrow$}
     \put(1.5,0.4){\tiny$^{\otimes d}$}
     \put(0,-0.2){$\circ$}
     \put(0.7,-0.2){$\circ$}
     \put(1.3,-0.2){$\bullet$}
     \put(2.1,-0.2){$\bullet$}
     \end{picture}}
\newsavebox{\boxpositionerl} 
   \savebox{\boxpositionerl}
   { \begin{picture}(2.6,1.5)
     \put(0.9,0.2){\line(0,1){0.9}}
     \put(2.3,0.2){\line(0,1){0.9}}
     \put(0.9,1.1){\line(1,0){1.4}}
     \put(0.05,-0.05){$^\uparrow$}
     \put(0.2,0.4){\tiny$^{\otimes l}$}
     \put(1.35,-0.05){$^\uparrow$}
     \put(1.5,0.4){\tiny$^{\otimes l}$}
     \put(0,-0.2){$\circ$}
     \put(0.7,-0.2){$\circ$}
     \put(1.3,-0.2){$\bullet$}
     \put(2.1,-0.2){$\bullet$}
     \end{picture}}
\newsavebox{\boxpositionerrpluseins} 
   \savebox{\boxpositionerrpluseins}
   { \begin{picture}(3.9,1.5)
     \put(1.6,0.2){\line(0,1){0.9}}
     \put(3.6,0.2){\line(0,1){0.9}}
     \put(1.6,1.1){\line(1,0){2}}
     \put(0.05,-0.05){$^\uparrow$}
     \put(0.2,0.4){\tiny$^{\otimes r+1}$}
     \put(2.05,-0.05){$^\uparrow$}
     \put(2.2,0.4){\tiny$^{\otimes r-1}$}
     \put(0,-0.2){$\circ$}
     \put(1.4,-0.2){$\bullet$}
     \put(2,-0.2){$\bullet$}
     \put(3.4,-0.2){$\bullet$}
     \end{picture}}
\newsavebox{\boxpositionerdt} 
   \savebox{\boxpositionerdt}
   { \begin{picture}(2.9,1.5)
     \put(1.2,0.2){\line(0,1){0.9}}
     \put(2.9,0.2){\line(0,1){0.9}}
     \put(1.2,1.1){\line(1,0){1.7}}
     \put(0.05,-0.05){$^\uparrow$}
     \put(0.2,0.4){\tiny$^{\otimes dt}$}
     \put(1.65,-0.05){$^\uparrow$}
     \put(1.8,0.4){\tiny$^{\otimes dt}$}
     \put(0,-0.2){$\circ$}
     \put(1,-0.2){$\circ$}
     \put(1.6,-0.2){$\bullet$}
     \put(2.7,-0.2){$\bullet$}
     \end{picture}}
\newsavebox{\boxpositioners} 
   \savebox{\boxpositioners}
   { \begin{picture}(2.6,1.5)
     \put(0.9,0.2){\line(0,1){0.9}}
     \put(2.3,0.2){\line(0,1){0.9}}
     \put(0.9,1.1){\line(1,0){1.4}}
     \put(0.05,-0.05){$^\uparrow$}
     \put(0.2,0.4){\tiny$^{\otimes s}$}
     \put(1.35,-0.05){$^\uparrow$}
     \put(1.5,0.4){\tiny$^{\otimes s}$}
     \put(0,-0.2){$\circ$}
     \put(0.7,-0.2){$\circ$}
     \put(1.3,-0.2){$\bullet$}
     \put(2.1,-0.2){$\bullet$}
     \end{picture}}
\newsavebox{\boxpositionerdinv} 
   \savebox{\boxpositionerdinv}
   { \begin{picture}(2.6,1.5)
     \put(0.9,0.2){\line(0,1){0.9}}
     \put(2.3,0.2){\line(0,1){0.9}}
     \put(0.9,1.1){\line(1,0){1.4}}
     \put(0.05,-0.05){$^\uparrow$}
     \put(0.2,0.4){\tiny$^{\otimes d}$}
     \put(1.35,-0.05){$^\uparrow$}
     \put(1.5,0.4){\tiny$^{\otimes d}$}
     \put(0,-0.2){$\circ$}
     \put(0.7,-0.2){$\bullet$}
     \put(1.3,-0.2){$\bullet$}
     \put(2.1,-0.2){$\circ$}
     \end{picture}}
\newsavebox{\boxpositionersinv} 
   \savebox{\boxpositionersinv}
   { \begin{picture}(2.6,1.5)
     \put(0.9,0.2){\line(0,1){0.9}}
     \put(2.3,0.2){\line(0,1){0.9}}
     \put(0.9,1.1){\line(1,0){1.4}}
     \put(0.05,-0.05){$^\uparrow$}
     \put(0.2,0.4){\tiny$^{\otimes s}$}
     \put(1.35,-0.05){$^\uparrow$}
     \put(1.5,0.4){\tiny$^{\otimes s}$}
     \put(0,-0.2){$\circ$}
     \put(0.7,-0.2){$\bullet$}
     \put(1.3,-0.2){$\bullet$}
     \put(2.1,-0.2){$\circ$}
     \end{picture}}
\newsavebox{\boxpositionerdpluszwei}
   \savebox{\boxpositionerdpluszwei}
   { \begin{picture}(3.4,1.5)
     \put(1.6,0.2){\line(0,1){0.9}}
     \put(3.0,0.2){\line(0,1){0.9}}
     \put(1.6,1.1){\line(1,0){1.4}}
     \put(0.05,-0.05){$^\uparrow$}
     \put(0.2,0.4){\tiny$^{\otimes d+2}$}
     \put(2.05,-0.05){$^\uparrow$}
     \put(2.2,0.4){\tiny$^{\otimes d}$}
     \put(0,-0.2){$\circ$}
     \put(1.4,-0.2){$\bullet$}
     \put(2,-0.2){$\bullet$}
     \put(2.8,-0.2){$\bullet$}
     \end{picture}}
\newsavebox{\boxpositionersminuszwei}
   \savebox{\boxpositionersminuszwei}
   { \begin{picture}(3.4,1.5)
     \put(1,0.2){\line(0,1){0.9}}
     \put(3.0,0.2){\line(0,1){0.9}}
     \put(1,1.1){\line(1,0){2}}
     \put(0.05,-0.05){$^\uparrow$}
     \put(0.2,0.4){\tiny$^{\otimes s}$}
     \put(1.45,-0.05){$^\uparrow$}
     \put(1.6,0.4){\tiny$^{\otimes s-2}$}
     \put(0,-0.2){$\circ$}
     \put(0.8,-0.2){$\bullet$}
     \put(1.4,-0.2){$\bullet$}
     \put(2.8,-0.2){$\bullet$}
     \end{picture}}
\newsavebox{\boxpositionerrnull}
   \savebox{\boxpositionerrnull}
   { \begin{picture}(3.8,1.5)
     \put(1.2,0.2){\line(0,1){0.9}}
     \put(3.4,0.2){\line(0,1){0.9}}
     \put(1.2,1.1){\line(1,0){2.2}}
     \put(0.05,-0.05){$^\uparrow$}
     \put(0.2,0.4){\tiny$^{\otimes r_0}$}
     \put(1.65,-0.05){$^\uparrow$}
     \put(1.8,0.4){\tiny$^{\otimes r_0-2}$}
     \put(0,-0.2){$\circ$}
     \put(1,-0.2){$\bullet$}
     \put(1.6,-0.2){$\bullet$}
     \put(3.2,-0.2){$\bullet$}
     \end{picture}}
\newsavebox{\boxpositionerwbwb}
   \savebox{\boxpositionerwbwb}
   { \begin{picture}(1.8,1)
     \put(0.6,0.2){\line(0,1){0.6}}
     \put(1.4,0.2){\line(0,1){0.6}}
     \put(0.6,0.8){\line(1,0){0.8}}
     \put(0.05,-0.05){$^\uparrow$}
     \put(0.85,-0.05){$^\uparrow$}
     \put(0,-0.2){$\circ$}
     \put(0.4,-0.2){$\bullet$}
     \put(0.8,-0.2){$\circ$}
     \put(1.2,-0.2){$\bullet$}
     \end{picture}}
\newsavebox{\boxpositionerwwbb}
   \savebox{\boxpositionerwwbb}
   { \begin{picture}(1.8,1)
     \put(0.6,0.2){\line(0,1){0.6}}
     \put(1.4,0.2){\line(0,1){0.6}}
     \put(0.6,0.8){\line(1,0){0.8}}
     \put(0.05,-0.05){$^\uparrow$}
     \put(0.85,-0.05){$^\uparrow$}
     \put(0,-0.2){$\circ$}
     \put(0.4,-0.2){$\circ$}
     \put(0.8,-0.2){$\bullet$}
     \put(1.2,-0.2){$\bullet$}
     \end{picture}}
\newsavebox{\boxpositionerwwww}
   \savebox{\boxpositionerwwww}
   { \begin{picture}(1.8,1)
     \put(0.6,0.2){\line(0,1){0.6}}
     \put(1.4,0.2){\line(0,1){0.6}}
     \put(0.6,0.8){\line(1,0){0.8}}
     \put(0.05,-0.05){$^\uparrow$}
     \put(0.85,-0.05){$^\uparrow$}
     \put(0,-0.2){$\circ$}
     \put(0.4,-0.2){$\circ$}
     \put(0.8,-0.2){$\circ$}
     \put(1.2,-0.2){$\circ$}
     \end{picture}}
\newsavebox{\boxpositionerrevwbwb}
   \savebox{\boxpositionerrevwbwb}
   { \begin{picture}(1.8,1)
     \put(0.2,0.2){\line(0,1){0.6}}
     \put(1.0,0.2){\line(0,1){0.6}}
     \put(0.2,0.8){\line(1,0){0.8}}
     \put(0.45,-0.05){$^\uparrow$}
     \put(1.25,-0.05){$^\uparrow$}
     \put(0,-0.2){$\circ$}
     \put(0.4,-0.2){$\bullet$}
     \put(0.8,-0.2){$\circ$}
     \put(1.2,-0.2){$\bullet$}
     \end{picture}}
\newsavebox{\boxpositionersalphaw} 
   \savebox{\boxpositionersalphaw}
   { \begin{picture}(2.6,1.5)
     \put(0.9,0.2){\line(0,1){0.9}}
     \put(2.3,0.2){\line(0,1){0.9}}
     \put(0.9,1.1){\line(1,0){1.4}}
     \put(0.05,-0.05){$^\uparrow$}
     \put(1.35,-0.05){$^\uparrow$}
     \put(0.2,0.4){\tiny$^{\otimes s}$}
     \put(1.5,0.4){\tiny$^{\otimes \alpha}$}
     \put(0,-0.2){$\circ$}
     \put(0.7,-0.2){$\circ$}
     \put(1.3,-0.2){$\circ$}
     \put(2.1,-0.2){$\bullet$}
     \end{picture}}
\newsavebox{\boxpositionersalphab} 
   \savebox{\boxpositionersalphab}
   { \begin{picture}(2.6,1.5)
     \put(0.9,0.2){\line(0,1){0.9}}
     \put(2.3,0.2){\line(0,1){0.9}}
     \put(0.9,1.1){\line(1,0){1.4}}
     \put(0.05,-0.05){$^\uparrow$}
     \put(1.35,-0.05){$^\uparrow$}
     \put(0.2,0.4){\tiny$^{\otimes s}$}
     \put(1.5,0.4){\tiny$^{\otimes \alpha}$}
     \put(0,-0.2){$\circ$}
     \put(0.7,-0.2){$\circ$}
     \put(1.3,-0.2){$\bullet$}
     \put(2.1,-0.2){$\bullet$}
     \end{picture}}

\newsavebox{\boxfatcrosswwww}
   \savebox{\boxfatcrosswwww}
   { \begin{picture}(2,1.6)
     \put(0.2,0.15){\line(0,1){0.2}}
     \put(0.2,0.95){\line(0,1){0.2}}
     \put(0.2,0.95){\line(1,0){0.5}}
     \put(0.2,0.35){\line(1,0){0.5}}
     \put(0.7,0.15){\line(0,1){0.2}}
     \put(0.7,0.95){\line(0,1){0.2}}
     \put(0,1.1){$\circ$}
     \put(0.5,1.1){$\circ$}
     \put(0,-0.2){$\circ$}
     \put(0.5,-0.2){$\circ$}
      \put(0.5,0.35){\line(2,1){1}}
      \put(0.5,0.85){\line(2,-1){1}}
     \put(1.2,0.15){\line(0,1){0.2}}
     \put(1.2,0.95){\line(0,1){0.2}}
     \put(1.2,0.95){\line(1,0){0.5}}
     \put(1.2,0.35){\line(1,0){0.5}}
     \put(1.7,0.15){\line(0,1){0.2}}
     \put(1.7,0.95){\line(0,1){0.2}}
     \put(1,1.1){$\circ$}
     \put(1.5,1.1){$\circ$}
     \put(1,-0.2){$\circ$}
     \put(1.5,-0.2){$\circ$}
     \end{picture}}
\newsavebox{\boxpairpositionerwwwwww}
   \savebox{\boxpairpositionerwwwwww}
   { \begin{picture}(2,1.6)
     \put(0.2,0.95){\line(0,1){0.2}}
     \put(0.2,0.95){\line(1,0){0.5}}
     \put(0.7,0.95){\line(0,1){0.2}}
     \put(0,1.1){$\circ$}
     \put(0.5,1.1){$\circ$}
     \put(0.25,-0.2){$\circ$}
       \put(0.45,0.1){\line(1,1){1}}
      \put(0.5,0.85){\line(2,-1){1}}
     \put(1.2,0.15){\line(0,1){0.2}}
     \put(1.2,0.35){\line(1,0){0.5}}
     \put(1.7,0.15){\line(0,1){0.2}}
     \put(1.25,1.1){$\circ$}
     \put(1,-0.2){$\circ$}
     \put(1.5,-0.2){$\circ$}
     \end{picture}}     

\newsavebox{\boxAspace}
   \savebox{\boxAspace}
   { \begin{picture}(0.1,1.5)
     \end{picture}}
\newsavebox{\boxBspace}
   \savebox{\boxBspace}
   { \begin{picture}(0,1.85)
     \end{picture}}

\newcommand{\idpartww}{\usebox{\boxidpartww}}

\newcommand{\paarpartww}{\usebox{\boxpaarpartww}}

\newcommand{\baarpartww}{\usebox{\boxbaarpartww}}
